\documentclass[11pt,reqno]{amsproc}
\allowdisplaybreaks
\usepackage[table,xcdraw,dvipsnames]{xcolor}
\usepackage{color,soul}
\usepackage{float}
\usepackage{psfrag}
\usepackage{fullpage}
\usepackage{rotating}
\usepackage{stmaryrd}
\usepackage{subfigure}
\usepackage{listings}
\usepackage{enumerate}
\usepackage[small]{caption}
\usepackage{morefloats}
\usepackage{stmaryrd}
\usepackage{mathtools}
\usepackage{cancel}
\usepackage{algorithm}
\usepackage{algorithmicx}
\usepackage{algpseudocode}

\numberwithin{equation}{section}
\usepackage{graphicx}
\DeclareGraphicsExtensions{.eps}
\usepackage{epstopdf}
\usepackage{array,multirow,graphicx}
\usepackage{rotating}
\usepackage{booktabs,tabularx}
\usepackage[font=small,labelfont=bf,tableposition=top]{caption}
\usepackage[utf8x]{inputenc}
\usepackage[english]{babel}
\usepackage{multicol}
\usepackage[flushleft]{threeparttable}
\usepackage[customcolors,shade]{hf-tikz}
\usepackage{tikz}
\usetikzlibrary{calc}
\newcommand*\circled[1]{\tikz[baseline=(char.base)]{
    \node[shape=circle, draw, inner sep=1pt, 
        minimum height={\f@size*1.6},] (char) {\vphantom{WAH1g}#1};}}
\makeatother
\usetikzlibrary{decorations.pathmorphing}
\usepackage[semicolon,square,authoryear]{natbib}
\usepackage[debug=false, colorlinks=true, pdfstartview=FitV, 
linkcolor=blue, citecolor=blue, urlcolor=blue]{hyperref}

\newtheorem{lemma}{Lemma}[section]

\newtheorem{proposition}{Proposition}[section]

\usepackage{makecell, boldline}
\usepackage{url}
\usepackage{soul}
\usepackage{siunitx}
\usepackage{pifont}
\newcommand{\cmark}{\ding{51}}%
\newcommand{\xmark}{\ding{55}}%

\usepackage{import}
\usepackage{xifthen}
\usepackage{pdfpages}
\usepackage{transparent}

\usepackage{morefloats}
\newlength{\drop}
\definecolor{amethyst}{rgb}{0.6, 0.4, 0.8}
\definecolor{burgundy}{rgb}{0.5, 0.0, 0.13}
\newcommand\bfx{{\bf x}}

\title{\textbf{Maximum-principle-satisfying discontinuous Galerkin methods for incompressible two-phase immiscible flow}}
\author{
    \textbf{M.~S.~Joshaghani},
    \textbf{B.~Riviere}
   and
    \textbf{M.~Sekachev}
    \\
	{\small 
\textbf{Correspondence to:}~\textsf{m.sarraf.j@rice.edu}}}

\newsavebox{\measurebox}
\begin{document}

\date{\today}

\begin{titlepage}
    \drop=0.1\textheight
    \centering
    \vspace*{\baselineskip}
    \rule{\textwidth}{1.6pt}\vspace*{-\baselineskip}\vspace*{2pt}
    \rule{\textwidth}{0.4pt}\\[0.25\baselineskip]
    {\Large \textbf{\color{burgundy}
    Maximum-principle-satisfying discontinuous Galerkin methods for incompressible two-phase immiscible flow }}

    \rule{\textwidth}{0.4pt}\vspace*{-\baselineskip}\vspace{2pt}
    \rule{\textwidth}{1.6pt}\\[0.25\baselineskip]
    \scshape
    An e-print of the paper will be made available on arXiv. \par 
    \vspace*{0.3\baselineskip}
    Authored by \\[0.3\baselineskip]

    {\Large M.~S.~Joshaghani\par}
    {\itshape Postdoctoral Research Associate, Rice University, Houston, Texas 77005 \\ 
    \textbf{phone:} +1-281-781-5331, \textbf{e-mail:} m.sarraf.j@rice.edu}\\[0.25\baselineskip]

    {\Large B.~Riviere\par}
    {\itshape Noah Harding Chair and Professor of Computational and Applied Mathematics  \\
    Rice University, Houston, Texas 77005} \\ 
    {\Large M.~Sekachev\par}
    {\itshape TotalEnergies, Houston Texas 77002}
    \vspace{1cm}
    \centering
		\begin{figure}[h]
			\centering
        \subfigure[DG with no limiters]{
        \includegraphics[clip,scale=0.15,trim=0 0cm 0cm 0]{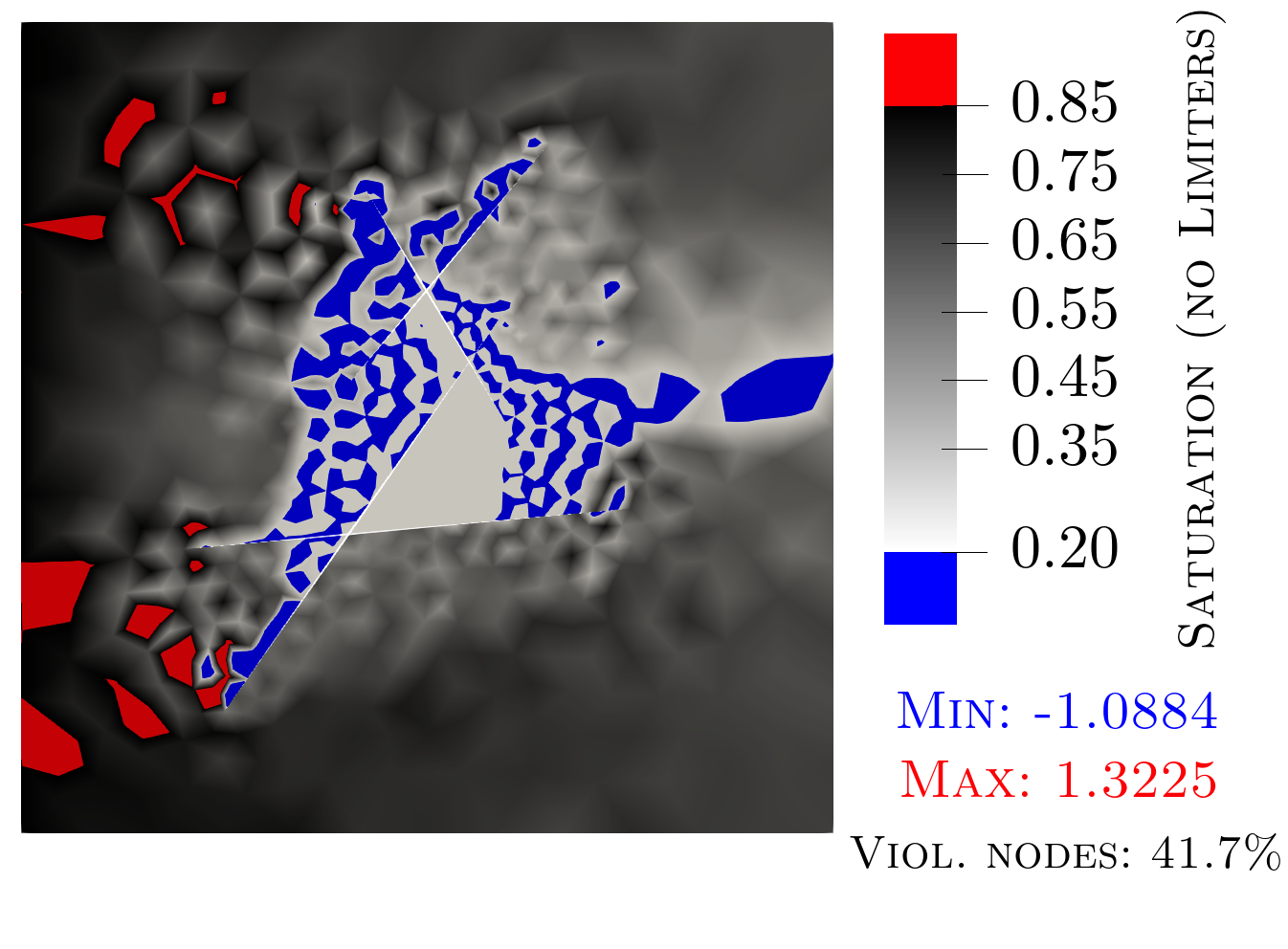}} 
        \hspace{.25cm}
        \subfigure[DG scheme with proposed limiters; $t=1000$ \si{\second}~\label{fig:}]{
        \includegraphics[clip,scale=0.15,trim=0 0cm 0cm 0]{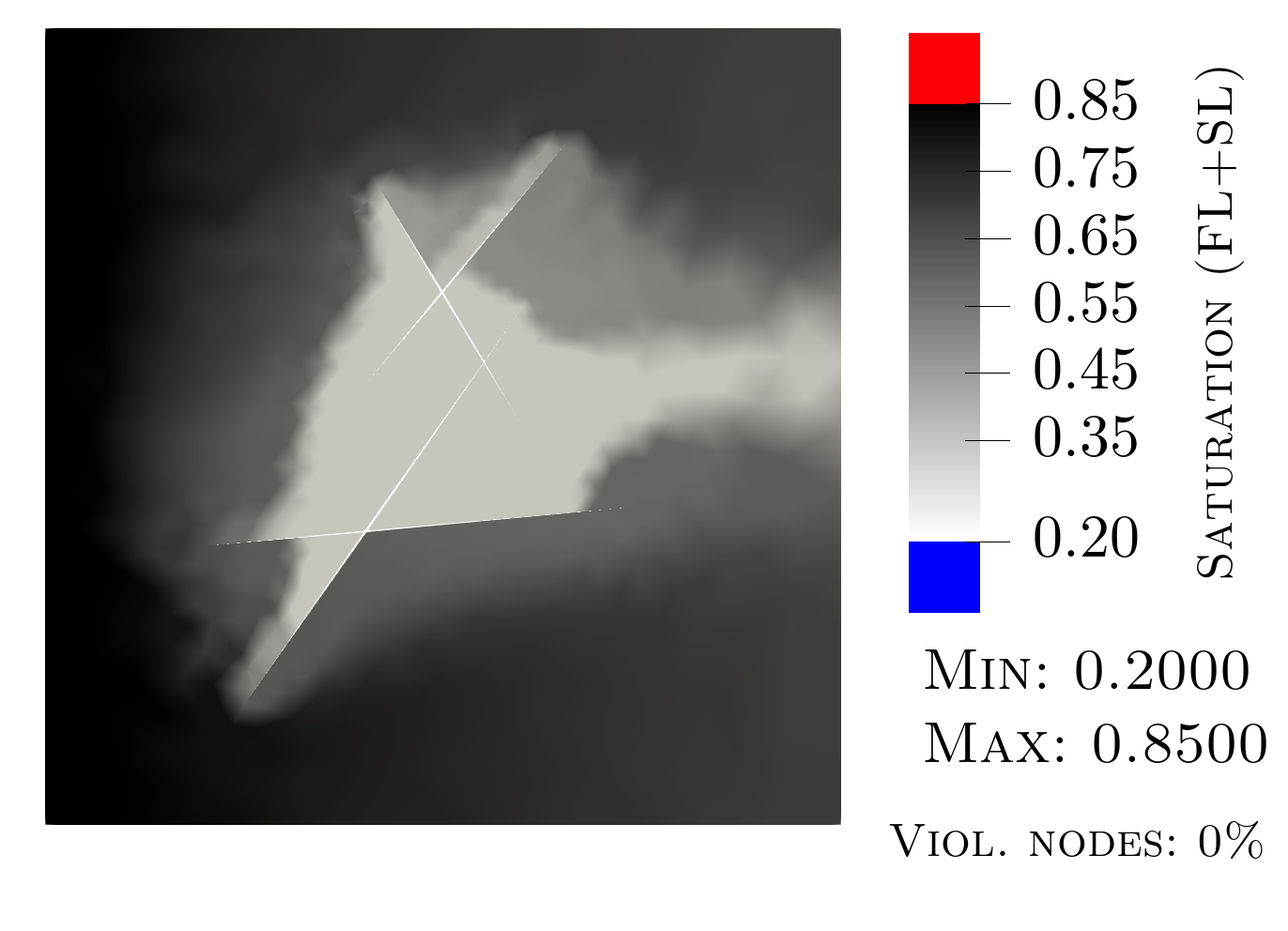}} 
        \vspace{0.25cm}

			\emph{{\small{
            This figure shows the saturation profiles of a pressure-driven flow problem at time $t=1000$
            \si{\second}. The porous medium is homogeneous and contains a thin barrier.
            Solutions are obtained using discontinuous Galerkin (DG) scheme without limiters (left) 
            and with the proposed limiters (right).             
            We observe that DG approximation with no limiter yields noticeable violations,
            while limited DG scheme is capable of providing maximum-principle satisfying results. 
            The physical range for solutions 
            is between $s_{rw}$ and $1-s_{r\ell}$ (between $0.2$ and $0.85$ in this problem); 
            and is shown in grayscale. Values below and above bounds are colored blue and red, 
            respectively.
				      }}
					}
					\captionsetup{labelformat=empty}
				\vspace{-.9cm}
				\caption{}
		\end{figure}
    \vfill
    {\scshape 2021} \\
    {\small Computational Modeling of Porous Media (COMP-M) Group} \par
\end{titlepage}
\setcounter{figure}{0}
\begin{abstract}
This paper proposes a fully implicit numerical scheme for immiscible incompressible two-phase flow in porous media taking into 
account gravity, capillary effects, and heterogeneity. The objective is to develop a fully implicit stable 
discontinuous Galerkin (DG) solver for this system that is accurate, bound-preserving,   
 and locally mass conservative. To achieve this, we augment our DG formulation with post-processing flux and slope limiters.
The proposed framework is applied to several benchmark problems and the discrete solutions
are shown to be accurate, to satisfy the maximum principle and local mass conservation.
\end{abstract}
\maketitle
\section{Introduction}

Multiphase flows in porous media appear in a large number of applications in engineering and sciences, for instance in
the environmental clean up of contaminated subsurface or in the energy production of hydrocarbons from reservoirs. This paper
introduces a numerical method for solving the immiscible two-phase flow equations, that produces bound-preserving discrete saturations. The proposed method utilizes a fully implicit in time stepping scheme, a discontinuous Galerkin in space discretization and 
post-processing flux and slope limiters techniques. The resulting numerical saturation is shown to satisfy a maximum principle theoretically and computationally. 

The numerical literature for immiscible two-phase flow problems is vast (see \cite{peaceman2000fundamentals,azizpetroleum,chen2006computational} and references herein). Suitable numerical methods should be locally mass conservative and should produce bound-preserving discrete saturations. Such methods include finite difference methods and finite volume methods, which are popular methods because of their simplicity and low cost \cite{Michel03,Droniou2014}.  However, finite difference methods are not adapted to unstructured meshes and cell-centered finite volume methods suffer from grid distortion and do not easily handle full anisotropy. The class of interior penalty discontinuous Galerkin methods has been applied to model multiphase flows in porous media for more than fifteen years \cite{KlieberRiviere2006,epshteyn2007fully,EpstheynRiviere2006,ErnMozolevskiSchuh2010,arbogast2013discontinuous,Bastian2014,JameiGhafouri2016} and they have been combined with other locally mass conservative methods like mixed finite element methods in \cite{HoteitFiroozabadi2008,HouChenSunChen2016}. DG methods are locally mass conservative, they do not suffer from grid distortion and they are accurate and robust even in the case of anisotropic heterogeneous media. However, it is well known that the DG approximation of the saturation does not satisfy a maximum principle because of  local overshoots and undershoots in the neighborhood of the saturation front. While the amount of overshoot and undershoot can be reduced by the choice of implicit time stepping, mesh refinement and appropriate penalty parameters, there is no guarantee that they will completely disappear. 
The literature on post-processing techniques to reduce or eliminate the amount of overshoots and undershoots for DG methods in general is significant.  Slope limiters adjust the gradient of the linear approximation in a heuristic way \cite{Burbeau2001,Hoteit2004,Krivodonova2007,Krivodonova2004,kuzmin2010vertex,kuzmin2013}.  Recently, flux limiters related to flux-corrected transport algorithms, were introduced for DG discretizations of conservation laws \cite{frank2019bound,kuzmin2012flux}.

The main contribution of this paper is the formulation of bound-preserving numerical method for the incompressible two-phase flow problems. 
Upwind fluxes are employed for the interior penalty discontinuous Galerkin discretization in space. 
We solve several benchmark problems to investigate the performance of the method and particularly the impact of the limiting techniques on local mass conservation.  The numerical method respects maximum principle by limiting the saturation profile to physical upper- and lower-bounds. The violation of maximum principle for the discontinuous approximation of the saturation has been an open problem over the last decade. Our proposed scheme guarantees that the saturation remains bounded in the physical range. In addition, we observe that the monotonicity of the saturation is significantly improved compared to the case of no limiters. Saturation fronts are sharp with 
minimal numerical diffusion. 
We present several numerical results that show overshoots and undershoots have been eliminated.  We verify that the local mass conservation property is also satisfied. We consider cases where flow is driven by boundary conditions and cases where flow is driven by injection and production wells. In the former case, a theoretical proof of the maximum principle is given.

The content of the paper is as follows. Section~\ref{Sec:S1_GEs} describes the mathematical equations.  The primary unknowns are the wetting phase pressure and saturation. Section~\ref{Sec:Numerical_method} contains the fully implicit numerical scheme, with the construction and analysis of the flux limiters, and the review of slope limiters used in this work.  Several numerical results, including benchmark problems and convergence tests, are given in Section~\ref{sec:S4_NR}. Conclusions follow.

%
\section{GOVERNING EQUATIONS}
\label{Sec:S1_GEs}

The incompressible two-phase flow in a porous medium $\Omega\subset\mathbb{R}^2$  over a time interval $[0,T]$,
is modeled by a system of mass balance equations for each phase, coupled with closure relations.
%
\begin{alignat}{1}
&\frac{\partial}{\partial t} (\phi  S_\alpha) 
-\nabla \cdot (\lambda_\alpha K (\nabla P_\alpha -   \rho_\alpha\mathbf{g})) = q_\alpha,
\quad \alpha=\ell,w,\\
&S_{\ell}+S_w = 1,\label{eq:sumsat}\\
&P_c = P_{\ell}-P_{w}.\label{eq:pcdef}
\end{alignat}
where $P_w$ (resp. $P_\ell$)  is the wetting phase (resp. non-wetting phase) pressure and 
$S_w$ (resp. $S_\ell$) is the wetting phase (resp. non-wetting phase)  saturation. The source/sink functions are denoted
by $q_\alpha$, the phase mobility coefficient by $\lambda_\alpha$ and the capillary pressure, $P_c$. 
The phase mobilities are ratios of the relative permeabilities, $k_{r\alpha}$, to the phase viscosities, $\mu_{\alpha}$. 
Relative permeabilities and capillary pressure are given functions of the wetting phase saturations \citep{brooks1964hydraulic}.
\begin{align}
    \lambda_{\alpha}(S_w) = \frac{k_{r\alpha}(S_w)}{\mu_{\alpha}}, \quad \alpha= w,\ell.
\end{align}
The other coefficients are the porosity $\phi$,
the absolute permeability $K$, and the gravity vector $\mathbf{g}$. Using \eqref{eq:sumsat} and \eqref{eq:pcdef}, and choosing
for primary unknowns the wetting phase pressure and saturation $(P,S)=(P_w,S_w)$, the system of equations reduces to:
\begin{alignat}{2}
    \label{Eqn:BoM_p}
	&\frac{\partial}{\partial t} \big(\phi (1-S)\big)
    -\nabla \cdot \Big( \lambda_\ell(S) K \big(\nabla P + {\color{black}\nabla P_c(S)} -  \rho_{\ell}\mathbf{g}\big) \Big) 
    = q_\ell, 
    &\qquad \mathrm{in} \; \Omega \times (0,T),
    \\
    \label{Eqn:BoM_s}
	&
	\frac{\partial}{\partial t} \big(\phi S\big)
    -\nabla \cdot \Big( \lambda_w(S) K \big(\nabla P - \rho_{w}\mathbf{g}\big) \Big) 
    = q_w, 
    &\qquad \mathrm{in} \; \Omega \times (0,T).
\end{alignat}
Let the boundary of the domain be divided into different disjoint sets
\[
\partial\Omega = \Gamma^{\mathrm{D},p}\cup\Gamma^{\mathrm{N},p}=\Gamma^{\mathrm{D},s}\cup\Gamma^{\mathrm{N},s}\cup\Gamma^{\mathrm{out},s}.
\]
Dirichlet and  Neumann boundary conditions are imposed on parts of the boundary:
\begin{alignat}{2}
	& P = g^p,
	&& \quad  \mathrm{on} \; \Gamma^{\mathrm{D},p}\times (0,T), \\
	& S = g^s, 
	&& \quad \mathrm{on} \; \Gamma^{\mathrm{D},s}\times (0,T), \\
    &  \lambda_{\ell}(S) K \big(\nabla P + \nabla P_c(S) -  \rho_{\ell}\mathbf{g} \big) \cdot \mathbf{n} = j^p,
	&& \quad\mathrm{on} \; \Gamma^{\mathrm{N},p}\times (0,T), \\
    &  \lambda_w(S)  K \big(\nabla P - \rho_{w}\mathbf{g}\big) \cdot \mathbf{n} = j^s,
	&&\quad \mathrm{on} \; \Gamma^{\mathrm{N},s}\times (0,T). 
    \label{eqn:bcs}
\end{alignat}
The boundary $\Gamma^{\mathrm{out},s}$ is referred to as a free boundary because no data is prescribed on that boundary. This means
that the surface integrals on this boundary are evaluated in terms of the unknowns.  This particular treatment of the outflow boundary has been highlighted in the works of \citep{papanastasiou1992new,griffiths1997no}.
In the case of pure homogeneous Neumann boundary conditions ($\Gamma^{\mathrm{N},s}=\Gamma^{\mathrm{N},p} = \partial\Omega)$ and $j^p=j^s=0$), the flow is driven by injection/production wells (source/sink functions) that depend on the wetting phase saturation as follows:
\[
q_\alpha(S) = f_{\alpha}(s_{\mathrm{in}})\bar{q}-f_{\alpha}(S)\underline{q}, \quad \alpha=\ell,w.
\]
The functions $\bar{q}$ and $\underline{q}$ correspond to the injection and production well rates and $s_\mathrm{in}$ is the prescribed wetting phase saturation at the injection wells. The fractional flow functions, $f_\alpha$, are the ratios of the phase mobility to the total mobility, $f_\alpha = \lambda_\alpha/(\lambda_\ell+\lambda_w)$.

Finally the model problem is completed by an initial condition on the saturation: $S=s_0$.


%
\section{Numerical method}
\label{Sec:Numerical_method}
The domain $\Omega$ is decomposed into a non-degenerate partition $\mathcal{E}_h=\{E\}_E$
consisting of $N_h$ triangular or rectangular elements of maximum diameter $h$.   
Let $\Gamma_h$ denote the set of all edges and $\Gamma_h^{i}$ denote the set of interior edges.
For any $e \in \Gamma_h^{i}$, fix a unit normal vector $\mathbf{n}_e$ and denote by $E^+$ and $E^-$ 
 the elements that share the edge $e$ such that the vector
$\mathbf{n}_e$ is  directed from $E^+$ to $E^-$.
We define the jump and average of a scalar function $\xi$ on $e$ as follows:
\begin{align}
    [\xi] = \xi\vert_{E^+}-\xi\vert_{E^-},\quad
    \{\xi\} = \frac{1}{2}\left( \xi\vert_{E^+}+\xi\vert_{E^-}\right). 
\end{align}
By convention, if $e$ is adjacent to $\partial \Omega$, then the jump and average of $\xi$ on $e$ coincide 
with the trace of $\xi$ on $e$ and the normal vector $\mathbf{n}_e$ coincides with the outward normal $\mathbf{n}$.
Let $\mathbb{P}_1(E)$ be the space of linear polynomials  on an element $E$. The discontinuous finite
element space of order one is:
\begin{align}
    \mathcal{D}(\mathcal{E}_h) = \left\{\xi\in L^2(\Omega): \xi\vert_E\in\mathbb{P}_1(E), \,  \forall E \in \mathcal{E}_h\right\}.
\end{align}
The time interval $T$ is divided into $N_\tau$ equal subintervals of length $\tau$. Let $P_{n}$ and $S_{n}$ be the numerical 
solutions at time $t_n$. The proposed discontinuous Galerkin scheme for equations~\eqref{Eqn:BoM_p}--\eqref{eqn:bcs} reads: 
Given $(P_{n},S_{n})\in \mathcal{D}(\mathcal{E}_{h}) \times \mathcal{D}(\mathcal{E}_{h})$, find
$(P_{n+1},S_{n+1})\in \mathcal{D}(\mathcal{E}_{h}) \times \mathcal{D}(\mathcal{E}_{h})$ such that: \\
\begin{align}
    &
    \frac{1}{\tau} \int_\Omega \phi (1-S_{n+1})  \xi
  +
  \sum_{E\in\mathcal{E}_h}\int_E 
  \lambda_{\ell}(S_{n+1})K \big(\nabla P_{n+1}
      +\nabla P_c(S_{n+1})
  - \rho_{\ell}\mathbf{g}\big) \cdot \nabla \xi \nonumber\\
  &-
  \sum_{e\in\Gamma_h^{i}} \int_e
          (\lambda_{\ell}(S_{n+1}) )^{\uparrow\mathbf{v}_\ell^n}
          \{ K\big(\nabla  P _{n+1} +\nabla P_c(S_{n+1})- \rho_{\ell}\mathbf{g}\big) \cdot {\mathbf{n}}_e \} [\xi] \nonumber\\
  &-
  \sum_{e\in\Gamma^{\mathrm{D},p}} \int_e
          \lambda_{\ell}(S_{n+1}) K
          {\big(\nabla P_{n+1}+\nabla P_c(S_{n+1}) - \rho_{\ell}\mathbf{g}\big) }\cdot {\mathbf{n}}_e \, \xi  
   + \sum_{e\in\Gamma_h \cup\Gamma^{\mathrm{D},p}}  \frac{\sigma}{h} \int_e [P_{n+1}]\, [\xi]
  %
\nonumber\\
  & =  \int_\Omega q_{\ell}(S_n) \xi  
 +\frac{1}{\tau} \int_\Omega \phi (1-S_{n})  \xi
  + \frac{\sigma}{h} \int_{\Gamma^{\mathrm{D},p}}  g^p  \xi 
 +  \int_{\Gamma^{\mathrm{N},p}}   j^{p}\xi,
 \quad \forall \xi \in \mathcal{D}(\mathcal{E}_h),
    \label{Eqn_disc1}
\end{align}
\begin{align}
	&  
    \frac{1}{\tau} \int_\Omega \phi S_{n+1}  \xi
    + \sum_{E\in\mathcal{E}_h}\int_E 
    \lambda_{w}(S_{n+1})K \big(\nabla P_{n+1} - \rho_{w}\mathbf{g}\big) \cdot \nabla \xi \nonumber\\
  & - \sum_{e\in\Gamma_h^{i}} \int_e (\lambda_{w}(S_{n+1}) )^{\uparrow\mathbf{v}_w^n}
          \{ K \big(\nabla P_{n+1} -  \rho_{w}\mathbf{g}\big)  \cdot {\mathbf{n}_e}
   \} [\xi] 
  - \sum_{e\in \Gamma^{\mathrm{D},s}} \int_e \lambda_{w}(g^s) K \big(\nabla P_{n+1} - \rho_{w}\mathbf{g}\big)\cdot{\mathbf{n}}  \, \xi   \nonumber\\
  &-
  \sum_{e\in \Gamma^{\mathrm{out}}} \int_e \lambda_{w}(S_{n+1})K
  \big(\nabla P_{n+1} - \rho_{w}\mathbf{g}\big)  \cdot {\mathbf{n}}_e \; \xi   
  + \sum_{e\in\Gamma_h^i\cup \Gamma^{\mathrm{D},s}}  \frac{\sigma}{h} \int_e [S_{n+1}]\, [\xi] \nonumber\\
  &  = \int_\Omega q_{w}(S_n) \xi  
	+\frac{1}{\tau} \int_\Omega \phi S_n \xi
    + \frac{\sigma}{h} \int_{\Gamma^{\mathrm{D},s}}  g^s  \xi 
    + \int_{\Gamma^{\mathrm{N},s}}  j^s \xi,
    \qquad \forall \xi \in \mathcal{D}(\mathcal{E}_h).
    \label{Eqn_disc2}
\end{align}
The penalty parameter $\sigma$ is constant on the interior edges and its value is chosen $10$ times larger on the Dirichlet boundaries.
The quantities $(\cdot)^{\uparrow\mathbf{v}_\ell^n}$ and $(\cdot)^{\uparrow\mathbf{v}_w^n}$ denote
the upwind values with respect to the vector functions $\mathbf{v}_\ell^n$ and $\mathbf{v}_w^n$ that are scaled quantities of the phase velocities. They depend on the pressure and saturation evaluated at the previous time $t_n$: 
\[
    \label{Eqn:Velocities}
\mathbf{v}_w^n  =  -K\big(\nabla P_{n} -  \rho_{w}\mathbf{g}\big),  \quad
\mathbf{v}_\ell^n  =  -K\big(\nabla P_{n} +\nabla P_c(S_{n}) -  \rho_{\ell}\mathbf{g}\big)  
\]
The definition of the upwind operator with respect to a generic discontinuous vector field $\mathbf{v}$ is:
\[
\forall e  = \partial E^+\cap \partial E^-, \quad
\xi^{\uparrow\mathbf{v}}|_e = \left\{
\begin{array}{c}
\xi\vert_{E^+}, \quad\mbox{if} \quad \{\mathbf{v}\}\cdot\mathbf{n}_e >0,\\
\xi\vert_{E^-}, \quad\mbox{if} \quad \{\mathbf{v}\}\cdot\mathbf{n}_e \leq 0.
\end{array}
\right.
\]
At the initial time, the discrete saturation is the $L^2$ projection of the initial condition.
\[
\int_\Omega S_0 v = \int_\Omega s_0 v, \quad \forall v\in\mathcal{D}(\mathcal{E}_h).
\]
At each time step, we solve \eqref{Eqn_disc1}-\eqref{Eqn_disc2} together with a Newton solver, followed
by flux and slope limiters (see Algorithm~\ref{alg}). 
Figure~\ref{Fig:Scheme} is a schematic that describes the actions of both flux and slope limiters on the discrete saturation.  The next
two sections describe these limiters in detail and contain a proof that the resulting saturation is bound-preserving.
\begin{algorithm}
\caption[Scheme]{DG+FL+SL method}\label{alg}
\begin{algorithmic}
\State Compute initial saturation $S_0$
\For{$n=0,\dots, (N_\tau-1)$}
\State Solve \eqref{Eqn_disc1}-\eqref{Eqn_disc2} with Newton's method
\State Apply flux limiter: $S_{n+1}^\mathrm{FL} = \mathcal{L}_\mathrm{flux}(S_{n+1})$
\State Apply slope limiter to $S_{n+1} = \mathcal{L}_\mathrm{slope}(S_{n+1}^\mathrm{FL})$
\EndFor
\end{algorithmic}
\end{algorithm}
\subsection{Flux limiter}
\label{Sub:Flux_limiter}
The flux limiter will enforce that the element-wise average of the saturation satisfies the desired physical bounds. 
We assume that the saturation at the previous time step, $t_n$,  satisfies:
\begin{align}
    s_\ast \le S_n(\mathbf{x}) \le s^\ast, \quad \forall \mathbf{x} \in \Omega.
\label{eq:physbounds}
\end{align}
for some constants $0\leq s_\ast\leq s^\ast\leq 1$. 
The flux limiting is applied to each element $E$ given the element-wise average of the saturation at the previous and current time steps and
given a flux function defined on each face $e\subset\partial E$. First we compute the element-wise
average at time $t_n$ and $t_{n+1}$:
\[
\overline{S_{i}}|_E = \overline{S_{i,E}}, \quad
\overline{S_{i,E}}= \frac{1}{|E|}\int_{E} S_{i}, 
\quad \forall E\in\mathcal{E}_h, \quad i=n,n+1.
\]
Next, for a fixed element $E$, let $\mathbf{n}_E$ be the unit outward normal vector to $E$. We 
define the flux function $\mathcal{H}_{n+1}|_E = \mathcal{H}_{n+1,E}$ as follows:
   \begin{align}
           \forall e = \partial E\cap \partial E', \quad \mathcal{H}_{n+1,E}(e) &=
          - \int_e (\lambda_{w}(S_{n+1}) )^{\uparrow\mathbf{v}_w^n}
          \{ K \big(\nabla P_{n+1} -\rho_{w}  \mathbf{g}\big) \cdot  {\mathbf{n}}_{E} \} \nonumber\\
          & + \frac{\sigma}{h} \int_e (S_{n+1}|_E-S_{n+1}|_{E'}) \\
           \forall e \in\partial E \cap \Gamma^{\mathrm{D},s}, \quad \mathcal{H}_{n+1,E}(e)& =
          - \int_e \lambda_{w}(g^s) K \big(\nabla P_{n+1} -\rho_{w}  \mathbf{g}\big) \cdot {\mathbf{n}}_{E} 
           + \frac{\sigma}{h} \int_e (S_{n+1} - g^s),\\
           \forall e \in\partial E \cap \Gamma^{\mathrm{N},s}, \quad \mathcal{H}_{n+1,E}(e) &= \int_e j^s,\\
           \forall e \in\partial E \cap \Gamma^{\mathrm{out}}, \quad \mathcal{H}_{n+1,E}(e) &= \int_e \lambda_w(S_{n+1}) K (\nabla P_{n+1}-\rho_w \mathbf{g})\cdot\mathbf{n}_E.
\end{align}
 For an interior face $e$ of the element $E$, the quantity $\mathcal{H}_{n+1,E}(e)$ measures the net mass flux 
across $e$ into the neighboring element $E'$ that also shares the face $e$. We note that: 
\[\mathcal{H}_{n+1,E}(e) = - \mathcal{H}_{n+1,E'}(e).
\]
After application of the flux limiter operator, the limited saturation has a possibly different cell-average:
\begin{equation}
S_{n+1}^{\mathrm{FL}} = \mathcal{L}_\mathrm{flux}(S_{n+1}), \quad
S_{n+1}^\mathrm{FL}(\bfx) =  S_{n+1}(\bfx)-\overline{S_{n+1,E}}+\bar{S}^{\mathrm{FL}}_{n+1}|_E, \quad \forall \bfx \in E.
\label{eq:updatecellaverage}
\end{equation}
The new cell-average of the saturation is obtained by an iterative process, that takes for input the cell average at the previous time
step and the flux function:
\[
\bar{S}^{\mathrm{FL}}_{n+1} = \mathcal{L}_\mathrm{avg}(\overline{S_{n}},\mathcal{H}_{n+1}).
\]
Before showing that the limited saturation satisfies \eqref{eq:physbounds}, we describe the algorithm for the 
operator $\mathcal{L}_\mathrm{avg}$.

\subsubsection{The algorithm for $\mathcal{L}_\mathrm{avg}$}

For a fixed element $E$, we  denote by $\mathcal{N}_E$ the set of elements that include $E$ and all neighboring elements $E'$ that share a face $e$ with $E$. 
The algorithm constructs a sequence of flux functions and element-wise averages for $E$ and its neighbors $E'$.   While the construction
of the element-wise averages are local to $E$ and its neighbors $E'$, the stopping criterion is global to ensure bound-preserving solutions.
We first initialize the sequences with the input arguments:
\[
\bar{S}_{\tilde{E}}^{(0)} =\overline{S_{n,\tilde{E}}},\quad \mathcal{H}_{\tilde{E}}^{(0)} = \mathcal{H}_{n+1,\tilde{E}},
\quad \forall \tilde{E} \in \mathcal{N}_E.
\]
Next, for $k\geq 1$, we have the following steps:
\begin{itemize}
\item[Step~1] Compute inflow and outflow fluxes:
        \begin{align}
        P_{\tilde{E}}^+ = \tau \sum_{e \in \partial \tilde{E}}\max(0,-\mathcal{H}_{\tilde{E}}^{(k-1)}(e)),
        \quad
        P_{\tilde{E}}^- = \tau \sum_{e \in \partial \tilde{E}}\min(0,-\mathcal{H}_{\tilde{E}}^{(k-1)}(e)), \quad \forall \tilde{E} \in \mathcal{N}_E.
        \end{align}
\item[Step~2] Compute admissible upper and lower bounds for all $\tilde{E}\in \mathcal{N}_E$:
        \begin{align}\label{eq:defQplus}
            Q_{\tilde{E}}^+ =& |\tilde{E}|\Big(\phi s^\ast  -\phi\bar{S}_{\tilde{E}}^{(k-1)}
                - \gamma_{1k} \tau \big(f_w(s_\mathrm{in})\bar{q}_{\tilde{E}} -  f_w(\bar{S}_{\tilde{E}}^{(k-1)})
                \underline{q}_{\tilde{E}}\big)  \Big),\\
            Q_{\tilde{E}}^- =& |\tilde{E}|\Big(\phi s_\ast
            -\phi\bar{S}_{\tilde{E}}^{(k-1)} - \gamma_{1k} \tau
                \big(
                 f_w(s_\mathrm{in})\bar{q}_{\tilde{E}} - f_w(\bar{S}_{\tilde{E}}^{(k-1)})
                \underline{q}_{\tilde{E}}\big).
             \Big)\label{eq:defQminus}
        \end{align}
The scalar factor $\gamma_{1k}$ is equal to $1$ if $k=1$ and $0$ otherwise.
The injection and production well rates, restricted to any element $\tilde{E}$, are denoted by $\bar{q}_{\tilde{E}}$ and $\underline{q}_{\tilde{E}}$ respectively.
They are assumed to be piecewise constant fields; otherwise we take the element-wise average of the flow rates.
\item[Step~3]   Compute limiting factors $\alpha_E^{(k-1)}(e)$ for all faces $e \subset \partial E$.
If $e$ is an interior face such that $e=\partial E\cap\partial E'$:
\[
\alpha_E^{(k-1)}(e) = \left\{
\begin{array}{cc}
\min\left( \min(1,Q^+_E/P^+_E), \min(1,Q^-_{E'}/P^-_{E'}) \right) & \mbox{if} \quad \mathcal{H}_E^{(k-1)}(e) < 0, \\
\min\left( \min(1,Q^-_E/P^-_E), \min(1,Q^+_{E'}/P^+_{E'}) \right) & \mbox{if} \quad \mathcal{H}_E^{(k-1)}(e) > 0.
\end{array}
\right.
\]
If $e$ is a boundary face:
\[
\alpha_E^{(k-1)}(e) = \left\{
\begin{array}{cc}
\min(1,Q^+_E/P^+_E) & \mbox{if} \quad \mathcal{H}_E^{(k-1)}(e) < 0, \\
\min(1,Q^-_E/P^-_E) & \mbox{if} \quad \mathcal{H}_E^{(k-1)}(e) > 0.
\end{array}
\right.
\]
\item[Step~4]  Update $\bar{S}_{E}^{(k)}$ and $\mathcal{H}_E^{(k)}$ as follows:
                  \begin{align}
                      \bar{S}_{E}^{(k)}
                      &=
                      \bar{S}_{E}^{(k-1)}
                      -
                      \frac{\tau}{\phi |E|} \sum_{e\subset\partial E}\alpha_E^{(k-1)}(e) \mathcal{H}_E^{(k-1)}(e)
                       + \frac{\gamma_{1k}\, \tau}{\phi}
                  \Big(  f_w(s_\mathrm{in}) \bar{q}_E  -  f_w(\bar{S}_{E}^{(k-1)})\underline{q}_E\Big),\label{eq:updateFL}\\
      \mathcal{H}_E^{(k)}(e) &= (1-\alpha_E^{(k-1)}(e))\, \mathcal{H}_E^{(k-1)}(e),
      \quad \forall e \subset \partial E.
          \end{align}
\item[Step~5.] Define a global stopping criterion\\
If  $\left(\max_{E\in\mathcal{E}_h}  \vert \mathcal{H}_E^{(k)}\vert
<\epsilon_1\right)$ or $\left(\max_{E\in\mathcal{E}_h}  \vert \mathcal{H}_E^{(k)}-\mathcal{H}_E^{(k-1)}\vert
<\epsilon_2\right)$ for $k\geq 2$\\
\hspace*{0.5cm} return $\bar{S}^{\mathrm{FL}}_{n+1}|_E = \bar{S}_E^{(k)}$.\\
Else\\
\hspace*{0.5cm} set $k\leftarrow k+1$ and go to Step~1.
\end{itemize}

\subsubsection{Bound-preserving solutions}

In this section, we show that the solution $S_{n+1}^\mathrm{FL}$ obtained in \eqref{eq:updatecellaverage} has a cell-average that is bound-preserving for the case where flow is driven by boundary conditions only (no wells).  Clearly, it suffices to show that $\bar{S}_{n+1}^{\mathrm{FL}}$ is bound-preserving. This is done in two steps. First, we show that each iterate in the flux-limiter algorithm is bound preserving. Second, we show that the stopping criterion is reached for some value $k_0$.
\begin{lemma}\label{thm:flux:boundedness}
Let $E$ be a mesh element and let $(\bar{S}_{E}^{(k)})_k$ be the sequence obtained in the algorithm $\mathcal{L}_\mathrm{flux}$.
Assume that the iterate $\bar{S}_{E}^{(k-1)}$ belongs to the interval $[s_\ast,s^\ast]$. Then the next iterate
$\bar{S}_{E}^{(k)}$ also belongs to the interval $[s_\ast,s^\ast]$.
\end{lemma}
\begin{proof}
Let us check the upper bound: $\bar{S}^{(k)}_{E} \leq s^\ast$. 
Since for an interior face $e$, we have: $\alpha_E^{(k-1)}(e) = \alpha_{E'}^{(k-1)}(e)$, 
it is easy to check by induction on $k$ that $\mathcal{H}_E^{(k)} = -\mathcal{H}_{E'}^{(k)}$. 
Since the iterate $\bar{S}^{(k-1)}_{E}$ belongs to the interval $[s_\ast,\,s^\ast]$, it then follows by its definition that
$\alpha_{E}^{(k-1)}(e)\geq0$ for all $ e\subset\partial{E}$.  We then apply the inequality $x\leq \max(0,x)$ to \eqref{eq:updateFL}
to obtain:
\begin{align}
\bar{S}_{E}^{(k)}\leq &\bar{S}_{E}^{(k-1)}+ \frac{\tau}{\phi |E|} 
\sum_{e\subset\partial E}\alpha_E^{(k-1)}(e) \max(0,-\mathcal{H}_E^{(k-1)}(e))\nonumber\\
\leq & \bar{S}_{E}^{(k-1)}+ \frac{\tau}{\phi |E|} 
\sum_{e\subset\partial E} \frac{Q_E^+}{P_E^+} \max(0,-\mathcal{H}_E^{(k-1)}(e))\nonumber\\
\leq & \bar{S}_{E}^{(k-1)}+ \frac{1}{\phi |E|} Q_E^+.
\end{align}
Therefore with the definition of $Q_E^+$ we have
\[
\bar{S}_{E}^{(k)} \leq \bar{S}_{E}^{(k-1)} + (s^\ast-\bar{S}_{E}^{(k-1)}) = s^\ast.
\]
The proof for the lower bound $\bar{S}^{(k)}_{E} \geq S_\ast$ follows a similar argument, after applying the identity $x\geq \min(0,x)$
to \eqref{eq:updateFL}.
\end{proof}

\begin{lemma}\label{thm:flux:average}
Assume that the cell averages $\overline{S_{n,E}}$ at time $t_n$ belong to the interval $[s_\ast,s^\ast]$ for all elements $E$. 
Then we have
\[
s_\ast \leq \bar{S}_{n+1}^\mathrm{FL}|_E \leq s^\ast, \quad \forall E\in\mathcal{E}_h.
\]
\end{lemma}
\begin{proof}
With Lemma~\ref{thm:flux:boundedness}, it suffices to show that the sequence $(\mathcal{H}_E^{(k)})$ converge as $k$ tends to infinity, for all $E$ in $\mathcal{E}_h$. Since $\alpha_E^{(k-1)}$ belongs to $[0,1]$, it is easy to show by induction on $k$ that
\[
\max_{E\in\mathcal{E}_h} \max_{e\subset\partial E} \vert \mathcal{H}_E^{(k)} \vert \leq
\max_{E\in\mathcal{E}_h} \max_{e\subset\partial E} \vert \mathcal{H}_E^{(k-1)}\vert.
\]
This implies convergence of $(\mathcal{H}_E^{(k)})$ for all elements $E$, so that there exists $k_0$ such that
\[
\left(\max_{E\in\mathcal{E}_h}  \vert \mathcal{H}_E^{(k_0)}\vert
<\epsilon_1\right), \mbox{ or } \left(\max_{E\in\mathcal{E}_h}  \vert \mathcal{H}_E^{(k_0)}-\mathcal{H}_E^{(k_0-1)}\vert
<\epsilon_2\right).
\]
Since $\bar{S}_{n+1}^\mathrm{FL}|_E=\bar{S}_E^{k_0}$, we conclude the proof.
\end{proof}

\subsection{Slope limiter}
\label{Sub:Slope_limiter}
The slope limiter operator, denoted by $\mathcal{L}_\mathrm{slope}$, is applied to the discrete 
saturation $S_{n+1}^\mathrm{FL}$ at each time step. 
The element-wise mean values of the saturation are left unchanged by this procedure.
%
There is a variety of slope limiters available in 
the literature. For convenience, we choose a vertex-based slope limiter that is well suited for  piecewise
linear polynomials \cite{kuzmin2010vertex} and that consists of two steps. 
\begin{enumerate}[(i)]
    \item We first mark the elements in which the maximum principle is not satisfied 
        (i.e., $S_{n+1}^\mathrm{FL}(\mathbf{x}) > s^\ast$ or $S_{n+1}^\mathrm{FL}(\mathbf{x}) < s_\ast$).
        We will apply the slope limiter on these marked elements only.
    \item By a Taylor expansion around the centroid $\mathbf{c}_E$ of element $E$, the linear saturation takes the form
        \begin{align}
            S_{n+1}^\mathrm{FL} \vert_E(\mathbf{x}) = \bar{S}_{n+1,E}^\mathrm{FL} + \nabla S_{n+1}^\mathrm{FL}
            \cdot (\mathbf{x}-\mathbf{c}_E), \quad \forall \mathbf{x}\in E.
        \end{align}
        For the marked elements, a slope limiter replaces the local solution $S_{n+1}^\mathrm{FL}|_E$ by 
        the following linear constrained reconstruction  
        \begin{align}
            S_{n+1}(\mathbf{x}) = \bar{S}_{n+1,E}^\mathrm{FL}+\beta_{E} \nabla S_{n+1}^\mathrm{FL}
            \cdot(\mathbf{x}-\mathbf{c}_E), \quad \forall \mathbf{x}\in E.
\label{eq:construction}
        \end{align}
Let $\mathbf{v}_{E,i}$ denote the $i^{th}$ vertex of element $E$. We determine the maximum
admissible slope for the constrained reconstruction by choosing values $\beta_E\in[0,1]$ such that
boundedness of $S_{n+1}$ is satisfied at all vertices of $E$:
\begin{align}
    S_*^{E,i} \le S_{n+1}(\mathbf{v}_{E,i}) \le S^*_{E,i},
\end{align}
where $S_*^{E,i}$ and $S^*_{E,i}$ are 
defined as maximum and minimum 
means values of the saturation over all the elements (including $E$) that contain the vertex 
$\mathbf{v}_{E,i}\in E$.
\begin{align}
    S_*^{E,i}=\min_{E'\in\mathcal{E}_h|\mathbf{v}_{E,i}\in E'} \bar{S}_{n+1,E'}^\mathrm{FL},  \quad
    S^*_{E,i}=\max_{E'\in\mathcal{E}_h|\mathbf{v}_{E,i}\in E'} \bar{S}_{n+1,E'}^\mathrm{FL}. 
\end{align}
The bounds of the saturation at all vertices are guaranteed if the correction factor $\beta_E$
is chosen as:
\begin{align}
     \beta_E = \min_i
    \begin{cases}
        \frac{S^*_{E,i}-\bar{S}_{n+1,E}^\mathrm{FL}}{ S_{n+1}^\mathrm{FL}(\mathbf{v}_{E,i})-\bar{S}_{n+1,E}^\mathrm{FL}} \quad 
        &\mbox{if} \qquad  S_{n+1}^\mathrm{FL}(\mathbf{v}_{E,i})  > S^*_{E,i},\\
        1 \quad
        &\mbox{if} \qquad S_*^{E,i} \le S_{n+1}^\mathrm{FL}(\mathbf{v}_{E,i}) \le S^*_{E,i},\\
        \frac{S_*^{E,i}-\bar{S}_{n+1,E}^\mathrm{FL}}{S_{n+1}^\mathrm{FL}(\mathbf{v}_{E,i}) -\bar{S}_{n+1,E}^\mathrm{FL}}
        &\mbox{if} \qquad S_{n+1}^\mathrm{FL}(\mathbf{v}_{E,i})  < S_*^{E,i}.
\end{cases}
\end{align}
\end{enumerate}

Using all the previous results, we obtain that the discrete saturation is bound-preserving.
\begin{proposition}
Let $(S_{n+1})_n$ be the sequence of discrete saturations defined by Algorithm~\ref{alg}.  Assume that the initial saturation is bounded below and above by $s_\ast$ and $s^\ast$ respectively. Then, we have
\begin{equation}\label{eq:satbounds}
s_\ast \leq S_{n+1}(\mathbf{x}) \leq s^\ast, \quad \forall \mathbf{x}\in\Omega.
\end{equation}
\end{proposition}


\begin{figure}[htpb]
    \centering
    \includegraphics[width=0.8\linewidth]{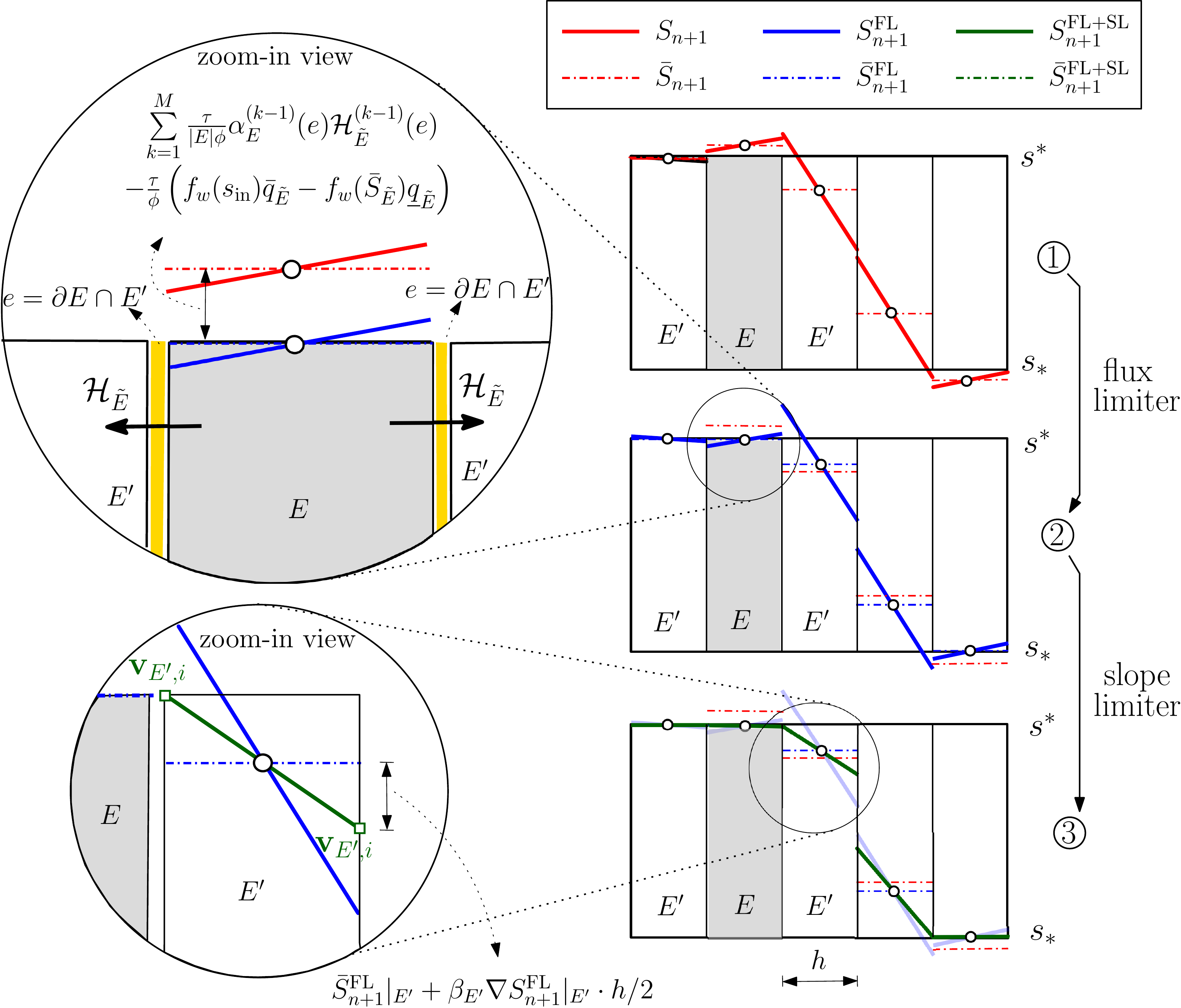}
    \caption{\textsf{Schematic of flux and slope limiters to achieve a pointwise bound-preserving DG solution:}
        Unlimited DG solution (figure \textcircled{\raisebox{-0.9pt}{1}}) are found to violate the upper
        bound $s^*$ and lower bound  $s_*$. 
        After implementing the flux limiter in a post-processing step, the local 
        average values shown in broken blue lines are bounded (figure \textcircled{\raisebox{-0.9pt}{2}}).
        The non-physical fluxes are limited by using the limiting factor $\alpha_{E}^{(k-1)}(e) \in[0,1]$ at the 
        discrete level and performing multiple correction cycles (see section \ref{Sub:Flux_limiter}).
        Note that this procedure is mass conservative and we observe that the decrease of the average value
        on element $E$ brings about an increase on the average value of neighboring elements $E'$. 
        We subsequently apply a slope limiting procedure to produce pointwise bound-preserving solutions
        (see section \ref{Sub:Slope_limiter}). The correction factor $\beta_{E}\in[0,1]$ at each element vertex 
        $\mathbf{v}_{E,i}$ takes the average values of neighboring elements as local bounds and determines 
        the maximum admissible slope (figure \textcircled{\raisebox{-0.9pt}{3}}).
        Note that the local average values are left unchanged during the slope limiting.
    \label{Fig:Scheme}}
\end{figure}

\subsection{Computer implementation and solvers}%
\label{sub:computer_implementation_and_solvers}
We implement the proposed computational framework using the finite element capabilities in Firedrake Project 
\citep{rathgeber2016firedrake,McRae2016,Homolya2016,Homolya2017,homolya2017a} with GNU compilers.
Firedrake is built upon several scientific packages and can employ various computing tools across either
CPUs or GPUs. Software dependencies can be accessed at \citep{zenodo/COFFEE:2020,zenodo/fiat:2021,zenodo/FInAT:2021,zenodo/petsc:2021,zenodo/PyOP2:2021,zenodo/tsfc:2021,zenodo/ufl:2021}. The structured meshes are generated
internally on top of DMPlex grid format \citep{KnepleyKarpeev09} and unstructured meshes are imported from GMSH \citep{geuzaine2009gmsh}.

We utilize the MPI-based PETSc library \citep{petsc-user-ref,petsc-web-page,Dalcin2011} as the linear 
algebra back-end to solve nonlinear equations \eqref{Eqn_disc1}-\eqref{Eqn_disc2}. 
We use Newton's method with (damped) step line search technique \citep{crisfield1979faster} and set 
the relative convergence tolerance to $10^{-6}$. For the inner linear solve at each Newton iteration, 
we rely on the MUMPS direct solver \citep{MUMPS:1,MUMPS:2} with relative pivoting threshold of $0.01$.
MUMPS uses several efficient preordering algorithms to permute the columns of matrix and thereby minimize 
the fill-in (number of nonzeros in the factorization) in the LU factorization.
At each time step, following the Newton solver convergence, we apply flux and slope limiters. 
Implementation of the flux limiter algorithm discussed in section \ref{Sub:Flux_limiter} 
is provided in the module \textsf{FluxLimiter} along with an auxiliary flux wrapper module named 
\textsf{Hsign}. Global stopping criteria for all problem sets are taken as $\epsilon_1=\epsilon_2= 10^{-6}$.
As for the slope limiter, we use the native \textsf{VertexBasedLimiter} module embedded in the 
Firedrake project.
All simulations are conducted on a single socket Intel i5-8257U node by utilizing a single MPI process. 

Codes used to perform experiments in this paper are publicly available at 
\citet{Sarraf2021Github} repository.
Firedrake and its component may be obtained from {https://www.firedrakeproject.org/}.
For reproducibility, we also cite archives of the exact software versions used to produce results in this paper.
All major Firedrake components have been archived on \citet{zenodo/firedrake:2021}. 
This record collates DOIs for the components and can be installed following the instructions at
https://www.firedrakeproject.org/download.html. 


%
\section{NUMERICAL RESULTS}
\label{sec:S4_NR}
In this section, several numerical experiments are carried out in following order:
(i) We first validate our proposed method on two benchmark problems: one-dimensional Buckley-Leverett 
problem and two-dimensional Buckley-Leverett problem with gravity. Further, we investigate the convergence 
rates by using method of manufactured solutions and verify that the flux limiter preserves accuracy.
(ii)
We then perform various pressure-driven flow problems on structured and unstructured meshes, 
to study the efficacy of limiters on capturing high-accuracy wetting phase saturation profiles.
(iii) The robustness of the scheme in the presence of injection and production wells is assessed 
using the quarter five-spot problem, with homogeneous and discontinuous highly varying permeability 
fields.
For both pressure-driven flow problem and quarter five-spot problems, we examine the element-wise mass 
balance property associated with the limiters and highlight the capability of the saturation in satisfying 
the maximum-principle.
(iv) Finally, we study the influence of 
gravity on the flows by testing our scheme with three different gravity numbers.

For all problems, we assume the following parameters unless otherwise specified:
\begin{align*}
    \centering
    &\rho_w = 1000 \mbox{~kg/m\textsuperscript{3}}, \quad
    \rho_{\ell} = 850\mbox{~kg/m\textsuperscript{3}}, \nonumber\\
    &\phi = 0.2, \quad
    s_{rw} = 0.2, \quad
    s_{r\ell} = 0.15,\quad s_0 = 0.2, \quad P_0 = 10^{6} \mbox{~Pa}. 
\end{align*}
The residual saturations imply the physical lower and upper bounds for the saturation:
\[
s_\ast = s_{rw} = 0.2, \quad s^\ast = 1-s_{r\ell} =  0.85.
\]

\subsection{Verification}
\subsubsection{One-dimensional Buckley-Leverett problem}%
\label{Sub:1D_BL}
The original Buckley-Leverett transport equation introduced in 1942 \citep{buckley1942mechanism}, 
also known as the frontal-advance equation, is a well-known non-linear hyperbolic equation 
for the description of one-dimensional immiscible displacement in a linear reservoir. Because the problem has a semi-analytical solution, it
is widely used to validate numerical methods for two-phase flows in porous media.
Since capillary pressure and gravity are neglected, the 
total velocity of the phases $\mathbf{u}_t=(u_t,v_t)^T$ can be written as:
 \begin{align}
     \label{Eqn:Total_vel}
     \mathbf{u}_t = -(\lambda_w+\lambda_{\ell})K\nabla P.
 \end{align}
By substituting $\nabla P$ from \eqref{Eqn:Total_vel}  into equation~\eqref{Eqn:BoM_s} and ignoring
source/sink terms, we obtain the general form of the Buckley-Leverett equation.
\begin{align}
    \label{Eqn:1DBL_general}
    \frac{\partial}{\partial t}(\phi S)- 
    \nabla \cdot \mathbf{f}_{BL} = 
    0 ,\quad \mbox{in} \quad \Omega \times (0,T).
\end{align}
The convection flux $\mathbf{f}_{BL}=(F(S),G(S))^T$ reduces in one-dimension to:
\begin{align}
    \label{Eqn:1DBL}
    F(S) =  \frac{\lambda_w(S){u}_t}{\lambda_w(S)+\lambda_{\ell}(S)}, \mbox{ and} \quad 
    G(S) = 0.
\end{align}
The relative permeabilities are chosen as: 
\begin{align}
    \label{Eqn:Rel-perm}
    k_{rw}(S) = S^4,~k_{r\ell}(S) = (1-S)^2(1-S^2).  
\end{align}
We take an interval domain $\Omega = [0,300]$ \si{\meter} with uniform mesh, and we fix the following parameters: 
\[
u_t=3\times10^{-7} \mbox{\si{\meter\per\second}}, \, \mu_w=\mu_{\ell}=1 \mbox{\si{\pascal\cdot\second}}, \,
 s_0 = 0.1, \, h=12 \mbox{\si{\meter}}, \, \tau=22.2 \mbox{days}.
\]
The Dirichlet boundary condition of $g^s=0.85$ is weakly prescribed at the left boundary $x=0$.
We assume outflow boundary at $x=300$. 
This setup gives rise to the classical Buckley-Leverett profile, which consists of a shock wave 
immediately followed by a rarefaction wave. Both lower and higher DG approximation of solution without any 
external bound-preserving mechanism do not respect maximum principle \citep{dawson2004compatible,zhang2011maximum}. 
Here, we employ the first-order implicit DG formulation with our proposed limiter scheme to discretize equations
\eqref{Eqn:1DBL_general}--\eqref{Eqn:1DBL} in 
space, and backward Euler scheme is utilized in time. The DG penalty parameter is set to $\sigma=10^{-6}$ and
the flux $\mathbf{f}_{BL}$ is approximated with a first-order upwind method 
\citep{fambri2020discontinuous,zhang2018runge} as it provides good results in conjunction with proposed 
limiters. 
To implement the flux limiter, the following flux functional $\mathcal{H}_E(e)$ is adopted on each face
$e\subset\partial{E}$:  
\begin{align}
    \forall e = \partial E \cap \partial E', 
    \;\; &\mathcal{H}_{n+1,E}(e) = \int_e \{ \mathbf{f}_{BL}(S_{n+1}) \}\cdot \mathbf{n}_E +
    \frac{1}{2} \int_e \left\vert \frac{d\mathbf{f}_{BL}(S_{n})}{dS}\cdot\mathbf{n}_E\right\vert
    \Big(S_{n+1}|_E-S_{n+1}|_{E'}\Big),\nonumber\\
    \forall e \in \partial E \cap \Gamma^{\mathrm{D},s},\;\;
    &\mathcal{H}_{n+1,E}(e)=\int_e \mathbf{f}_{BL}(g^s)\cdot\mathbf{n}_E,\nonumber\\
    \forall e \in \partial E \cap \Gamma^{\mathrm{out},s},\;\;
    &\mathcal{H}_{n+1,E}(e)=\int_e \mathbf{f}_{BL}(S_{n+1})\cdot \mathbf{n}_{E} \nonumber.
\end{align}
The final simulation time is $T=800$ days, and the saturation profile is depicted in Figure~\ref{Fig:1DBL}
for $t=400$ and $t=800$ days. We performed a four-step mesh refinement study and linearly refined $\tau$ 
at each step. We observe that the location of the front obtained from the proposed numerical scheme
is in good agreement with the location of the front for the analytical solution even for the case of coarse mesh and as we proceed with
refinement, the discrete solution converges to the analytical solution. To calculate the semi-analytical solution of 
Buckley-Leverett equation (i.e., the position of the saturation front), we resorted to Welge graphical 
method \citep{welge1952simplified}. Figure~\ref{Fig:1DBL} also provides a zoom-in view at the location of 
front for $t=800$ days for better visualization.
As expected, the numerical saturation remains within physical bounds and no undershoots and 
overshoots are observed.
The choice of implicit time marching algorithm is shown to have no erroneous smearing effect on 
the saturation front.
\begin{figure}[]
    \centering
    \includegraphics[width=0.9\linewidth]{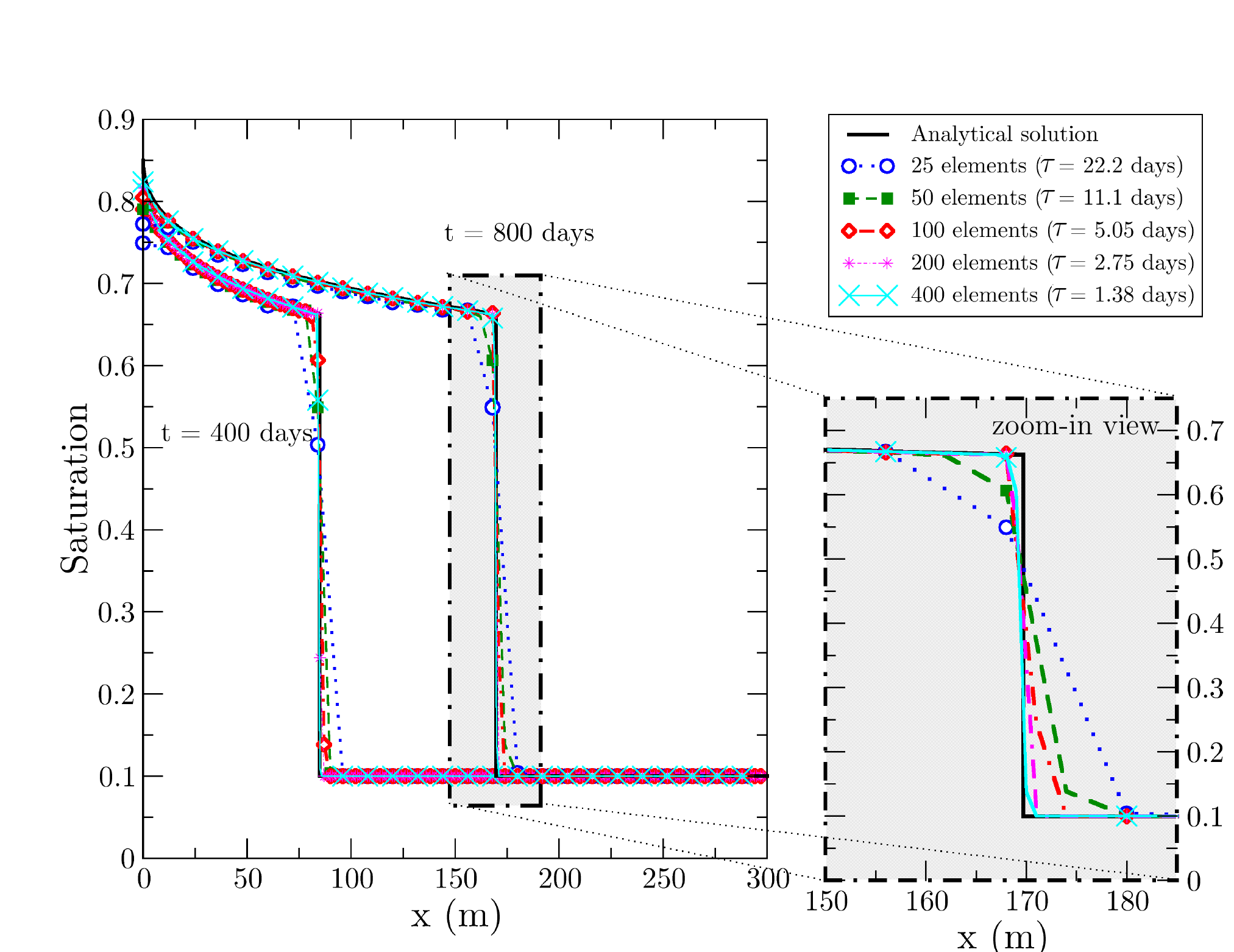}
    \caption{\textsf{One-dimensional Buckley-Leverett problem:}
    This figure shows the saturation profiles obtained from the limited DG scheme at two different time-steps
    $t=400$ days and  $t=800$ days.
    As we refine the mesh, the approximation converges to the analytical solution. 
    The numerical solution is satisfactory with respect to maximum principle as
    no undershoots and overshoots are observed.
\label{Fig:1DBL}}
\end{figure}

\subsubsection{Buckley-Leverett problem with gravity}%
\label{Sub:2D_BL}
In this numerical experiment, we study the effect of proposed limiters in the two-dimensional Buckley-Leverett 
equation that incorporates gravitational effects along the y-axis 
\citep{zhang2002adaptive,christov2008new,de2007construction}. 
Consider equation~\eqref{Eqn:1DBL_general} with the following non-convex flux functions in the x- and y- 
directions:
\begin{align}
    \label{Eqn:2DBL}
    F(S) = \frac{\lambda_w(S)u_t}{(\lambda_w(S)+\lambda_{\ell}(S))}, \quad 
    G(S) = \frac{F(S)v_t}{u_t}(1-5\lambda_{\ell}).
\end{align}
This benchmark problem was solved by finite element method 
combined with operator-splitting method in \citep{karlsen1998corrected}. For comparison purpose, hence,
we take $u_t=v_t=1$ \si{\meter\per\second}, $\lambda_w(S)=S^2$, $\lambda_{\ell}(S)=(1-S^2)$ and  $\phi=1$. 
We solve \eqref{Eqn:1DBL_general} and \eqref{Eqn:2DBL} on the square domain $[0,3]^2$ \si{\meter\squared} with structured 
triangular mesh of size $h=0.03$ \si{\meter} subject
to the initial condition:
\begin{align}
    \label{Eqn:2D_BL_in}
    s_0(x,y) = 
    \begin{cases}
        1,\quad \mbox{for} \, (x-1.5)^2 + (y-1.5)^2 <  0.5 \\
        0,\quad \mbox{otherwise.}
    \end{cases}
\end{align}
Finally, we impose no flow condition $\mathbf{u_t}\cdot \mathbf{n}=0$ everywhere on the boundary
$\partial \Omega$.
Similar to the one-dimensional problem, we use backward Euler time marching and discretize the problem with 
implicit DG formulation (with $\sigma=0.1$) augmented with the proposed flux and slope limiters scheme. 
Herein, the flux limiter functional $\mathcal{H}_{n+1,E}(e)$ on the interior edges is the same as that of the 
one-dimensional Buckly-Leverett and on all exterior edges 
it  is set to $0$.
The simulation runs to $T=0.5$ \si{\second} with $440$ time steps.
In Figure \ref{Fig:2D_BL_main}, we show the numerical results at the final time obtained from the implicit DG
formulation without limiters (see Figure \ref{Fig:2DBL_a}) and with limiters (see Figure \ref{Fig:2DBL_b}). We 
compare the results with the reference solution. Evidently,
DG scheme with no limiters produces an oscillatory solution that results in strong violations with respect to 
maximum principle. However, the application of limiters give rise to bound-preserving solution 
(i.e., $0\leq S_{n+1} \leq 1$); and undershoots and overshoots are eliminated completely. 
This result does not exhibit extra numerical diffusion and is consistent with the reference solution shown in 
Figure \ref{Fig:2DBL_c}.

\begin{figure}
    \subfigure[DG approximation\label{Fig:2DBL_a}]{
        \includegraphics[clip,scale=0.125,trim=0 0cm 0cm 0]{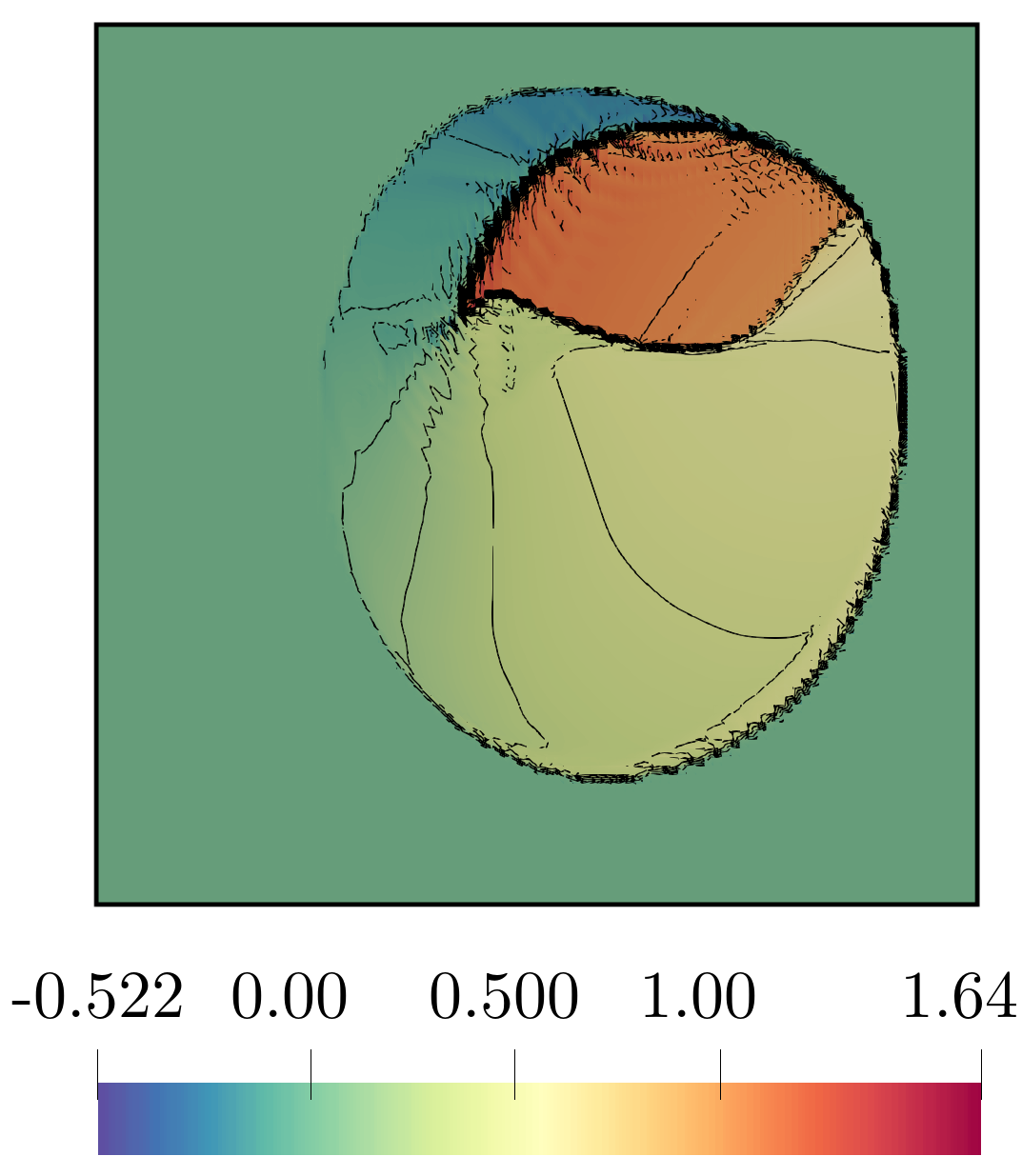}} 
        \hspace{0.1cm}
    \subfigure[DG approximation with flux and slope limiters\label{Fig:2DBL_b}]{
        \includegraphics[clip,scale=0.125,trim=0 0cm 0cm 0]{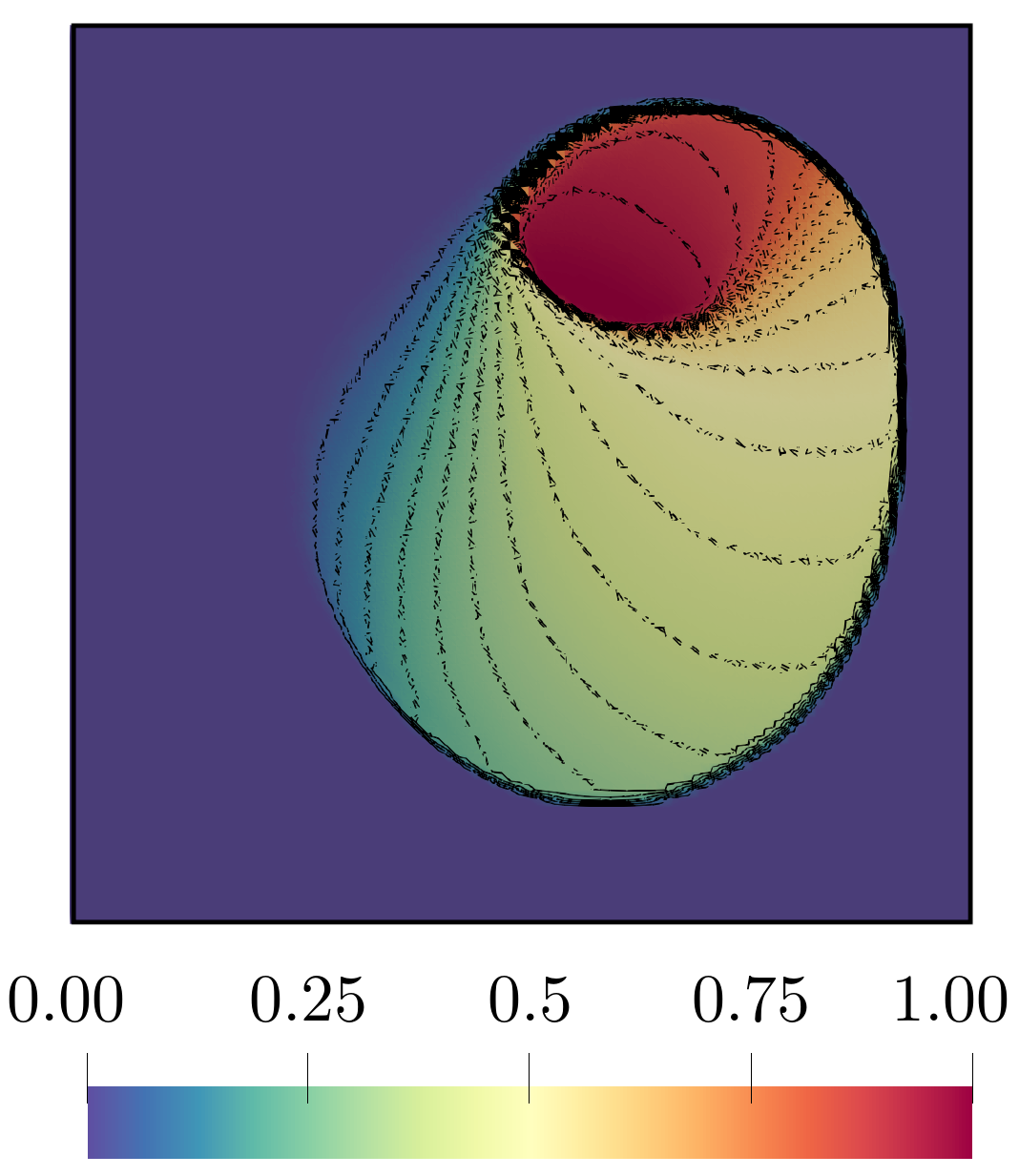}} 
        \hspace{0.15cm}
        \subfigure[Finite element with corrected operator splitting \citep{karlsen1998corrected}\label{Fig:2DBL_c}]{
        \includegraphics[clip,scale=0.16,trim=0cm 0.5cm 0cm 0cm]{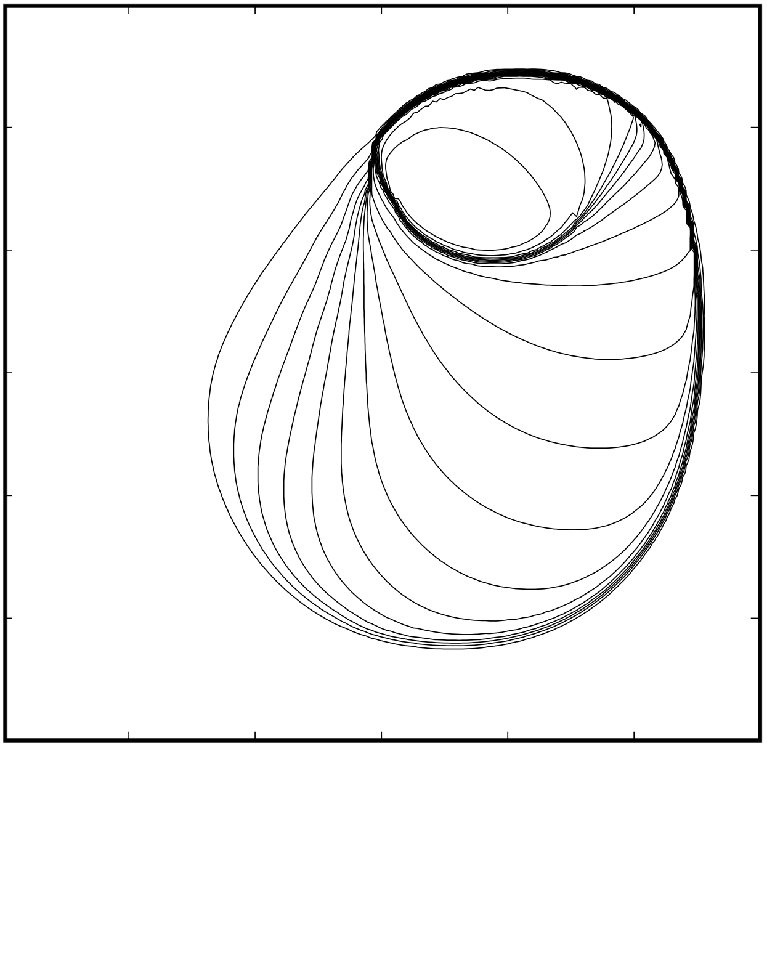}} 
        \caption{\textsf{Two-dimensional Buckley-Leverett problem with gravity subject to the initial
            condition \eqref{Eqn:2D_BL_in}:}
            This figure shows saturation approximations and contour plots 
            at $t=\SI{0.5}{\second}$. Solutions are obtained using the DG scheme without limiter (left) and with the proposed 
            the flux and slope limiters (middle); and are compared with a benchmark solution on fine mesh 
            (right).
            DG method without any bound-preserving mechanism fails to predict the correct profile and
            non-physical values are generated (i.e., $S_{n+1}\in[-0.52,+1.64]$). 
            However, the DG solution post-processed with limiters is accurate and  satisfies the
            maximum-principle.
        \label{Fig:2D_BL_main}}
\end{figure}
\subsubsection{Convergence study}
\label{subsec:Convergence}

We carry out an $h$-convergence study on two-dimensional structured triangular meshes 
in order to verify convergence properties of our limiter scheme.
The computational domain $\Omega$ is the unit square and the exact solutions are:
\begin{subequations}
    \label{Eqn:analytical}
\begin{align}
    \label{Eqn:analytical_sat}
&s(x,y,t) = 0.4+0.4xy+0.2\cos(t+x),\\
    \label{Eqn:analytical_pres}
&p(x,y,t) = 2+x^2y-y^2+x^2\sin(y+t)-\frac13 \cos(t) + \frac13 \cos(t+1) -\frac{11}{6}.
\end{align}
\end{subequations}
Through the method of manufactured solutions, we replace the source/sink terms (i.~e.,~wells flow rates) 
of equations \eqref{Eqn:BoM_p}--\eqref{Eqn:BoM_s} by body force terms obtained from 
manufactured solutions. 
Dirichlet boundary conditions are applied on $\partial\Omega$ on both saturation and 
pressure fields. The input parameters are:
\[
    \phi=0.2, \, K=\SI{1}{\meter\squared}, \, \mu_w=\mu_\ell=\SI{1}{\pascal\cdot\second},\,
    s_{rw}=s_{r\ell}=0, \,  k_{rw}(S)=S^2, \, k_{r\ell}(S)=(1-S)^2.
\]
The capillary pressure is based on Brooks-Corey model:
\begin{align}
    \label{Eqn:CapillaryPres}
    P_c(S)= 
    \begin{cases}
        p_d S^{\frac{-1}{\theta}} \quad \mbox{if} \; S>R\\
        p_d R^{\frac{-1}{\theta}}-\frac{p_d}{\theta}R^{-1-\frac{1}{\theta}}(S-R)\quad \mbox{otherwise},
    \end{cases}
\end{align}
where the entry pressure is set to $p_d=50$ \si{\pascal}, inhomogeneity characterization parameter is set to $\theta=2$, 
and linearization tolerance is set to $R=0.05$.
The convergence properties are computed by using a time step $\tau$ set to $h^2$.
We note that the admissible global bounds for the flux limiter algorithm are updated throughout the simulation.
In other words, at every time step, $s_*$ and $s^*$ bounds are determined by the maximum and minimum of the exact solution
\eqref{Eqn:analytical_sat}, respectively.
When using the limiters no upper and lower bound violations are observed in the discrete solution.
Table~\ref{Tab:Convergence_normal} shows the errors in $L^2$ and $H^1$ norms evaluated at $T=\SI{1}{\second}$
and the corresponding convergence rates for saturation and pressure. We compare rates for four cases:
(i) no limiters (DG), (ii) with both flux and slope limiters (DG+FL+SL), (iii) with only 
slope limiter (DG+SL), and (iv) with only flux limiter (DG+FL).
For both unknowns, DG returns expected optimal convergence rates of $2$ in the $L^2$ norm
and $1$ in the $H^1$ norm. 
However, we observe that applying both limiters results in suboptimal rates.
Cases (iii) and (iv) indicate that the application of flux limiters only preserve optimal rates whereas the application
of slope limiters only yields a decline in the convergence rates. 
The slope limiter scheme taken from~\cite{kuzmin2010vertex} 
is completely independent of the proposed flux limiter in Section~\ref{Sec:Numerical_method}. 
Designing a slope limiter that produces optimal rates remains a challenge. 
\begin{table}[]
    \caption{Errors in $L^2$ and $H^1$ norms and convergence rates, for $\tau = h^2$ and $T=1$ \si{\second}. 
Note that flux limiter algorithm preserves the  accuracy of the DG discretization. However,
slope limiter slightly degrades the rates of convergence.
\label{Tab:Convergence_normal}}
\centering
\resizebox{\textwidth}{!}{%
\begin{tabular}{lc@{\hskip 0.3in}cc@{\hskip 0.25in}cc@{\hskip 0.25in}cc@{\hskip 0.25in}cc}
\Xhline{2\arrayrulewidth}
 &
  $h$ &
  $||S_n-S(T)||_{L^2(\Omega)}$ &
  Rate &
  $||P_n-P(T)||_{L^2(\Omega)}$&
  Rate &
  $||S_n-S(T)||_{H^1(\Omega)}$ &
  Rate &
  $||P_n-P(T)||_{H^1(\Omega)}$&
  Rate \\ \hline
\\[-0.9em]
 &
  1/2 &
  $9.454\times10^{-4}$ &
  \cellcolor[HTML]{EFEFEF}$-$ &
  $7.607\times10^{-3}$ &
  \cellcolor[HTML]{EFEFEF}$-$ &
  $7.480\times10^{-3}$ &
  \cellcolor[HTML]{EFEFEF}$-$ &
  $7.325\times10^{-2}$ &
  \cellcolor[HTML]{EFEFEF}$-$ \\
 &
  1/4 &
  $5.373\times10^{-4}$ &
  \cellcolor[HTML]{EFEFEF}0.82 &
  $2.999\times10^{-3}$ &
  \cellcolor[HTML]{EFEFEF}1.34&
  $4.319\times10^{-3}$ &
  \cellcolor[HTML]{EFEFEF}0.79&
  $4.248\times10^{-2}$ &
  \cellcolor[HTML]{EFEFEF}0.79 \\
 &
  1/8 &
  $1.732\times10^{-4}$ &
  \cellcolor[HTML]{EFEFEF}1.63&
  $8.690\times10^{-4}$ &
  \cellcolor[HTML]{EFEFEF}1.79 &
  $1.836\times10^{-3}$ &
  \cellcolor[HTML]{EFEFEF}1.23&
  $2.322\times10^{-2}$ &
  \cellcolor[HTML]{EFEFEF}0.87\\
 &
  1/16 &
  $4.633\times10^{-5}$ &
  \cellcolor[HTML]{EFEFEF}1.90&
  $2.303\times10^{-4}$ &
  \cellcolor[HTML]{EFEFEF}1.92 &
  $8.374\times10^{-4}$ &
  \cellcolor[HTML]{EFEFEF}1.13&
  $1.245\times10^{-2}$ &
  \cellcolor[HTML]{EFEFEF}0.90 \\
\multirow{-5}{*}{{i.~DG}} &
  1/32 &
  \cellcolor[HTML]{EFEFEF}$1.176\times10^{-5}$ &
  \cellcolor[HTML]{C0C0C0}1.98 &
  \cellcolor[HTML]{EFEFEF}$5.912\times10^{-5}$ &
  \cellcolor[HTML]{C0C0C0}1.96&
  \cellcolor[HTML]{EFEFEF}$4.103\times10^{-4}$ &
  \cellcolor[HTML]{C0C0C0}1.03&
  \cellcolor[HTML]{EFEFEF}$6.492\times10^{-3}$ &
  \cellcolor[HTML]{C0C0C0}0.94 \\
\\[-0.9em]
  \hline
\\[-0.9em]
 &
  1/2 &
  $1.720\times10^{+0}$ &
  \cellcolor[HTML]{EFEFEF}$-$ &
  $3.180\times10^{-2}$ &
  \cellcolor[HTML]{EFEFEF}$-$ &
  $1.740\times10^{+0}$ &
  \cellcolor[HTML]{EFEFEF}$-$ &
  $2.110\times10^{-1}$ &
  \cellcolor[HTML]{EFEFEF}$-$ \\
 &
  1/4 &
  $7.620\times10^{-3}$ &
  \cellcolor[HTML]{EFEFEF}7.82 &
  $2.870\times10^{-3}$ &
  \cellcolor[HTML]{EFEFEF}3.47 &
  $1.270\times10^{-1}$ &
  \cellcolor[HTML]{EFEFEF}3.77&
  $4.210\times10^{-2}$ &
  \cellcolor[HTML]{EFEFEF}2.32\\
 &
  1/8 &
  $2.650\times10^{-3}$ &
  \cellcolor[HTML]{EFEFEF}1.53 &
  $8.130\times10^{-4}$ &
  \cellcolor[HTML]{EFEFEF}1.82 &
  $8.260\times10^{-2}$ &
  \cellcolor[HTML]{EFEFEF}0.63 &
  $2.320\times10^{-2}$ &
  \cellcolor[HTML]{EFEFEF}0.86\\
 &
  1/16 &
  $9.190\times10^{-4}$ &
  \cellcolor[HTML]{EFEFEF}1.53 &
  $2.130\times10^{-4}$ &
  \cellcolor[HTML]{EFEFEF}1.93 &
  $5.540\times10^{-2}$ &
  \cellcolor[HTML]{EFEFEF}0.58 &
  $1.250\times10^{-2}$ &
  \cellcolor[HTML]{EFEFEF}0.89\\
\multirow{-5}{*}{{ii.~DG+FL+SL}} &
  1/32 &
  \cellcolor[HTML]{EFEFEF}$3.260\times10^{-4}$ &
  \cellcolor[HTML]{C0C0C0}1.50 &
  \cellcolor[HTML]{EFEFEF}$6.030\times10^{-5}$ &
  \cellcolor[HTML]{C0C0C0}1.82&
  \cellcolor[HTML]{EFEFEF}$3.920\times10^{-2}$ &
  \cellcolor[HTML]{C0C0C0}0.50 &
  \cellcolor[HTML]{EFEFEF}$6.550\times10^{-3}$ &
  \cellcolor[HTML]{C0C0C0}0.93\\
\\[-0.9em]
  \hline
\\[-0.9em]
 &
  1/2 &
  $2.570\times10^{-2}$ &
  \cellcolor[HTML]{EFEFEF}$-$ &
  $7.330\times10^{-3}$ &
  \cellcolor[HTML]{EFEFEF}$-$ &
  $2.300\times10^{-1}$ &
  \cellcolor[HTML]{EFEFEF}$-$ &
  $7.160\times10^{-2}$ &
  \cellcolor[HTML]{EFEFEF}$-$ \\
 &
  1/4 &
  $7.620\times10^{-3}$ &
  \cellcolor[HTML]{EFEFEF}1.75&
  $2.870\times10^{-3}$ &
  \cellcolor[HTML]{EFEFEF}1.35&
  $1.270\times10^{-1}$ &
  \cellcolor[HTML]{EFEFEF}0.85 &
  $4.210\times10^{-2}$ & \cellcolor[HTML]{EFEFEF}0.77 \\
 &
  1/8 &
  $2.650\times10^{-3}$ &
  \cellcolor[HTML]{EFEFEF}1.53 &
  $8.130\times10^{-4}$ &
  \cellcolor[HTML]{EFEFEF}1.82 &
  $8.260\times10^{-2}$ &
  \cellcolor[HTML]{EFEFEF}0.63 &
  $2.320\times10^{-2}$ &
  \cellcolor[HTML]{EFEFEF}0.86\\
 &
  1/16 &
  $9.190\times10^{-4}$ &
  \cellcolor[HTML]{EFEFEF}1.53 &
  $2.130\times10^{-4}$ &
  \cellcolor[HTML]{EFEFEF}1.93 &
  $5.540\times10^{-2}$ &
  \cellcolor[HTML]{EFEFEF}0.58 &
  $1.250\times10^{-2}$ & \cellcolor[HTML]{EFEFEF}0.89\\
\multirow{-5}{*}{{iii.~DG+SL}} &
  1/32 &
  \cellcolor[HTML]{EFEFEF}$3.260\times10^{-4}$ &
  \cellcolor[HTML]{C0C0C0}1.50 &
  \cellcolor[HTML]{EFEFEF}$6.030\times10^{-5}$ &
  \cellcolor[HTML]{C0C0C0}1.82&
  \cellcolor[HTML]{EFEFEF}$3.920\times10^{-2}$ &
  \cellcolor[HTML]{C0C0C0}0.50 &
  \cellcolor[HTML]{EFEFEF}$6.550\times10^{-3}$ &
  \cellcolor[HTML]{C0C0C0}0.93\\ 
\\[-0.9em]
  \hline
\\[-0.9em]
 &
  1/2 &
  $1.720\times10^{+0}$ &
  \cellcolor[HTML]{EFEFEF}$-$ &
  $3.210\times10^{-2}$ &
  \cellcolor[HTML]{EFEFEF}$-$ &
  $1.720\times10^{+0}$ &
  \cellcolor[HTML]{EFEFEF}$-$ &
  $2.120\times10^{-1}$ &
  \cellcolor[HTML]{EFEFEF}$-$ \\
 &
  1/4 &
  $5.370\times10^{-4}$ &
  \cellcolor[HTML]{EFEFEF}11.64&
  $3.000\times10^{-3}$ &
  \cellcolor[HTML]{EFEFEF}3.42 &
  $4.320\times10^{-3}$ &
  \cellcolor[HTML]{EFEFEF}8.64 &
  $4.250\times10^{-2}$ &
  \cellcolor[HTML]{EFEFEF}2.32\\
 &
  1/8 &
  $1.730\times10^{-4}$ &
  \cellcolor[HTML]{EFEFEF}1.63&
  $8.690\times10^{-4}$ &
  \cellcolor[HTML]{EFEFEF}1.79 &
  $1.840\times10^{-3}$ &
  \cellcolor[HTML]{EFEFEF}1.23&
  $2.320\times10^{-2}$ &
  \cellcolor[HTML]{EFEFEF}0.87\\
 &
  1/16 &
  $4.630\times10^{-5}$ &
  \cellcolor[HTML]{EFEFEF}1.90&
  $2.300\times10^{-4}$ &
  \cellcolor[HTML]{EFEFEF}1.92 &
  $8.370\times10^{-4}$ &
  \cellcolor[HTML]{EFEFEF}1.13&
  $1.250\times10^{-2}$ &
  \cellcolor[HTML]{EFEFEF}0.90 \\
\multirow{-5}{*}{{iv.~DG+FL}} &
  1/32 &
  \cellcolor[HTML]{EFEFEF}$1.180\times10^{-5}$ &
  \cellcolor[HTML]{C0C0C0}1.98 &
  \cellcolor[HTML]{EFEFEF}$5.910\times10^{-5}$ &
  \cellcolor[HTML]{C0C0C0}1.96&
  \cellcolor[HTML]{EFEFEF}$4.100\times10^{-4}$ &
  \cellcolor[HTML]{C0C0C0}1.03 &
  \cellcolor[HTML]{EFEFEF}$6.490\times10^{-3}$ &
  \cellcolor[HTML]{C0C0C0}0.94 \\
\\[-0.9em]
\Xhline{2\arrayrulewidth}
\end{tabular}%
}
\end{table}
We show in Table~\ref{Tab:Convergence_avg} the errors in the $L^2$ norm of the cell average for the saturation and the
corresponding convergence rates. Optimal rate of $2$ is obtained for either DG or DG+FL+SL. 
This result  reiterates that the flux limiter does not reduce the accuracy and the slope limiter does not impact
the rates since it does not alter the element-wise averages.
%

\begin{table}[]
    \caption{Errors and rates for the cell average values of saturation $\bar{S}$. 
The time step $\tau$ is set to $h^2$ and $L^2$ norms are computed at the final time $T=1$.
     DG approximation with limiters return optimal convergence rate with respect to average values.
    \label{Tab:Convergence_avg}}
\centering
\resizebox{7cm}{!}{%
\begin{tabular}{cc@{\hskip 0.3in}cc}
\Xhline{2\arrayrulewidth}
\\[-0.95em]
 & $h$    & \multicolumn{1}{c}{$||\bar{S}_{n}-\bar{S}(T)||_{L^2(\Omega)}$} & \multicolumn{1}{c}{Rate}      
 \\ \\[-0.95em]\hline\\[-0.9em]
 & 1/2  & $5.900\times10^{-4}$                  & \cellcolor[HTML]{EFEFEF}$-$    \\
 & 1/4  & $4.839\times10^{-4}$                  & \cellcolor[HTML]{EFEFEF}0.286  \\
 & 1/8  & $1.661\times10^{-4}$                  & \cellcolor[HTML]{EFEFEF}1.543  \\
 & 1/16 & $4.506\times10^{-5}$                  & \cellcolor[HTML]{EFEFEF}1.882  \\
    \multirow{-5}{*}{{DG}}       & 1/32 & \cellcolor[HTML]{EFEFEF}$1.148\times10^{-5}$ & \cellcolor[HTML]{C0C0C0}1.973 
    \\  \\[-0.9em]\hline \\[-0.9em]
                                    & 1/2  & $1.721\times10^{+0}$                  & \cellcolor[HTML]{EFEFEF}$-$    \\
                                    & 1/4  & $4.861\times10^{-4}$                  & \cellcolor[HTML]{EFEFEF}11.790 \\
                                    & 1/8  & $1.658\times10^{-4}$                  & \cellcolor[HTML]{EFEFEF}1.552  \\
                                    & 1/16 & $4.467\times10^{-5}$                  & \cellcolor[HTML]{EFEFEF}1.892  \\
    \multirow{-5}{*}{{DG+FL+SL}} & 1/32 & \cellcolor[HTML]{EFEFEF}$1.142\times10^{-5}$ & \cellcolor[HTML]{C0C0C0}1.967 \\[-0.9em]\\ 
\Xhline{2\arrayrulewidth}
\end{tabular}%
}
\end{table}

\subsection{Two-dimensional pressure-driven flow}%
\label{sub:two_dimensional_patch_test}
We take a computational domain of $\Omega=[0,100]^2$ m\textsuperscript{2} with zero gravity field for all
problems in this section.
The wetting phase is injected along the left boundary and the non-wetting phase is pushed out 
through the right boundary. Dirichlet boundary conditions are:  $P=3\times10^6$ \si{\pascal} and $S=0.85$ on 
$\{0\}\times(0,100)$; and $P=10^{6}$ \si{\pascal} on $\{100\}\times(0,100)$ \si{\meter}. 
Outflow boundary condition is chosen for saturation on the right boundary and remaining boundaries are set as 
no-flow ($j^s=j^p=0$). The pictorial descriptions of the pressure-driven flow problem are provided in Figure \ref{Fig:2Dpatch_BVP}.

\subsubsection{Homogeneous domain}
\label{sub:2Dpatch_homogen}
A homogeneous test problem with constant permeability of $K=10^{-8}$ \si{\meter\squared} is examined here, 
with similar setup and parameters as in the work of \citet{epshteyn2007fully}. 
Relative permeabilities and capillary pressure are defined in equations \eqref{Eqn:Rel-perm} and 
\eqref{Eqn:CapillaryPres}, respectively, with entry pressure $p_d = 1000$ \si{\pascal}, $\theta=2$ and
$R=0.05$.  
The viscosities are $\mu_w= 10^{-3}$ \si{\pascal\cdot\second} and $\mu_{\ell}=\;$\SI{e-2}{\pascal\cdot\second}.
\begin{figure}
    \subfigure[Boundary condition for pressure~\label{Fig:BC_pressure}]{
        \includegraphics[clip,scale=0.25,trim=0 0cm 0cm
    0]{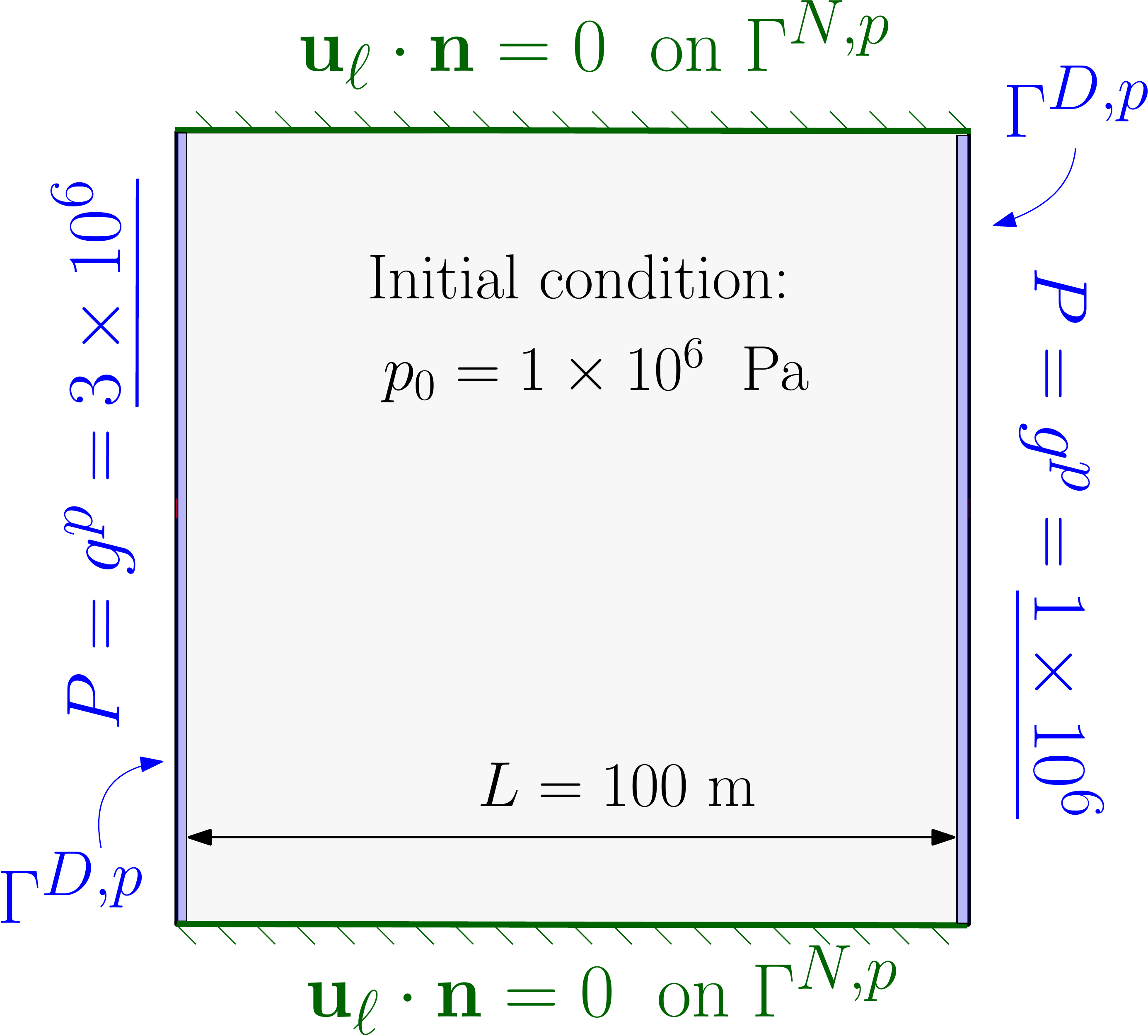}}
        \hspace{.45cm}
    \subfigure[Boundary condition for saturation~\label{fig:BC_saturation}]{
        \includegraphics[clip,scale=0.25,trim=0 0cm 0cm 0]{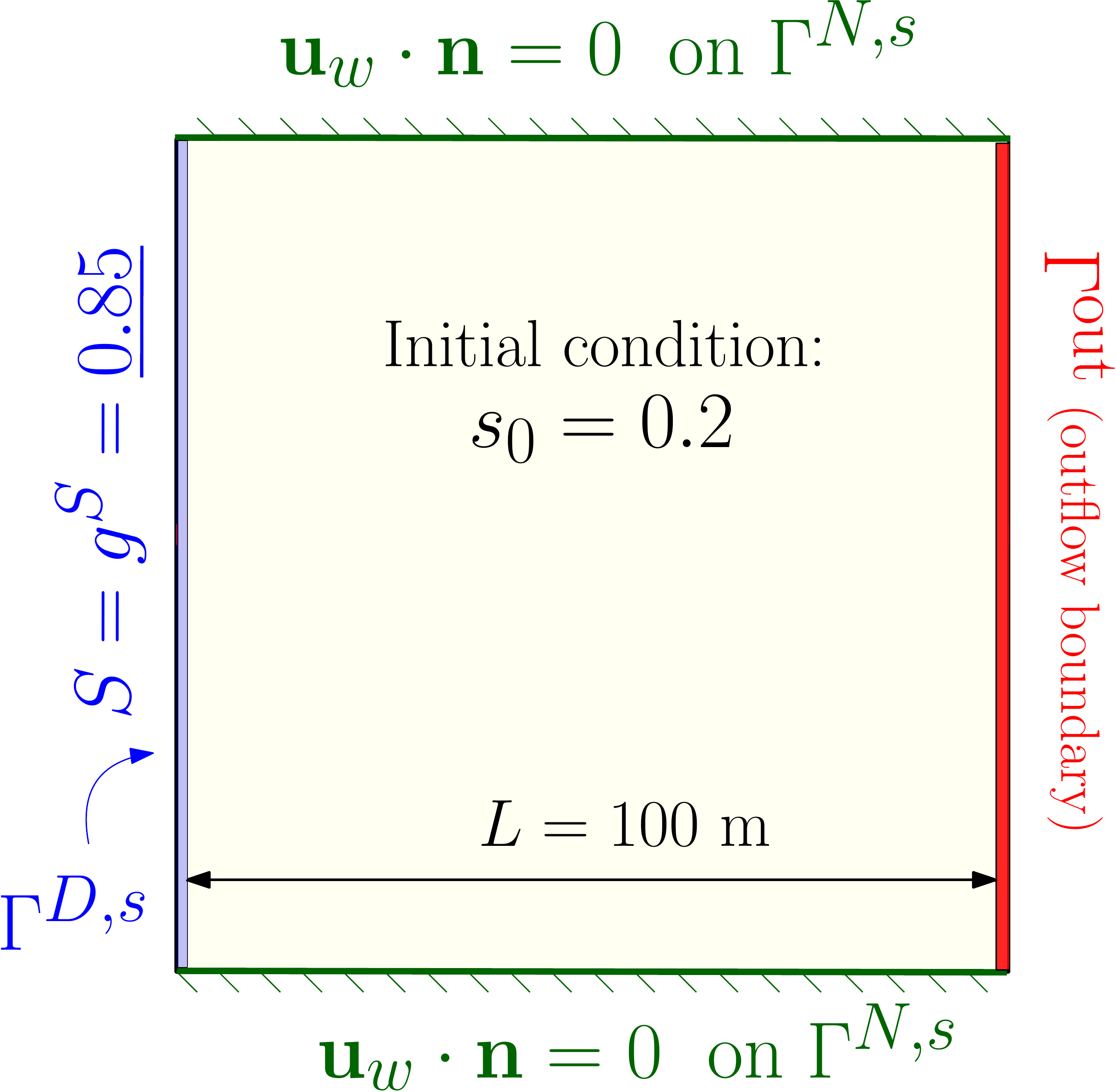}} \\
        \caption{\textsf{Two-dimensional pressure-driven flow problem:}
            This figure provides a pictorial description of the computational domain and boundary value 
            problem.
        \label{Fig:2Dpatch_BVP}}
\end{figure}
Two quadrilateral meshes are considered: 
(i) a uniform mesh with size of $h=1.25$ m and 
(ii) a non-uniform mesh with $256$ elements and with size of $h_{\mathrm{bnd}}=1.25$ m at the left
boundary and $h=6.583$ m for the rest of domain (see Figure \ref{Fig:Q4_mesh}).  
It is known that slope limiters by design flatten steep slopes near discontinuities 
(e.g., at left-most elements when simulation starts).
Using a mesh with increased density at the (left) boundary  reduces the effect of overflattening 
on the accuracy of solutions \citep{may2013two,giuliani2018analysis}.
\begin{figure}
    \subfigure[Quadrilateral mesh \label{Fig:Q4_mesh}]{
        \includegraphics[clip,scale=0.18,trim=0 0cm 0cm
        0]{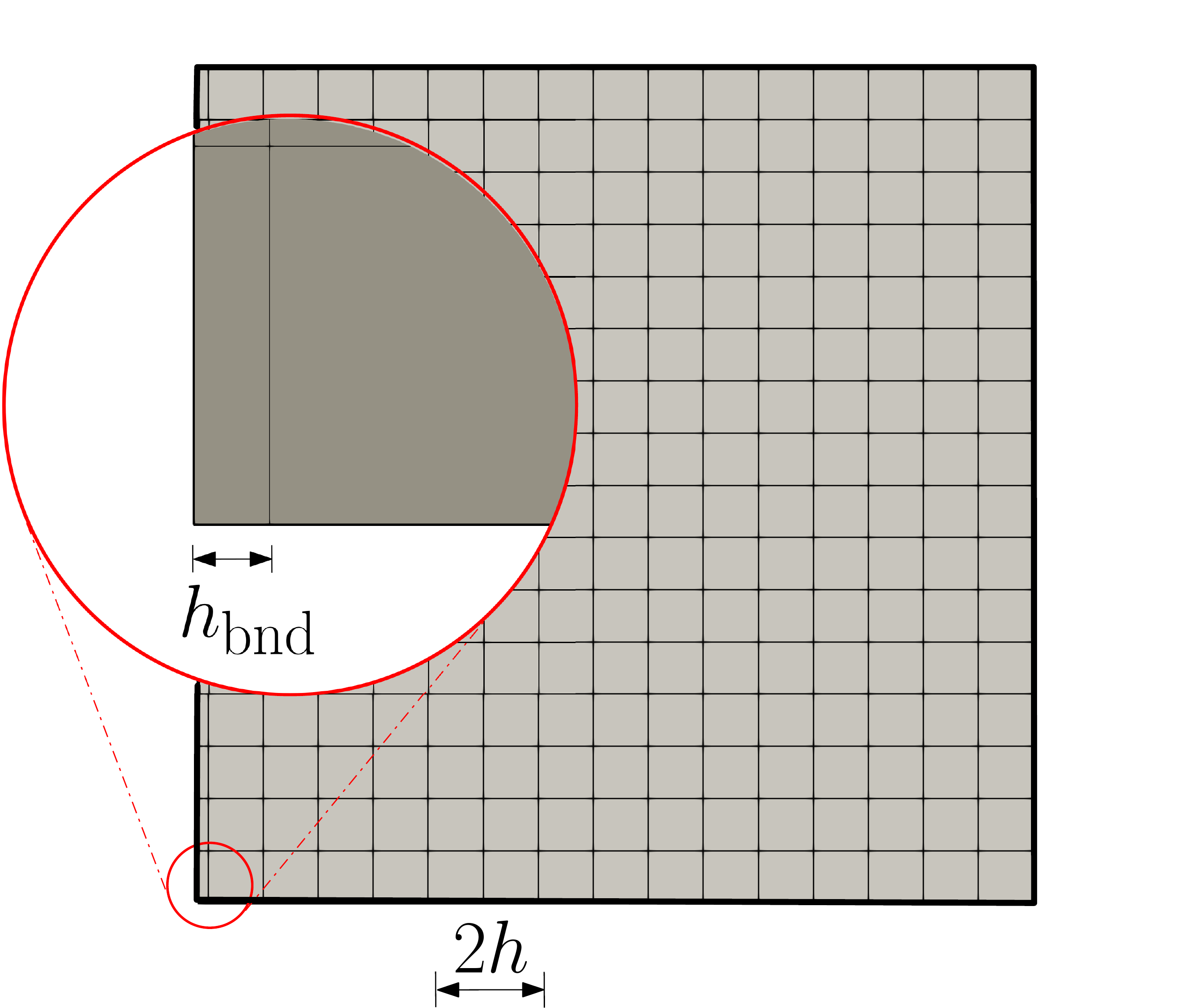}}
        \hspace{.25cm}
    \subfigure[Crossed triangle mesh \label{Fig:Crossed_mesh}]{
        \includegraphics[clip,scale=0.18,trim=0 0cm 0cm 0]{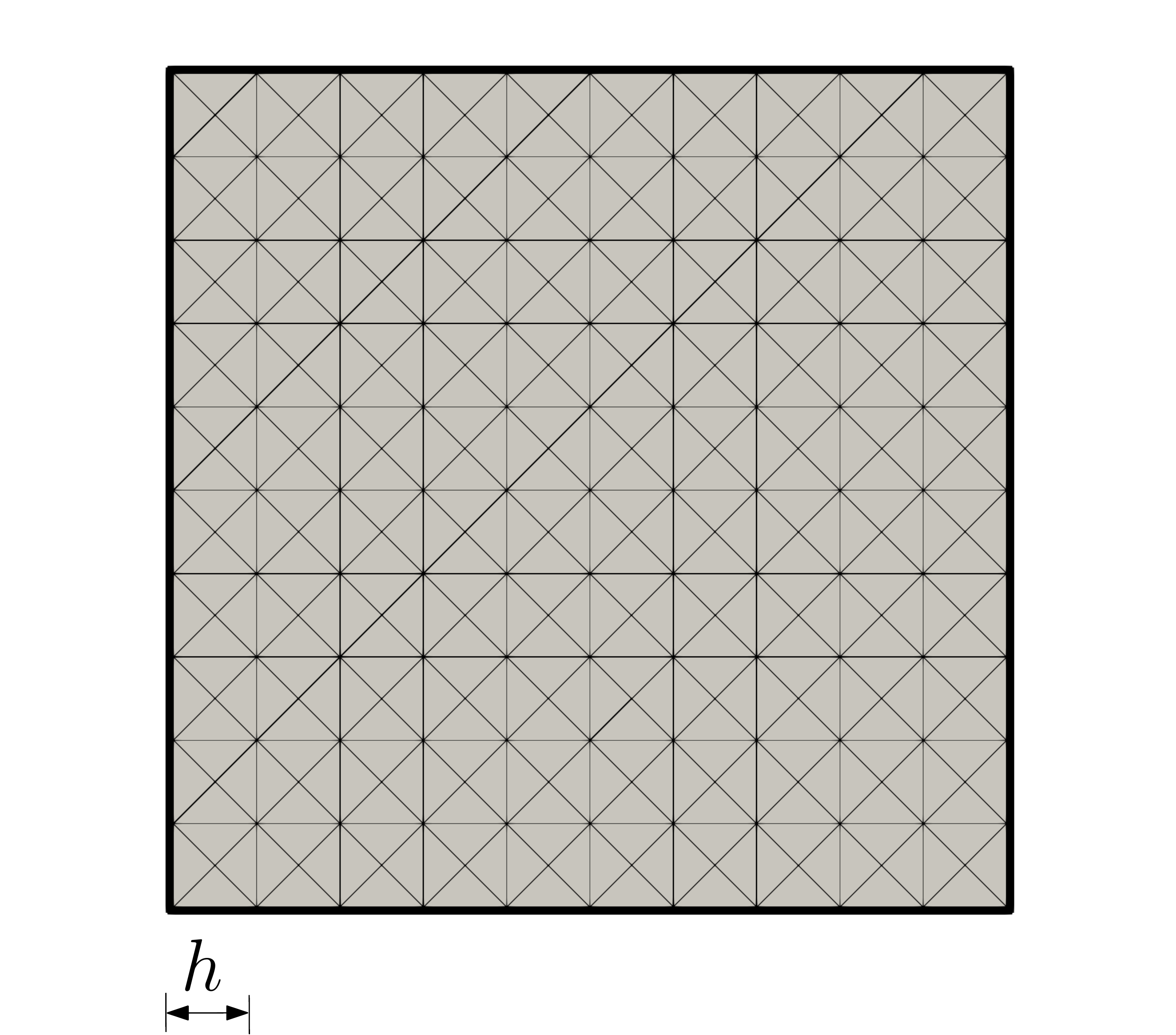}} 
        \hspace{.25cm}
    \subfigure[Mesh for thin-barrier problem \label{Fig:Barrier_mesh}]{
        \includegraphics[clip,scale=0.16,trim=0 0cm 0cm 0]{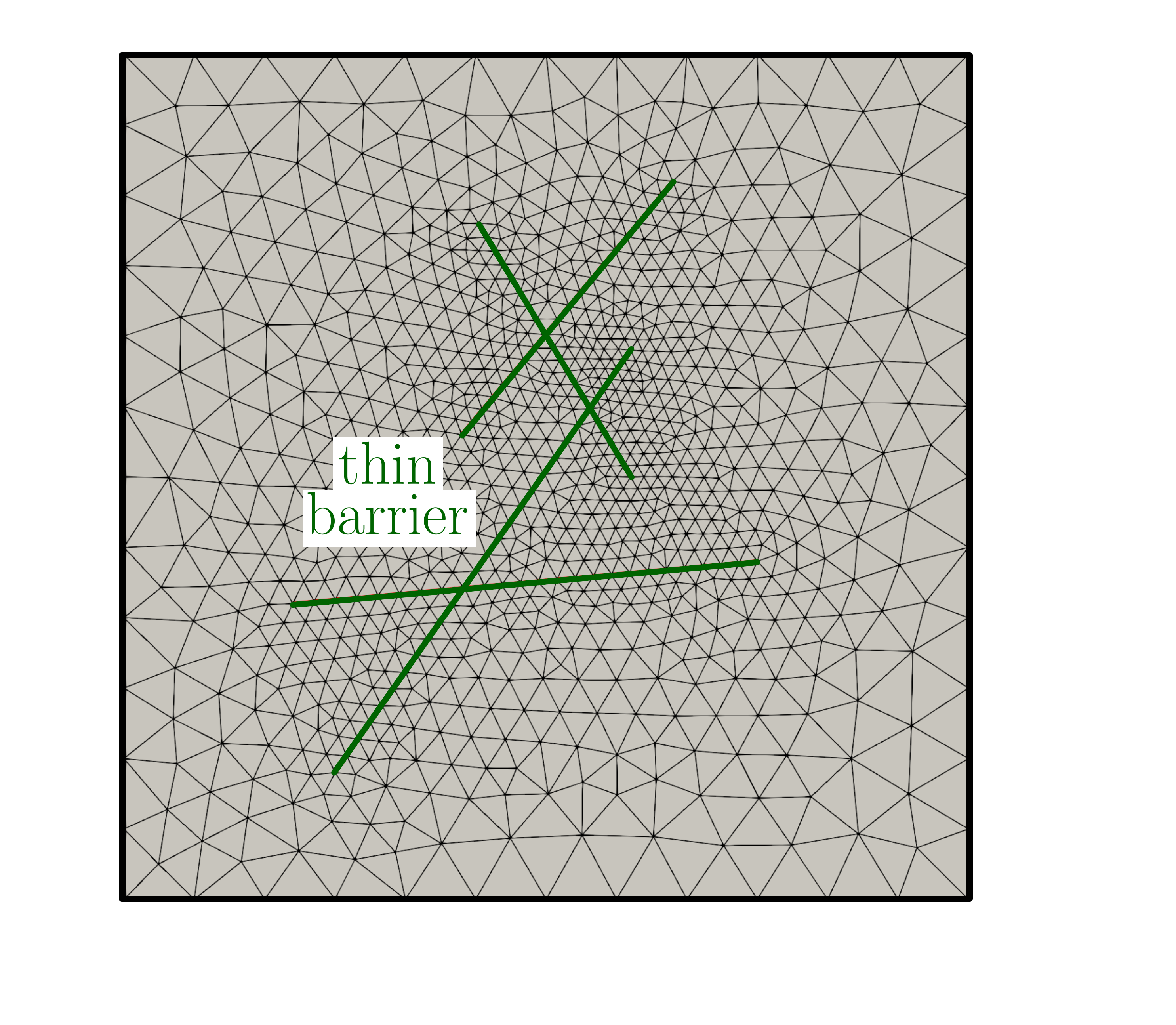}} 
        \caption{\textsf{Two-dimensional pressure-driven flow problem:}
            This figure shows the typical meshes employed in our numerical simulations.
        \label{Fig:2Dpatch_meshes}}
\end{figure}
%
The time step is chosen as $\tau = 0.2$ \si{\second}, the final time is $T=300$ \si{\second}, and the penalty parameter is $\sigma = 100$.
%
We compare our numerical  solutions with a reference unlimited solution obtained from the fully implicit DG
formulation developed by \citet{epshteyn2007fully} on a quadrilateral mesh with $256$ elements.
The saturation and pressure profiles obtained with our proposed scheme, 
along the line $y=50$ \si{\meter} are illustrated in Figures \ref{Fig:2Dpatch_sat} and \ref{Fig:2Dpatch_pres}.
Numerical solutions, compared to reference solution, are accurate and in good agreement with respect to 
front location.
As expected, the finer mesh tracks the saturation front with more accuracy.
It is also evident that our limiting scheme successfully yields pointwise bound-preserving and monotone solutions.
However, the reference solution unsurprisingly 
violates undershoot bound (about 4\% right after the saturation front) and produces an oscillatory saturation 
profile. 
\begin{figure}
    \subfigure[Saturation profile \label{Fig:2Dpatch_sat}]{
        \includegraphics[clip,scale=0.33,trim=0 1.2cm 7cm 0]{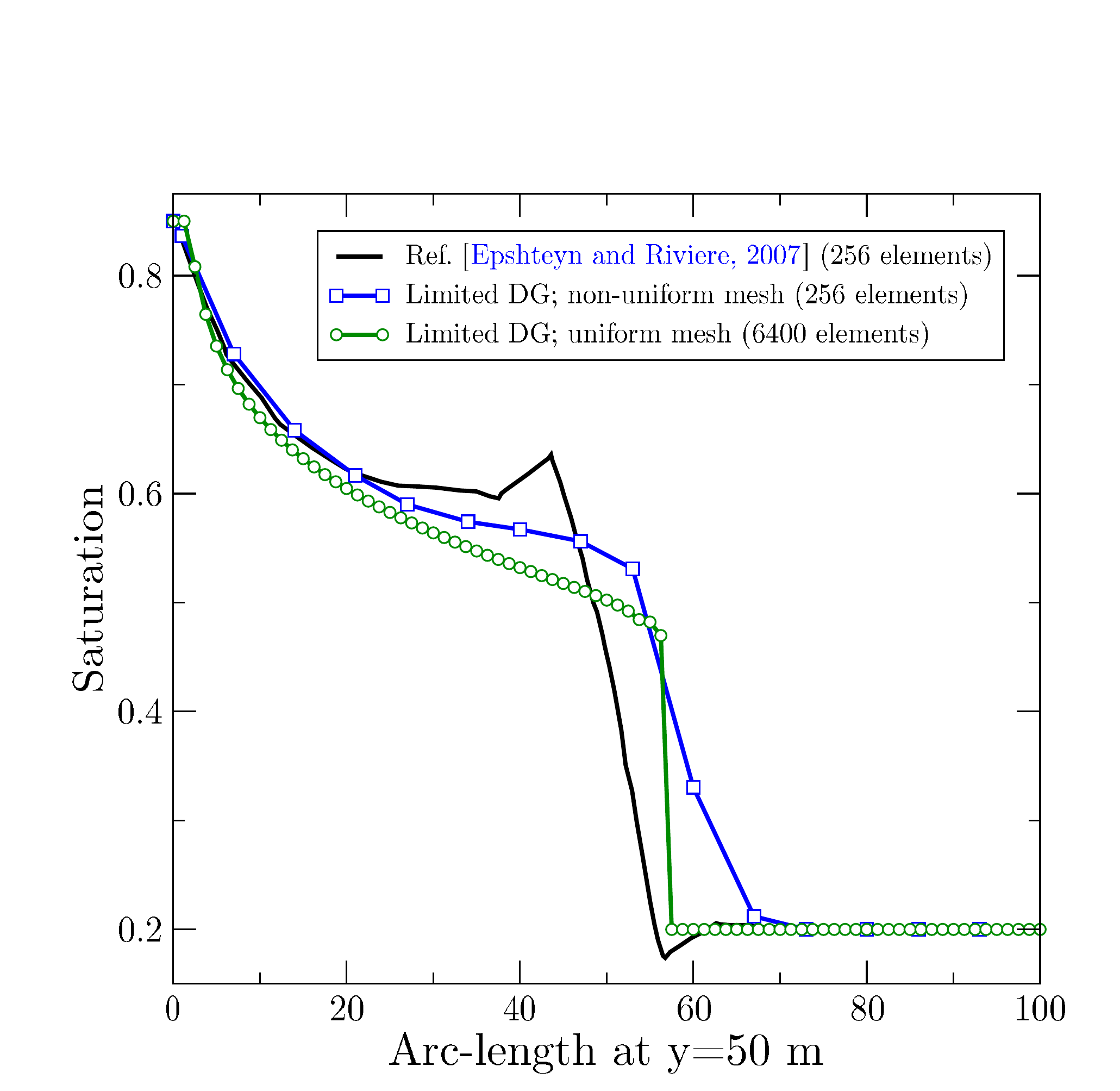}} 
        \hspace{.05cm}
    \subfigure[Pressure profile \label{Fig:2Dpatch_pres}]{
        \includegraphics[clip,scale=0.32,trim=0.25cm 1.2cm 0cm
    0cm]{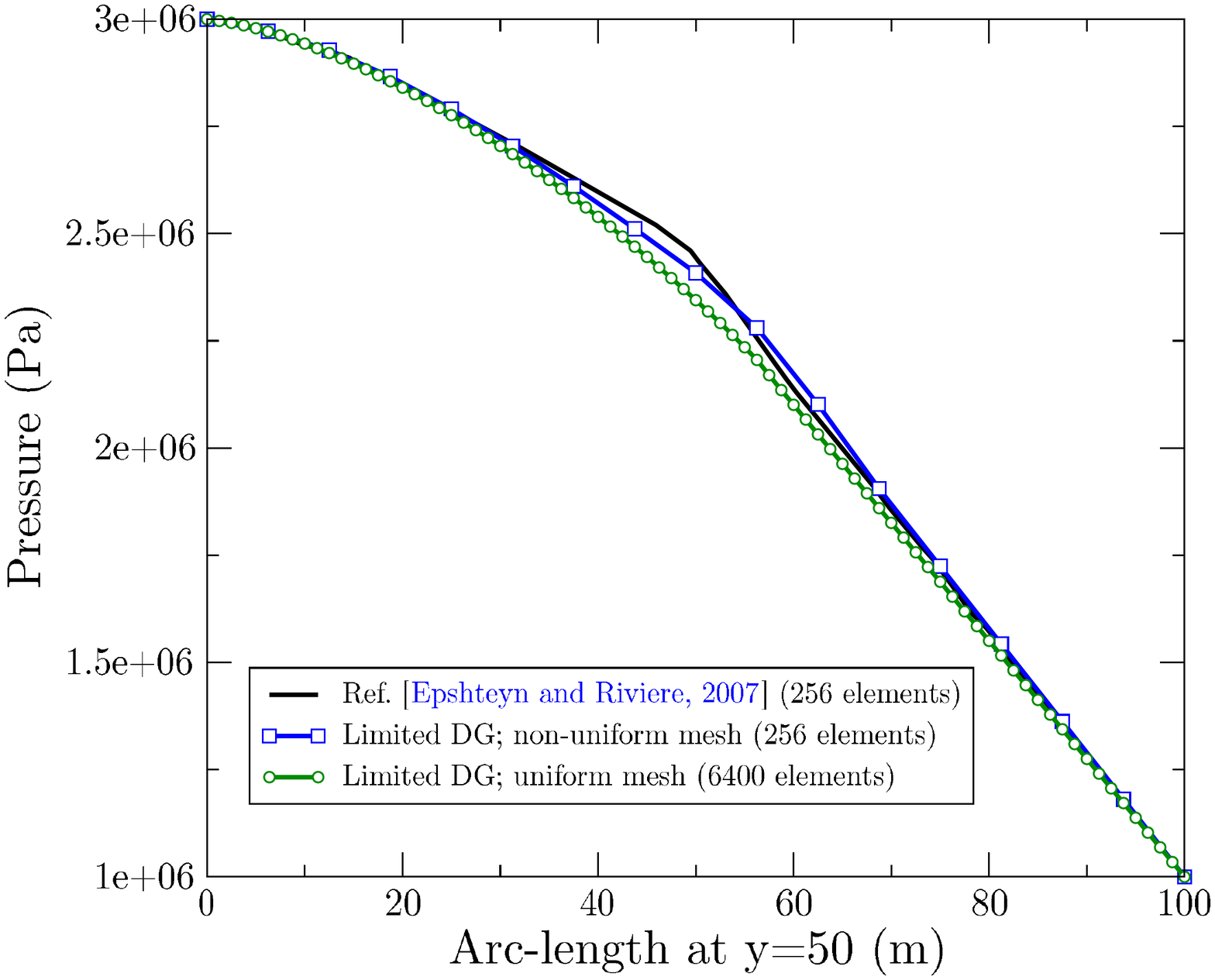}}
        \caption{\textsf{Two-dimensional pressure-driven flow in homogeneous domain:}
            This figure exhibits the saturation and pressure profiles obtained
            from limited DG approximations (with
            $\mathbb{P}=1$ and $\sigma=100$). Solutions on uniform and non-uniform meshes 
            are plotted along the line $y=50$ \si{\meter} at $t=300$ \si{\second} 
            and are compared with a reference DG solution.
            Regardless of the mesh size, limiters completely suppress unphysical overshoots and undershoots
            and accurately predict the location of saturation front
            which is in good agreement with that of the reference solution. 
            As expected, the uniform finer mesh gives rise to a sharper front.  
            On the other hand, the reference solution is not equipped with any bound-preserving mechanism 
            and thus does not enjoy maximum principle and lower bound violations ($S<0.2$) and non-monotone 
            behavior are captured.
        \label{Fig:2Dpatch_sat_pres}}
\end{figure}

%
To better understand the efficacy of the proposed limiting algorithm (i.e., DG+SL+FL), 
and distinguish it from the vertex-based slope limiter of \citet{kuzmin2010vertex} (i.e., DG+SL), 
we solve the problem again (with same parameters as before) on a crossed structured mesh (shown in 
figure \ref{Fig:Crossed_mesh}) for total duration of $T=450$ \si{\second}.
The initial size of $h=10$ \si{\meter} is chosen for this analysis and four-step refinement is performed.
Table \ref{Tab:2D_patch_compare} reports the performance of limiters and compare them
with respect to bound-preserving properties, local mass balance violations, and monotonocity.
We observe that mesh refinement reduces maximum undershoots of unlimited DG from $47.61$ \% to $29.21$ \% and 
maximum overshoots to less than $0.1$ \% but does not eliminate violations. The application of slope limiter to DG
eliminates undershoots at all time steps and significantly reduces maximum overshoots to $4.79$ \% 
for the coarsest mesh and to $3.32$ \% for the finest mesh. 
It can be seen that DG+SL falls short to satisfy maximum principle even under excessive mesh refinement. 
Further, it should be noted that both DG and DG+SL approximations fail to obtain monotone solutions near 
the saturation front. This means that they are susceptible to local spurious oscillations near the front 
even when the global bounds are not violated.
However, approximations under DG+FL+SL enjoy pointwise maximum principle and the saturation field remains 
monotone over the entire domain, independently of the mesh size.
\begin{table}[]
    \caption{This table shows the efficacy of the limiters when applied to the pressure-driven flow problem 
        with homogeneous domain. Simulations are carried out for the duration of $450$ \si{\second} on crossed 
        mesh (see Figure \ref{Fig:Crossed_mesh}) for different mesh-sizes.
        In this table, max $\mathcal{M}$ denotes the maximum magnitude value of local mass balance error observed for all 
        time steps.
    \label{Tab:2D_patch_compare}}
\centering
\resizebox{\textwidth}{!}{%
    \begin{tabular}{cc@{\hskip 0.4in}cc@{\hskip 0.4in}cc@{\hskip 0.4in}c@{\hskip 0.3in}c}
\Xhline{2\arrayrulewidth}
\begin{tabular}[c]{@{}c@{}}Mesh-size\\ (m)\end{tabular} &
  Algorithm 
   &
  \begin{tabular}[c]{@{}c@{}}max\\ Undershoot\end{tabular} &
  \begin{tabular}[c]{@{}c@{}}max\\ Undershoot (\%)\end{tabular} &
  \begin{tabular}[c]{@{}c@{}}max\\ Overshoot\end{tabular} &
  \begin{tabular}[c]{@{}c@{}}max\\ Overshoot (\%)\end{tabular} &
  \begin{tabular}[c]{@{}c@{}}max \\ $\mathcal{M}$ \end{tabular}&
  Monotonocity \\ \hline
  \\[-.9em]
 &
  DG &
  {\color[HTML]{000000} {\color{red}-0.109}} &
  \cellcolor[HTML]{EFEFEF}{\color[HTML]{000000} {\color{red}47.61} } &
  {\color[HTML]{000000} {\color{red}0.854}} &
  \cellcolor[HTML]{EFEFEF}{\color[HTML]{000000} {\color{red}0.56} } &
  $1.37\times10^{-15}$ &
  \xmark \\
 &
  DG+SL &
  {\color[HTML]{000000} {\color{red}0.169}} &
  \cellcolor[HTML]{EFEFEF}{\color[HTML]{000000} {\color{red}4.79} } &
  {\color[HTML]{000000} 0.85} &
  \cellcolor[HTML]{EFEFEF}{\color[HTML]{000000} 0} &
  $2.23\times10^{-9}$ &
  \xmark   \\
\multirow{-3}{*}{$h=10$} &
  DG+FL+SL &
  {\color[HTML]{000000} 0.2} &
  \cellcolor[HTML]{EFEFEF}{\color[HTML]{000000} 0} &
  {\color[HTML]{000000} 0.85} &
  \cellcolor[HTML]{EFEFEF}{\color[HTML]{000000} 0} &
   $7.66\times10^{-12}$&
   \cmark  \\ \\[-0.95em]\hline
  \\[-.95em]
 &
  DG &
  {\color{red}-0.093} &
  \cellcolor[HTML]{EFEFEF}{\color{red}45.13} &
  {\color{red}0.852} &
  \cellcolor[HTML]{EFEFEF}{\color{red}0.28} &
  $5.41\times10^{-15}$ &
  \xmark   \\
 &
  DG+SL &
  {\color{red}0.169} &
  \cellcolor[HTML]{EFEFEF}{\color{red}4.78} &
  0.85 &
  \cellcolor[HTML]{EFEFEF}0 &
  $3.24\times10^{-9}$ &
  \xmark   \\
\multirow{-3}{*}{$h=5$} &
  DG+FL+SL &
  0.2 &
  \cellcolor[HTML]{EFEFEF}0 &
  0.85 &
  \cellcolor[HTML]{EFEFEF}0 &
  $7.03\times10^{-11}$ &
  \cmark   \\
  \\[-0.95em]\hline
  \\[-.95em]
 &
  DG &
  {\color{red}-0.059} &
  \cellcolor[HTML]{EFEFEF} {\color{red}40} &
  {\color{red}0.851}&
  \cellcolor[HTML]{EFEFEF}{\color{red}0.136} &
  $2.19\times10^{-14}$ &
  \xmark   \\
 &
  DG+SL &
  {\color{red}0.172}&
  \cellcolor[HTML]{EFEFEF}{\color{red}4.29} &
  0.85 &
  \cellcolor[HTML]{EFEFEF}0 &
  $2.57\times10^{-8}$ &
  \xmark   \\
\multirow{-3}{*}{$h=2.5$} &
  DG+FL+SL &
  0.2 &
  \cellcolor[HTML]{EFEFEF}0 &
  0.85 &
  \cellcolor[HTML]{EFEFEF}0 &
  $2.22\times10^{-14}$ &
  \cmark    \\ 
  \\[-0.9 em]\hline\\[-0.9 em]
 &
  DG &
  {\color{red}0.010} &
  \cellcolor[HTML]{EFEFEF}{\color{red}29.21} &
  {\color{red}0.8502} &
  \cellcolor[HTML]{EFEFEF}{\color{red}0.07} &
  $9.78\times10^{-14}$ &
  \xmark   \\
 &
  DG+SL &
  {\color{red}0.178} &
  \cellcolor[HTML]{EFEFEF}{\color{red}3.32} &
  0.85 &
  \cellcolor[HTML]{EFEFEF}0 &
  $2.57\times10^{-7}$&
  \xmark   \\
\multirow{-3}{*}{$h=1.25$} &
  DG+FL+SL &
  0.2 &
  \cellcolor[HTML]{EFEFEF}0 &
  0.85 &
  \cellcolor[HTML]{EFEFEF}0 &
  $1.08\times10^{-13}$ &
  \cmark    \\ 
  \\[-0.95em]
\Xhline{2\arrayrulewidth}
\end{tabular}%
}
\end{table}

As shown in Table~\ref{Tab:Solver_performance}, only $3$ to $4$ Newton's iterations are needed at each time 
step for convergence of either limited DG or unlimited DG approximation. 
This means that the limiters do not have a significant effect on the number of solver iterations.  
However, as we refine the mesh, flux limiter algorithm $\mathcal{L}_\mathrm{avg}$ requires more iterations to converge.
\begin{table}[]
    \caption{Nonlinear Newton iterations and flux limiter iterations per time step.  {\label{Tab:Solver_performance}}}
\centering
\resizebox{0.7\textwidth}{!}{%
\begin{tabular}{cc|cc}
\hline
\multirow{2}{*}{\begin{tabular}[c]{@{}c@{}}Mesh-size\\ (m)\end{tabular}} & DG                  & \multicolumn{2}{c}{DG+FL+SL}            \\ \cline{2-4} 
& Newton's iter. & Newton's iter. & flux limiter iter. \\ \hline
$10$   & $3-4$ & $3-4$ & $1-25$  \\
$5$    & $3-4$ & $3-4$ & $1-55$  \\
$2.5$  & $3-4$ & $3-4$ & $1-116$ \\
$1.25$ & $3-4$ & $3-4$ & $1-238$ \\ \hline
\end{tabular}%
}
\end{table}
\subsubsection{Local mass balance}
DG methods are known for their local mass conservation properties. 
\citep{riviere2008discontinuous,joshaghani2019stabilized}.
In this section,  we investigate the effect of the proposed limiters on altering local mass conservation properties.
Upon applying element-wise averages and choosing unit test function in  \eqref{Eqn_disc2}, we obtain 
the local mass conservation of an element $E\in \mathcal{E}_h$ at time $t_n$:
\begin{align}
    \label{Eqn:incomp_BoM}
    \mathcal{M}(E)=
    \frac{\phi(\bar{S}_{n+1}\vert_E-\bar{S}_{n}\vert_E)}{\tau}
    + \frac{1}{|E|} \sum_{e \subset \partial E}  
    \mathcal{H}_{n+1,E}(e)
     -
     \left( f_w(s_{\mathrm{in}}) \bar{q}_E - \overline{f_w(S_{n}\vert_{E})} \underline{q}_E \right).
\end{align}
We compute the magnitude of mass balance error for the pressure-driven flow problem discussed in the previous section. 
Table \ref{Tab:2D_patch_compare} contains the value of maximum error observed throughout the simulation. 
Evidently, DG+SL scheme is slightly worse than other two schemes with respect to errors, which is consistent 
for all mesh-sizes. However, values are all very small and below than the solver tolerance.
In Figure \ref{Fig:2Dpatch_LMB}, the values of $\mathcal{M}(E)$ are displayed at $t=450$ \si{\second} on a crossed mesh of size
$h=2.5$ m for three cases of DG, DG+SL, and DG+FL+SL. One can see that applying slope limiter
(without flux limiter) instigates an erroneous patch (shown with dark brown color in Figure
\ref{Fig:2Dpatch_LMB_DG+SL}).
It is also clear that the proposed numerical scheme (i.e., DG+FL+SL) is locally mass conservative and 
slightly outperforms DG+SL scheme.
\begin{figure}
    \subfigure[DG with no limiter]{
        \includegraphics[clip,scale=0.13,trim=0 0cm 0cm 0]{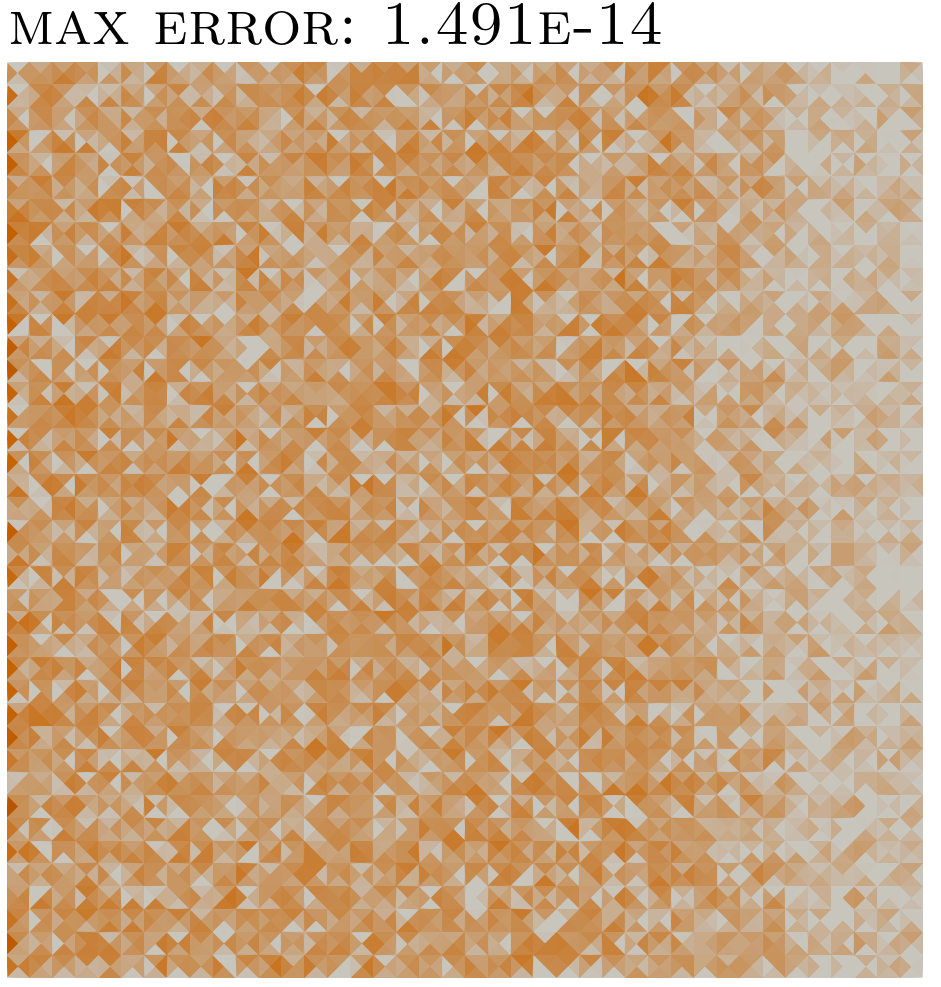}}
        \hspace{0.1cm}
        \subfigure[DG+SL \label{Fig:2Dpatch_LMB_DG+SL}]{
        \includegraphics[clip,scale=0.13,trim=0 0cm 0cm 0]{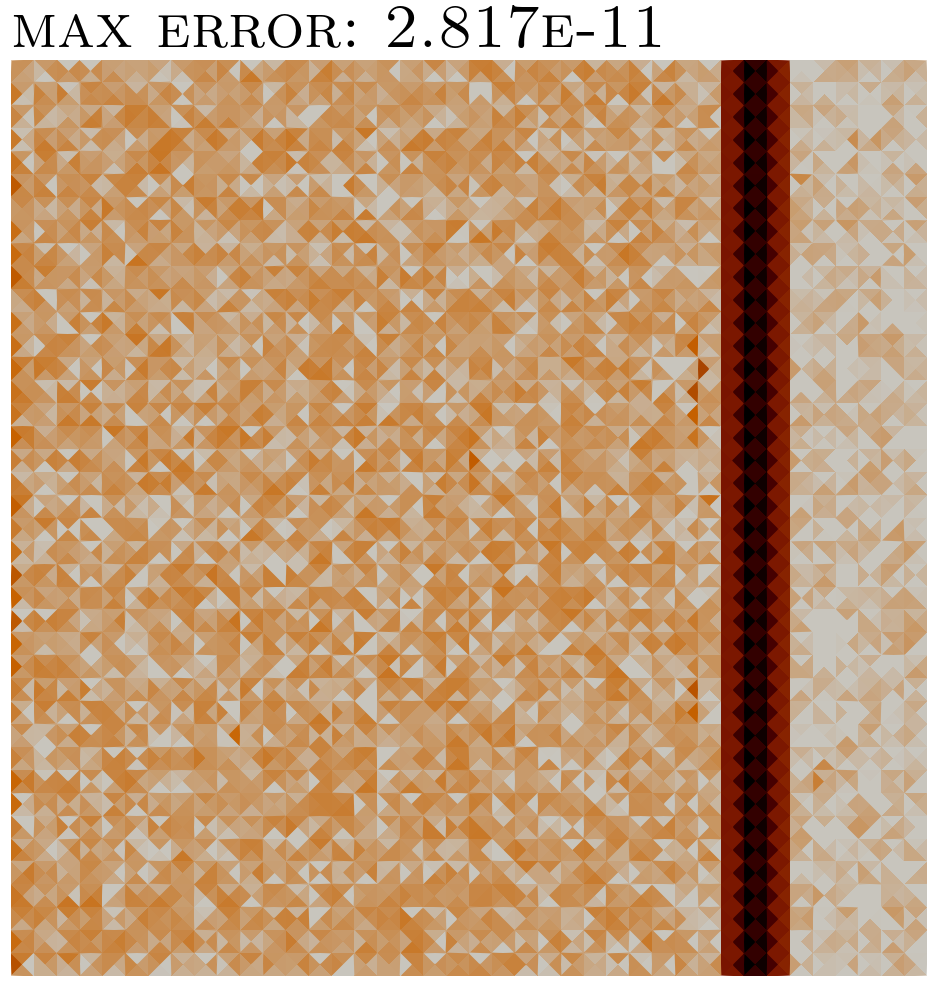}} 
        \hspace{0.1cm}
    \subfigure[DG+FL+SL]{
        \includegraphics[clip,scale=0.13,trim=0 0cm 0cm 0]{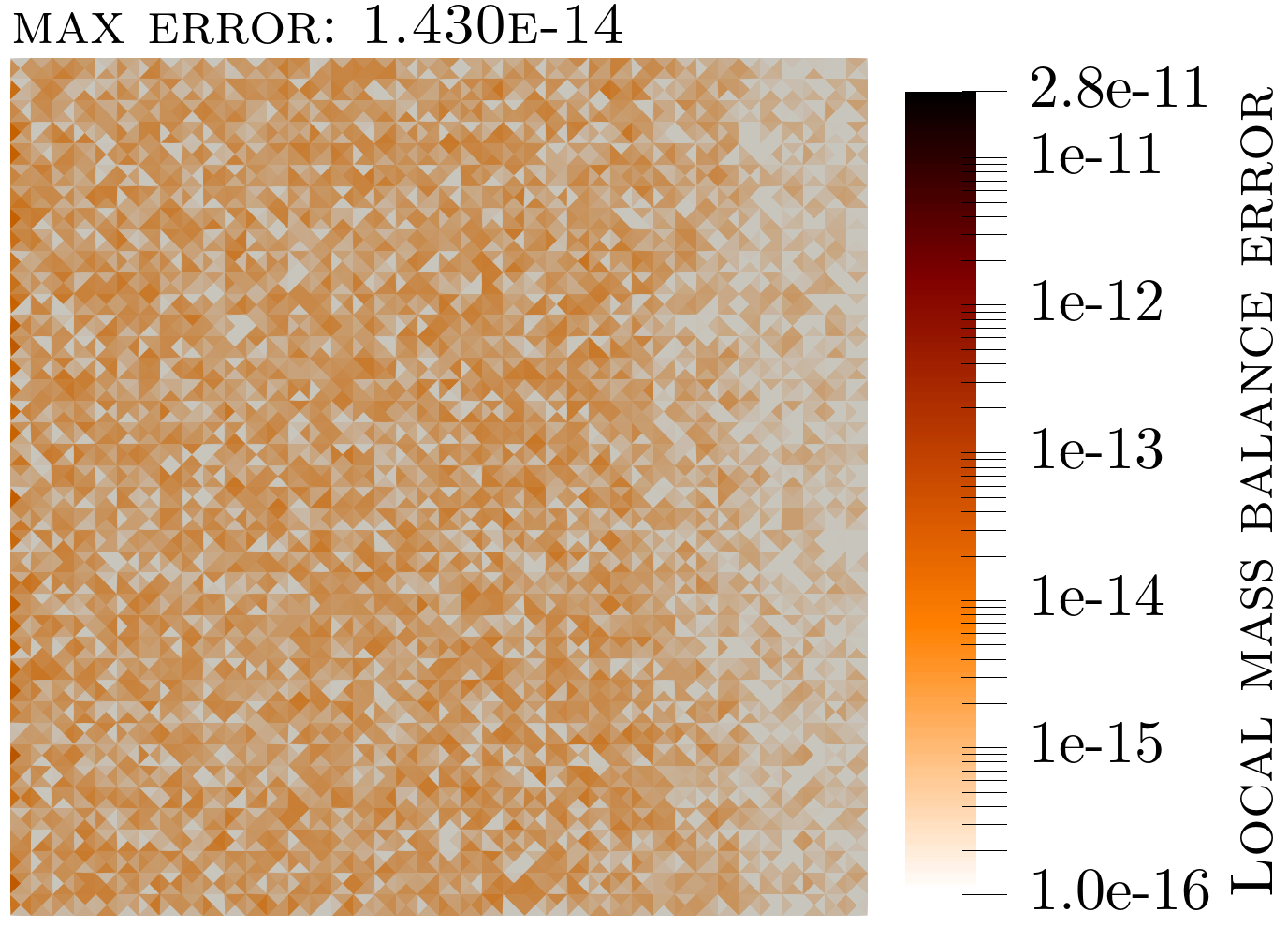}} 
        \caption{\textsf{Local mass balance conservation for pressure-driven flow problem:}
                This figure shows the local mass conservation properties of the limited and unlimited DG
                approximations on a homogeneous domain with $h=2.5$ \si{\meter} at $t=450$ \si{\second}. 
                DG+SL near the front induces slight increase in mass balance error but,
                overall errors remain small for all three cases.
        \label{Fig:2Dpatch_LMB}}
\end{figure}

\subsubsection{Domain with thin barrier}
In this example, the porous medium contains a thin barrier and it is partitioned into 
an unstructured triangular mesh (see Figure \ref{Fig:Barrier_mesh}). Total time is set to $T=4500$ \si{\second} and the time step is $t=0.5$ \si{\second}. 
Additionally, noflow boundary conditions are imposed on the barrier edges. All other parameters are the 
same as in Section \ref{sub:2Dpatch_homogen}.
Figure \ref{Fig:2Dpatch_barriers_sat} exhibits the saturation profile under limited and unlimited DG 
at three different time steps. Limited DG, unlike its unlimited version, generates saturation that 
remains bounded and neither undershoots (blue-colored cells) nor overshoots (red-colored cells) are detected during 
the simulation. Nonetheless, the saturation front, under both unlimited and limited DG, propagates with the same 
speed and tends to avoids the barrier as expected.
\begin{figure}
    \subfigure[DG with no limiter; $t=250$ \si{\second}~\label{Fig:}]{
        \includegraphics[clip,scale=0.15,trim=0 0cm 0cm
    0]{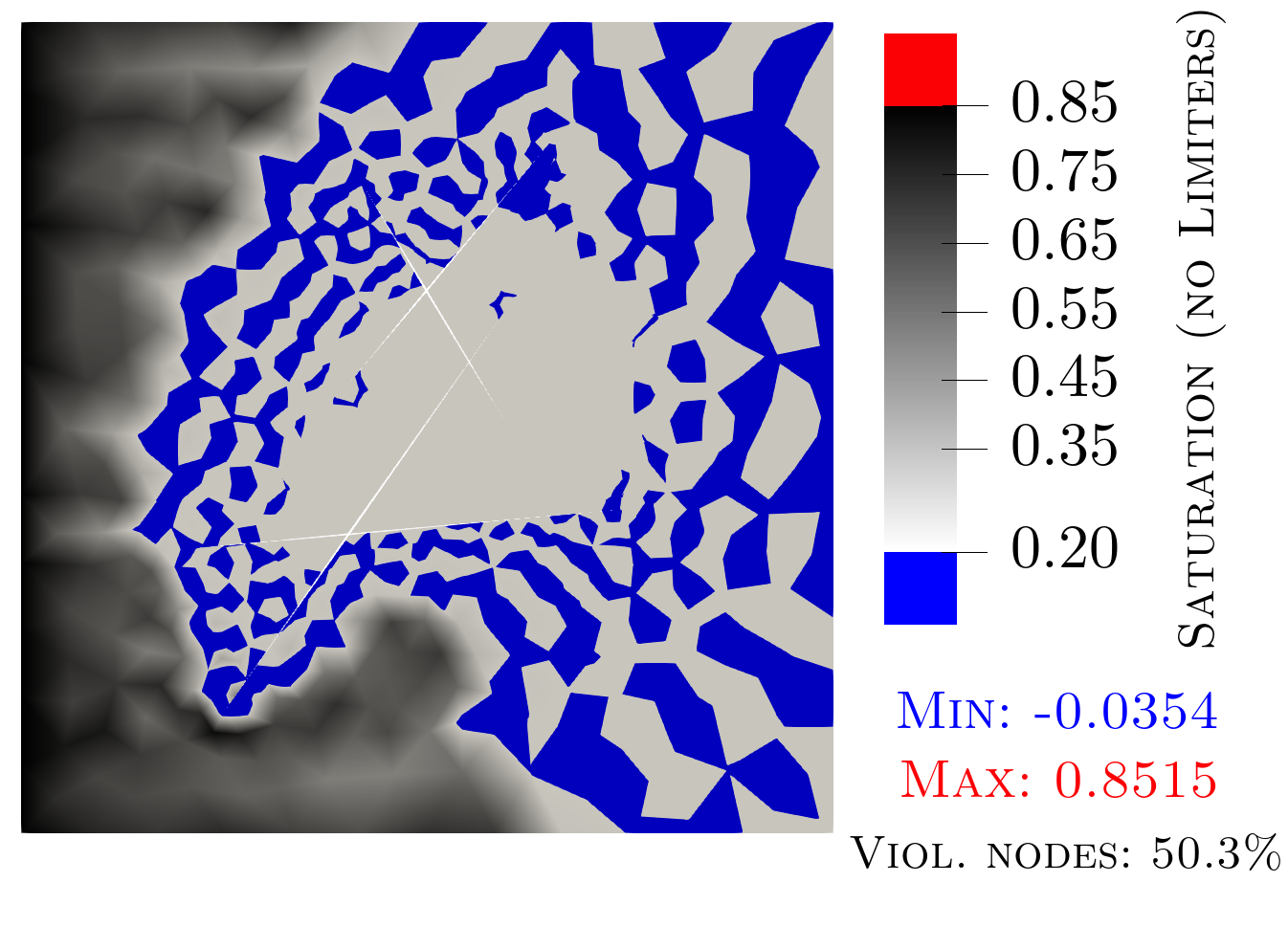}}
        \hspace{.25cm}
        \subfigure[DG+FL+SL; $250$ \si{\second}~\label{fig:}]{
        \includegraphics[clip,scale=0.15,trim=0 0cm 0cm 0]{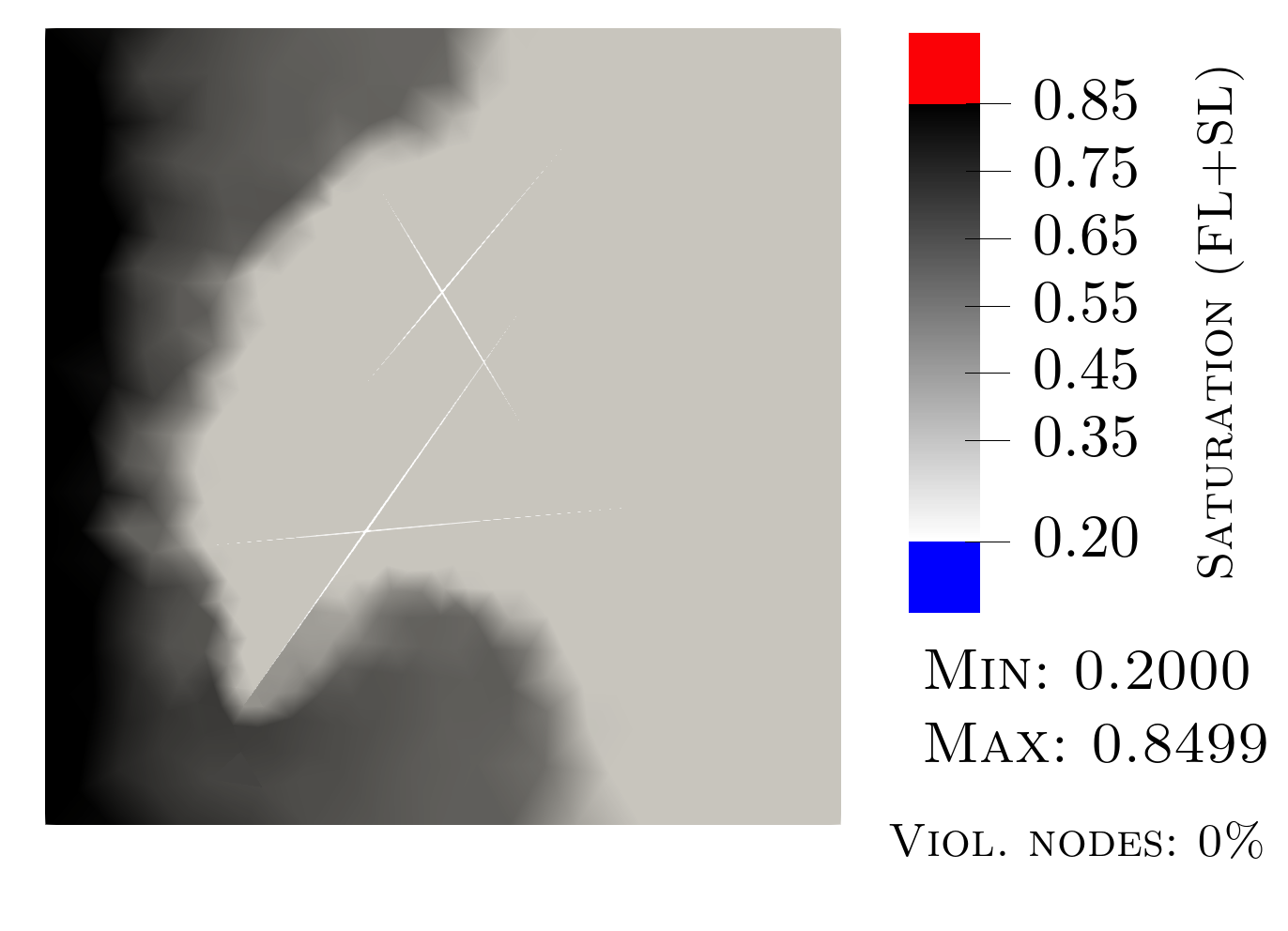}} \\
        \subfigure[DG with no limiters; $t=1000$ \si{\second}~\label{fig:}]{
        \includegraphics[clip,scale=0.15,trim=0 0cm 0cm 0]{Figures/Fig8c.png}} 
        \hspace{.25cm}
        \subfigure[DG+FL+SL; $t=1000$ \si{\second}~\label{fig:}]{
        \includegraphics[clip,scale=0.15,trim=0 0cm 0cm 0]{Figures/Fig8d.png}} 
        \subfigure[DG with no limiters; $t=4500$ \si{\second}~\label{fig:}]{
        \includegraphics[clip,scale=0.15,trim=0 0cm 0cm 0]{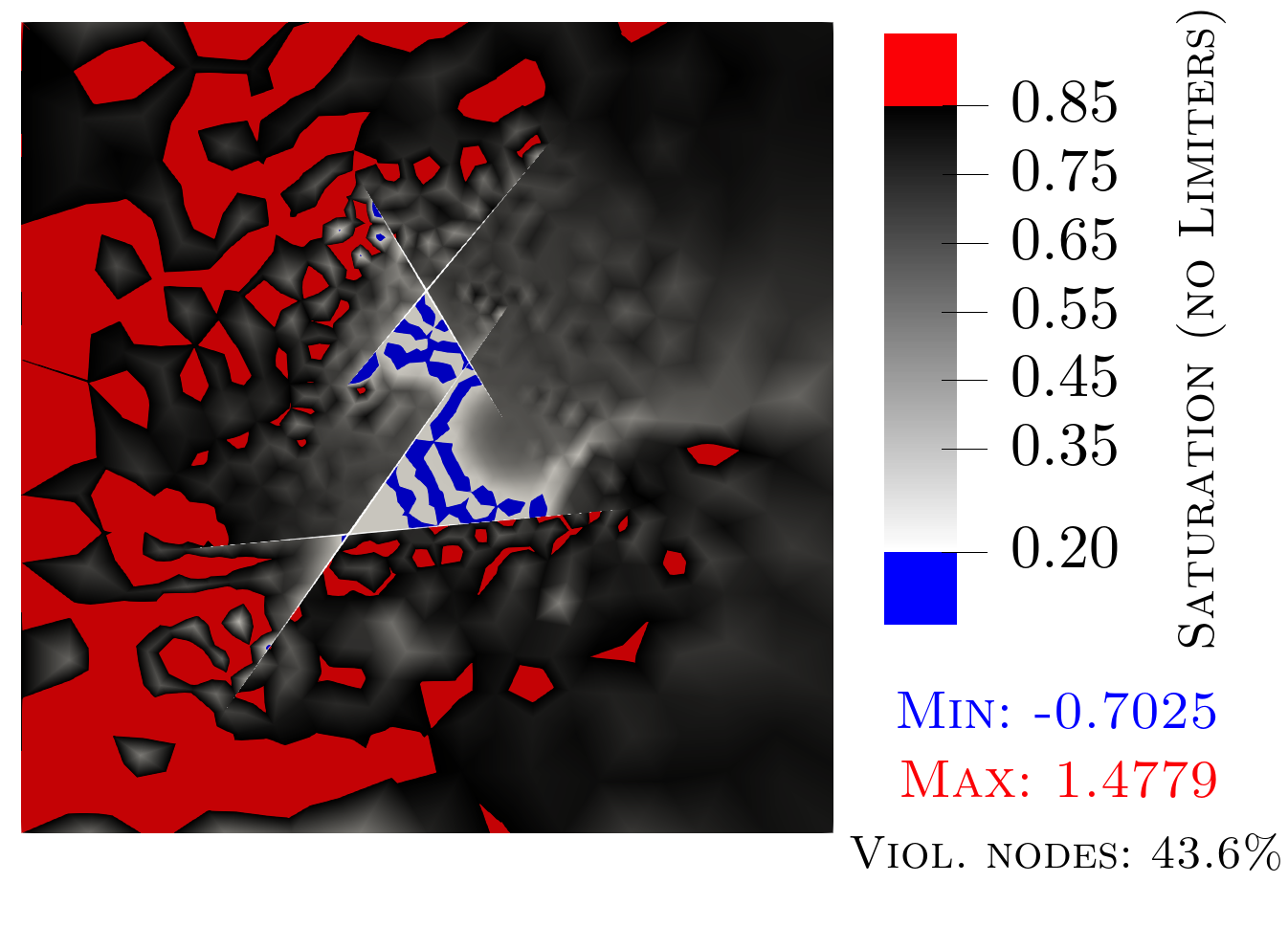}} 
        \hspace{.25cm}
        \subfigure[DG+FL+SL; $t=4500$ \si{\second}~\label{fig:}]{
        \includegraphics[clip,scale=0.15,trim=0 0cm 0cm 0]{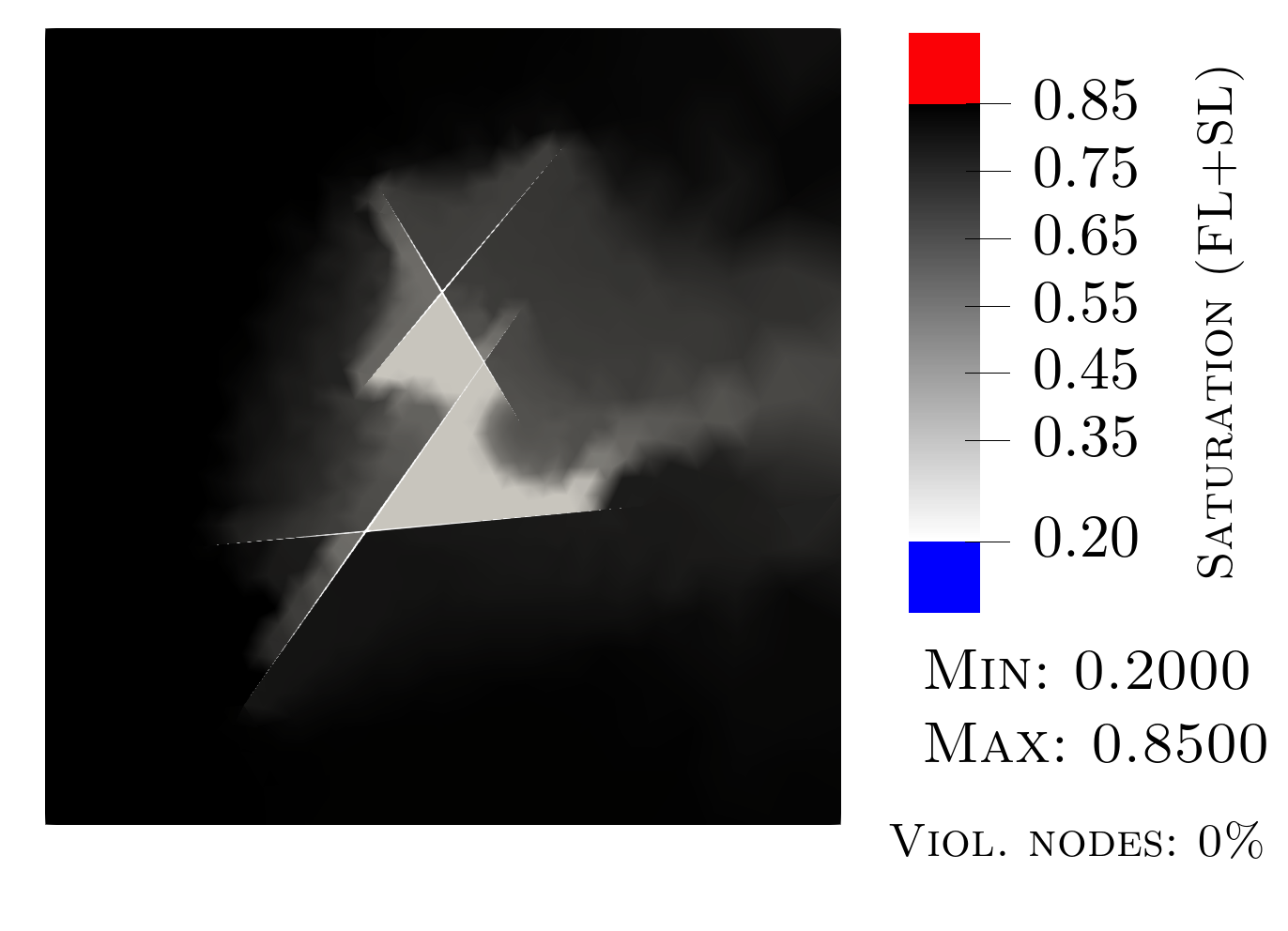}} 
        \caption{\textsf{Homogeneous domain with thin barrier:~}
            This figure shows the evolution of saturation profile using DG scheme without limiter (left) 
            and with the proposed limiters (right). The color mapping for $S$ in $[0.2,0.85]$ is grayscale, 
            while values below and above bounds are colored blue and red, respectively.
            As expected, DG approximation with no limiter yields noticeable violations,
            while limited DG scheme is capable of providing maximum-principle satisfying results. 
            In spite of this, the front under both unlimited and limited DG, propagates with the same 
            speed. 
        \label{Fig:2Dpatch_barriers_sat}}
\end{figure}
Figure \ref{Fig:2Dpatch_barriers_pres} and \ref{Fig:2Dpatch_barriers_vel} show the wetting phase pressure contour and 
velocity field at $t=4500$ \si{\second}. Velocities are computed at time $t_n$, 
using  the formula: $\mathbf{u}_w^n = -K \lambda_w(S_{n}) \nabla P_{n}$.
We can see that pressure drops linearly near the top and bottom edges, which confirms that fluid steers 
clear of the central barrier and flows around it. 
When no limiter is used, spurious oscillations and erroneous high-velocity regions are visible in velocity 
solutions. Limited DG, on the other hand, gives very smooth approximations. 
From these results we conclude that the proposed numerical scheme is bound-preserving on unstructured meshes. 
\begin{figure}
    \subfigure[DG with not limiter\label{Fig:}]{
        \includegraphics[clip,scale=0.13,trim=0 0cm 0cm
    0]{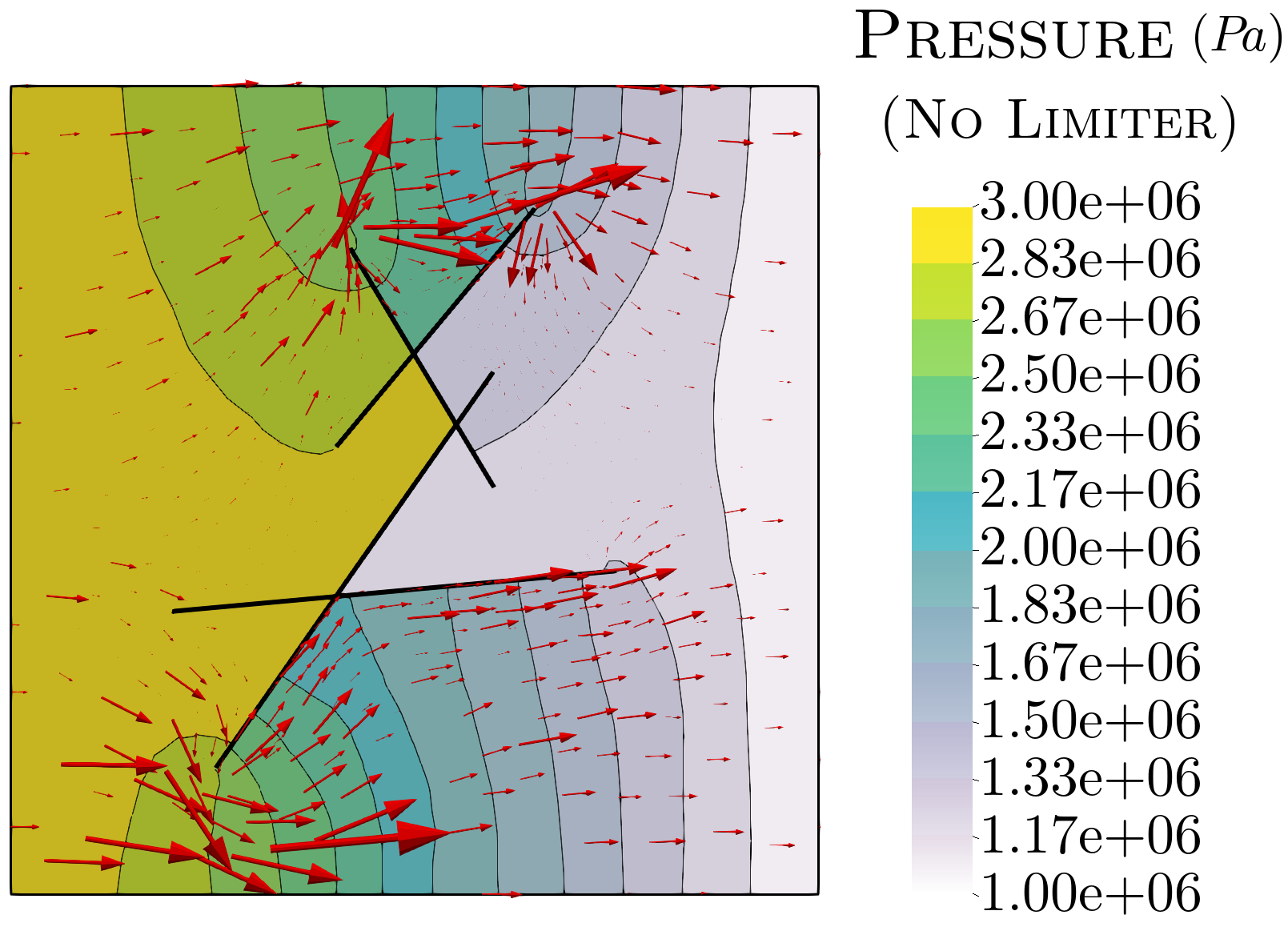}}
        \hspace{.25cm}
    \subfigure[DG+FL+SL\label{fig:}]{
        \includegraphics[clip,scale=0.13,trim=0 0cm 0cm 0]{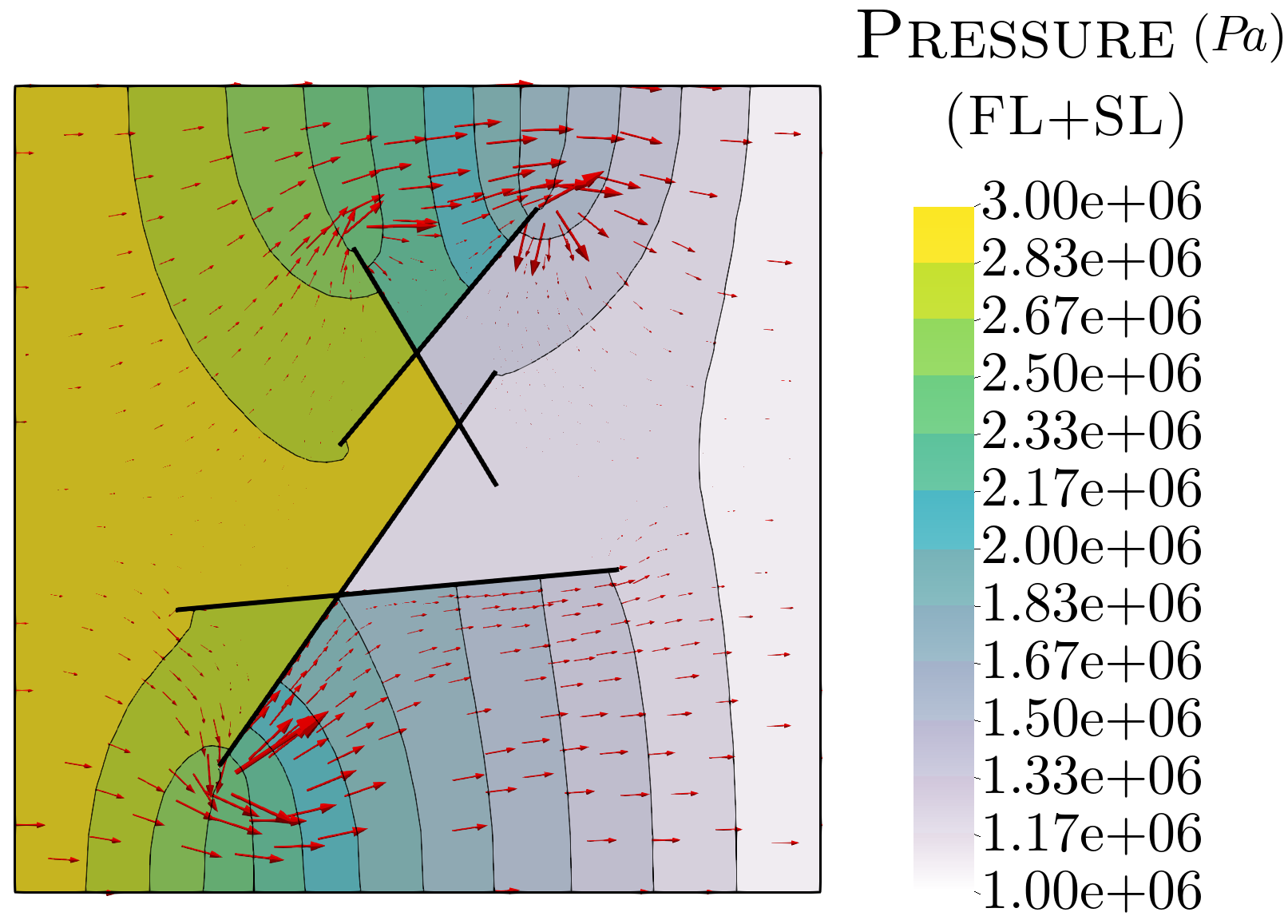}} \\
        \caption{\textsf{Homogeneous domain with thin barrier:}
            This figure depicts the pressure solutions at final time $t=4500$ \si{\second} using DG scheme (a) without 
            limiter and (b) with limiters. The color contours represent the wetting phase pressure 
            and the red arrows represent the velocity field. The length of the arrows scale with the magnitude of 
            velocity. 
            For both cases, flow goes around the impassible barrier and pressure linearly drops near top and 
            bottom channels.
        \label{Fig:2Dpatch_barriers_pres}}
\end{figure}
\begin{figure}
    \subfigure[DG with no limiter\label{Fig:}]{
        \includegraphics[clip,scale=0.13,trim=0 0cm 0cm
    0]{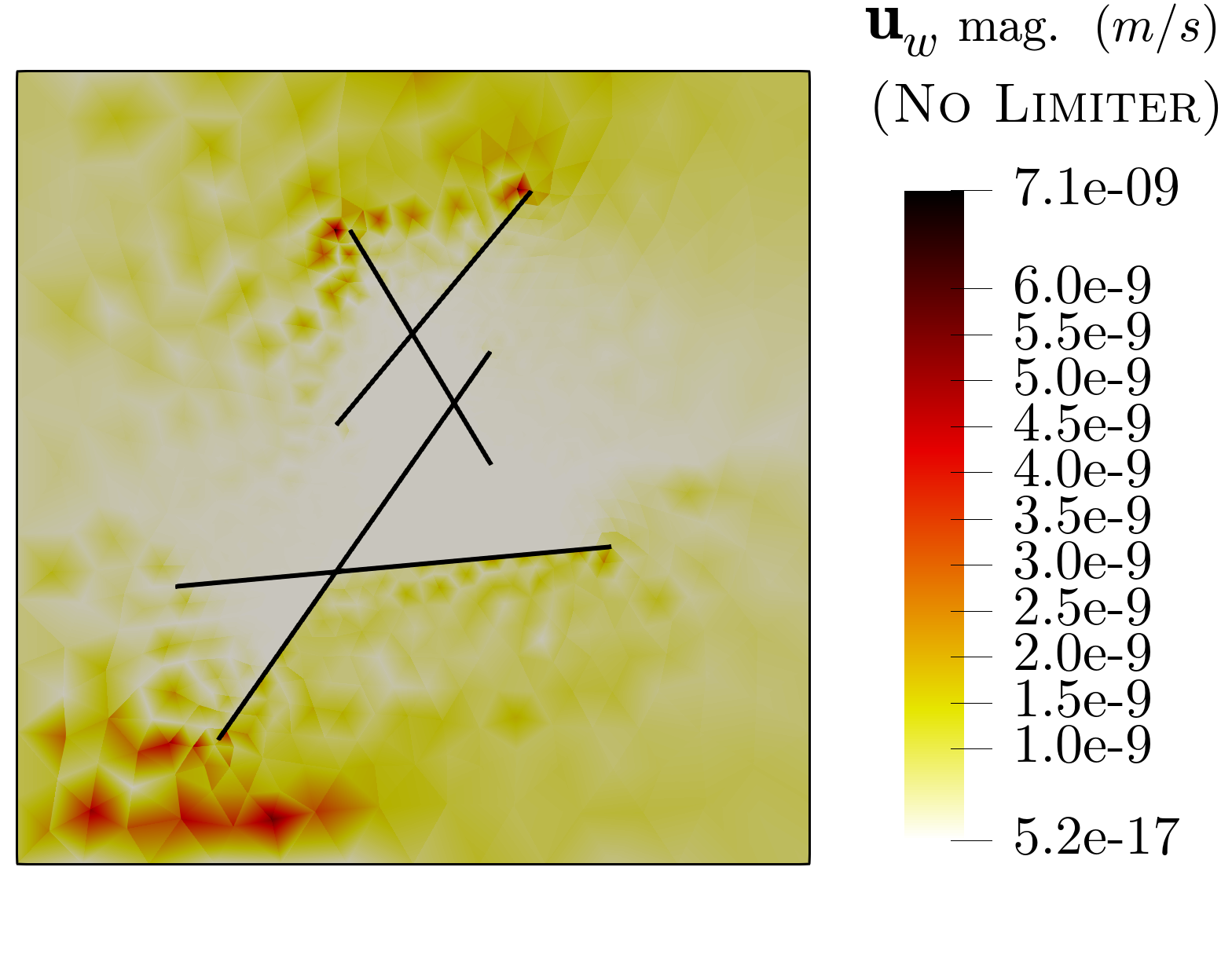}}
        \hspace{.25cm}
    \subfigure[DG+FL+SL\label{fig:}]{
        \includegraphics[clip,scale=0.13,trim=0 0cm 0cm 0]{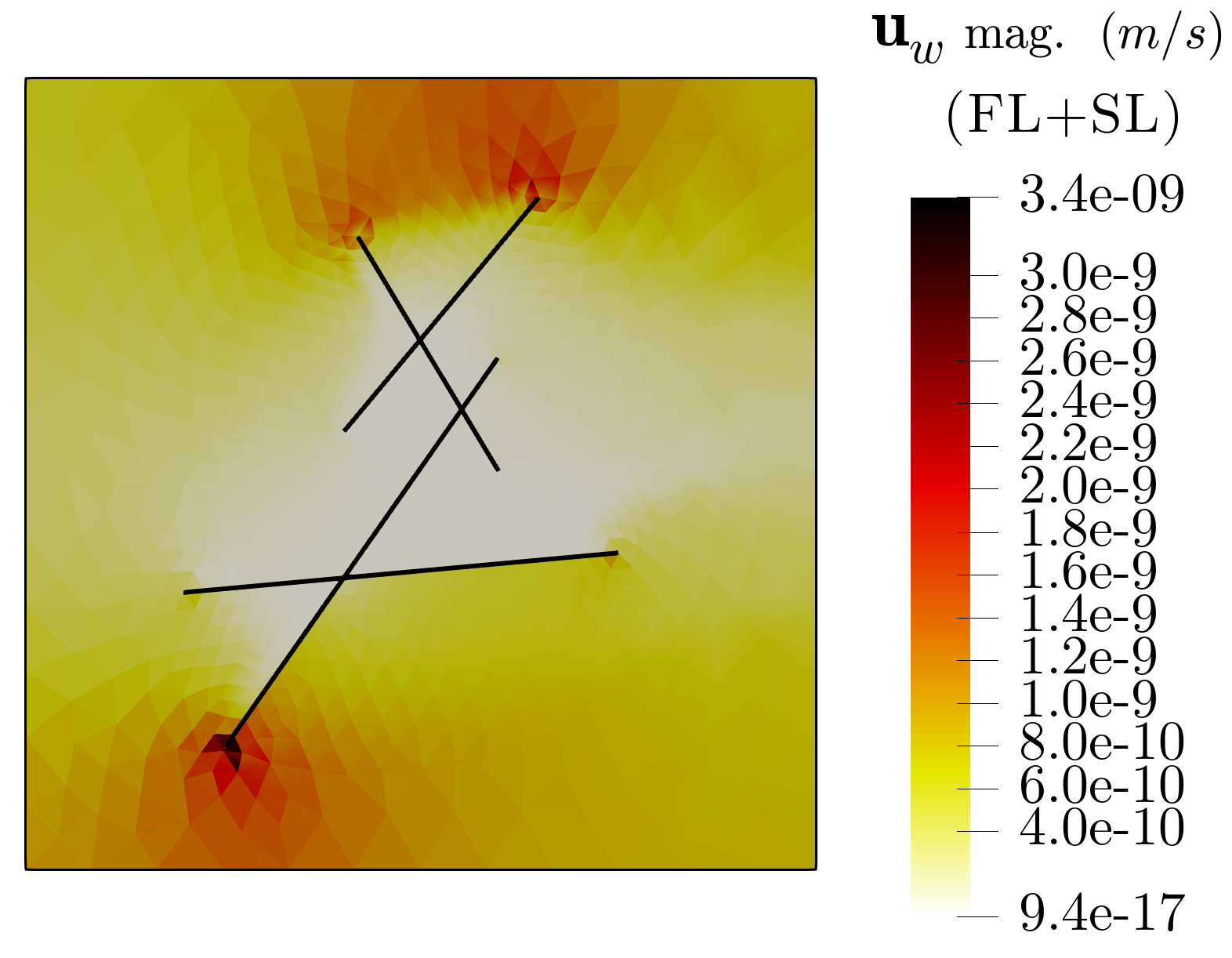}} 
        \caption{\textsf{Homogeneous domain with thin barrier:}
            This figure shows the magnitude  of wetting phase velocity at final time $t=4500$ 
            \si{\second} using DG scheme (a) without limiter and (b) with limiters. 
            The main inference from this figure is that when no limiter is used (a), DG approximation induces
            overestimation and spurious oscillations in velocity field.  The proposed limiting scheme mitigates this issue 
            and yields smooth solutions (b).
        \label{Fig:2Dpatch_barriers_vel}}
\end{figure}

\subsubsection{Non-homogeneous domain}
For this problem, permeability is $10^{-8}$ \si{\meter\squared} everywhere except inside a square 
inclusion of length $20$ \si{\meter} located at the center of the domain, where the permeability is 
$10^4$ times smaller.  The domain is discretized with a structured rectangular mesh of size
$h=1.25$ \si{\meter}. Time step is set to $\tau=0.5$ \si{\second} and the simulation advances up to $T=650$ \si{\second}. 
The remaining parameters are the same as in Section \ref{sub:2Dpatch_homogen}.
The discrete saturation at different snapshots of $t=350$ and $t=650$ \si{\second} are depicted in Figure 
\ref{Fig:2Dpatch_nonHomogen_sat}, where saturation values beyond the physical bounds 
(i.e., $S>0.85$ and  $S<0.2$) are clipped away. Evidently, no matter if limiters are used or not,
the injected wetting phase travels from left to right while avoiding the region of lower permeability. 
Both limited and unlimited schemes generate sharp and consistent saturation fronts. 
However, without limiter, the DG scheme presents strong oscillations behind and ahead of the 
inclusion. When limiters are activated, oscillations are suppressed and solutions are free of 
undershoots/overshoots.
\begin{figure}
    \subfigure[DG with no limiter; $t=350$ \si{\second}\label{Fig:}]{
        \includegraphics[clip,scale=0.135,trim=0 0cm 0cm
        0]{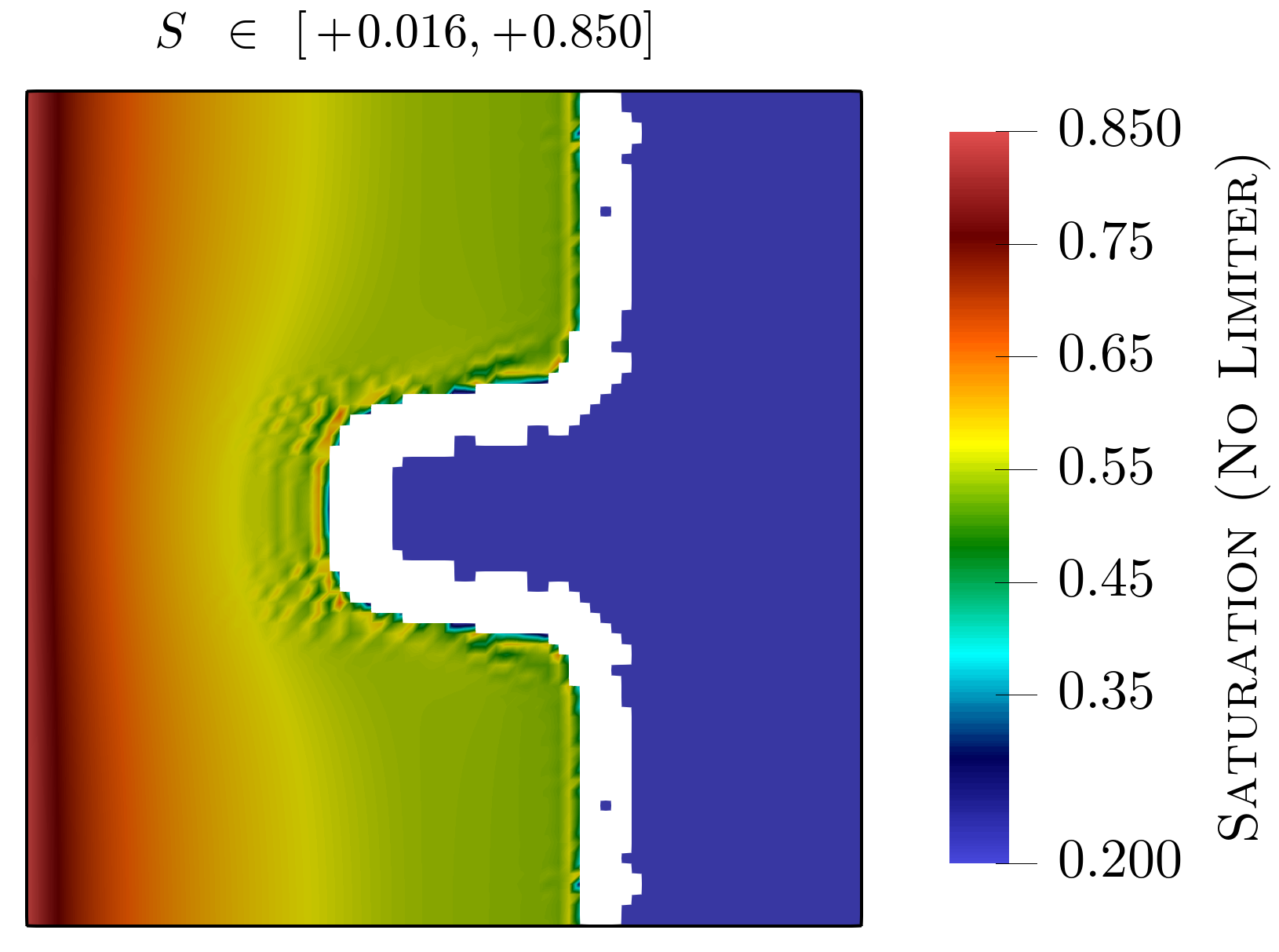}}
        \hspace{.05cm}
        \subfigure[DG with limiters; $t=350$ \si{\second}\label{fig:}]{
        \includegraphics[clip,scale=0.135,trim=0 0cm 0cm 0]{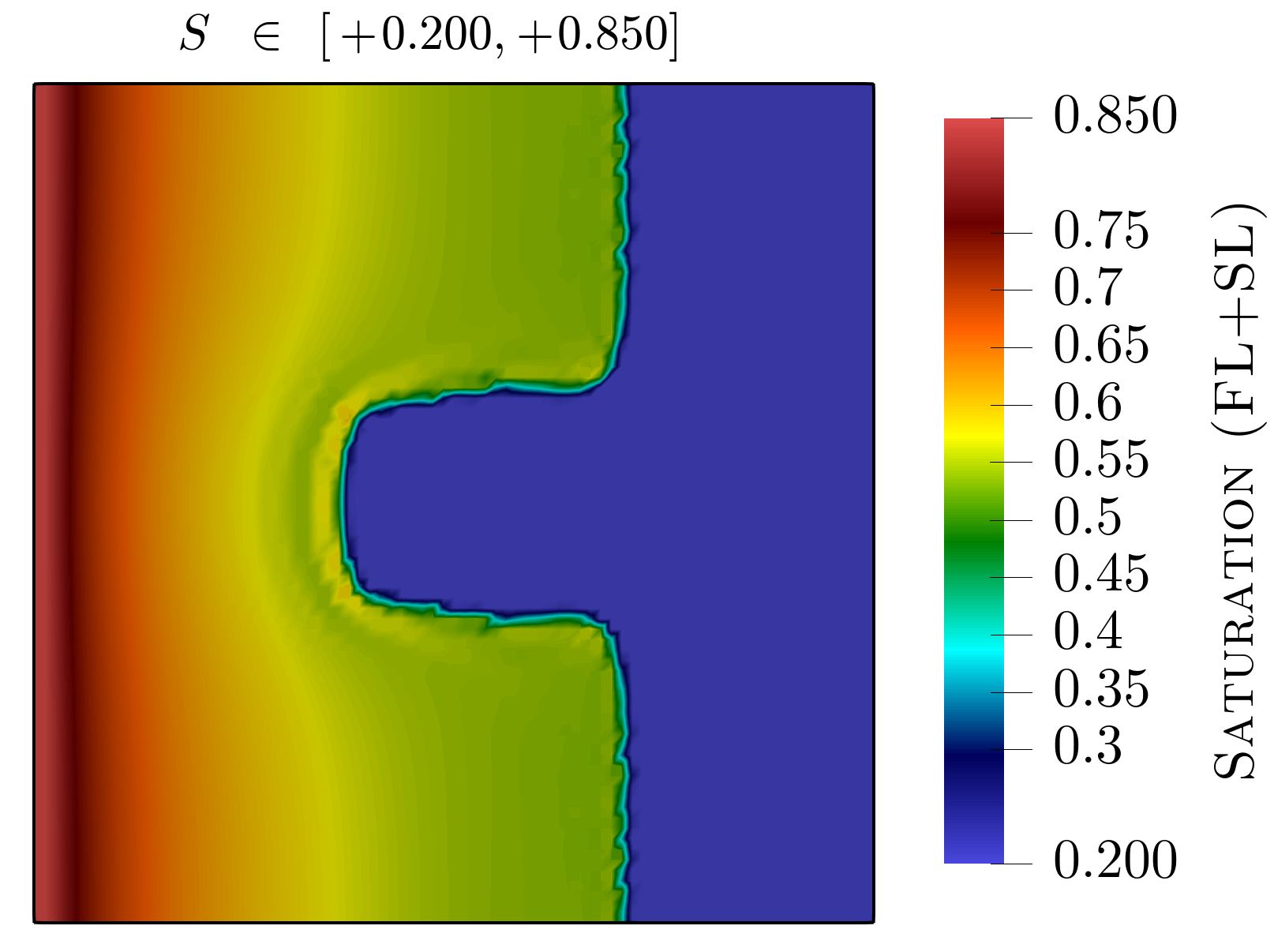}} \\
        \subfigure[DG with no limiter; $t=650$ \si{\second}\label{Fig:}]{
        \includegraphics[clip,scale=0.135,trim=0 0cm 0cm
        0]{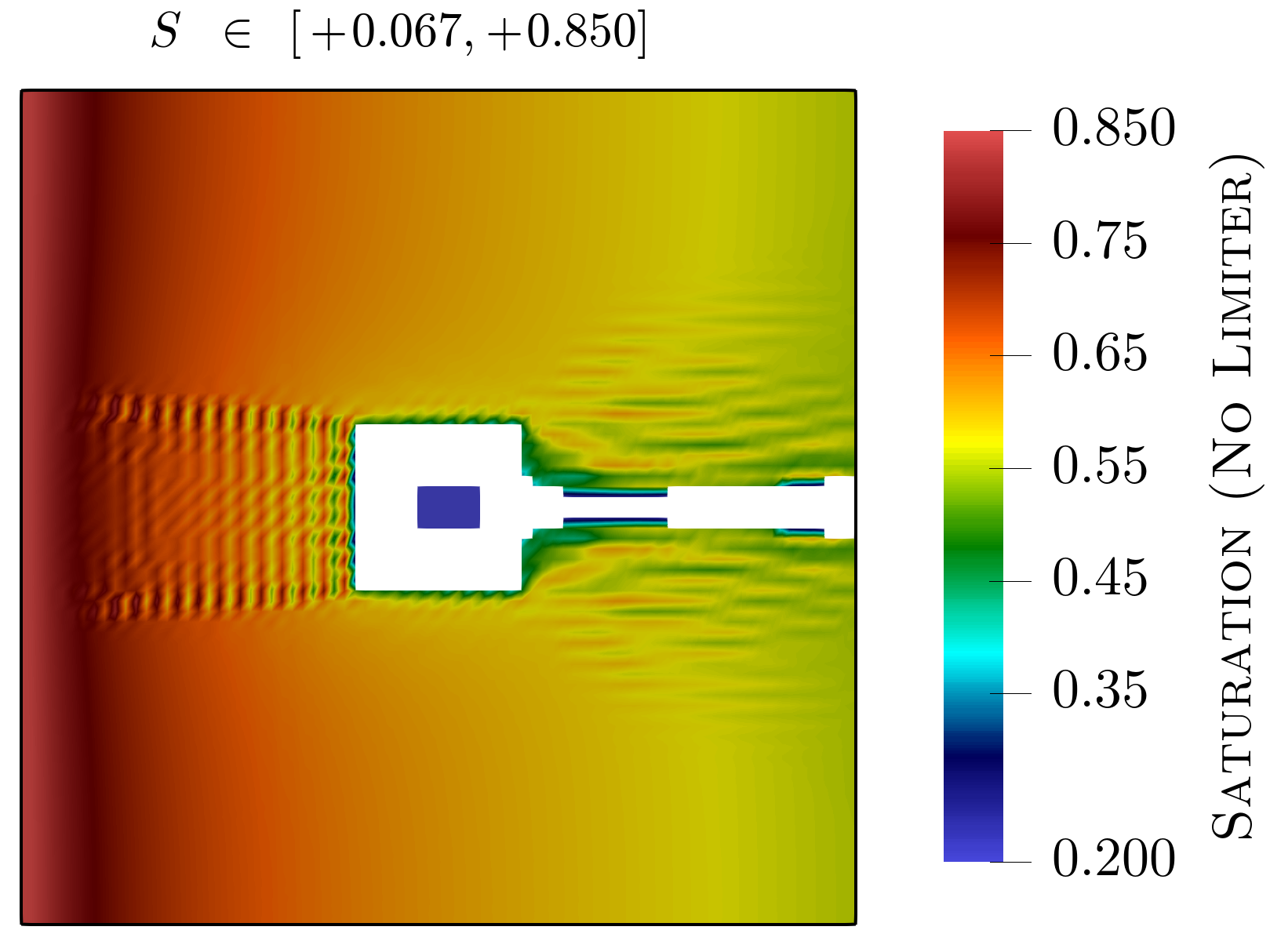}}
        \hspace{.05cm}
        \subfigure[DG with limiters; $t=650$ \si{\second}\label{fig:}]{
        \includegraphics[clip,scale=0.135,trim=0 0cm 0cm 0]{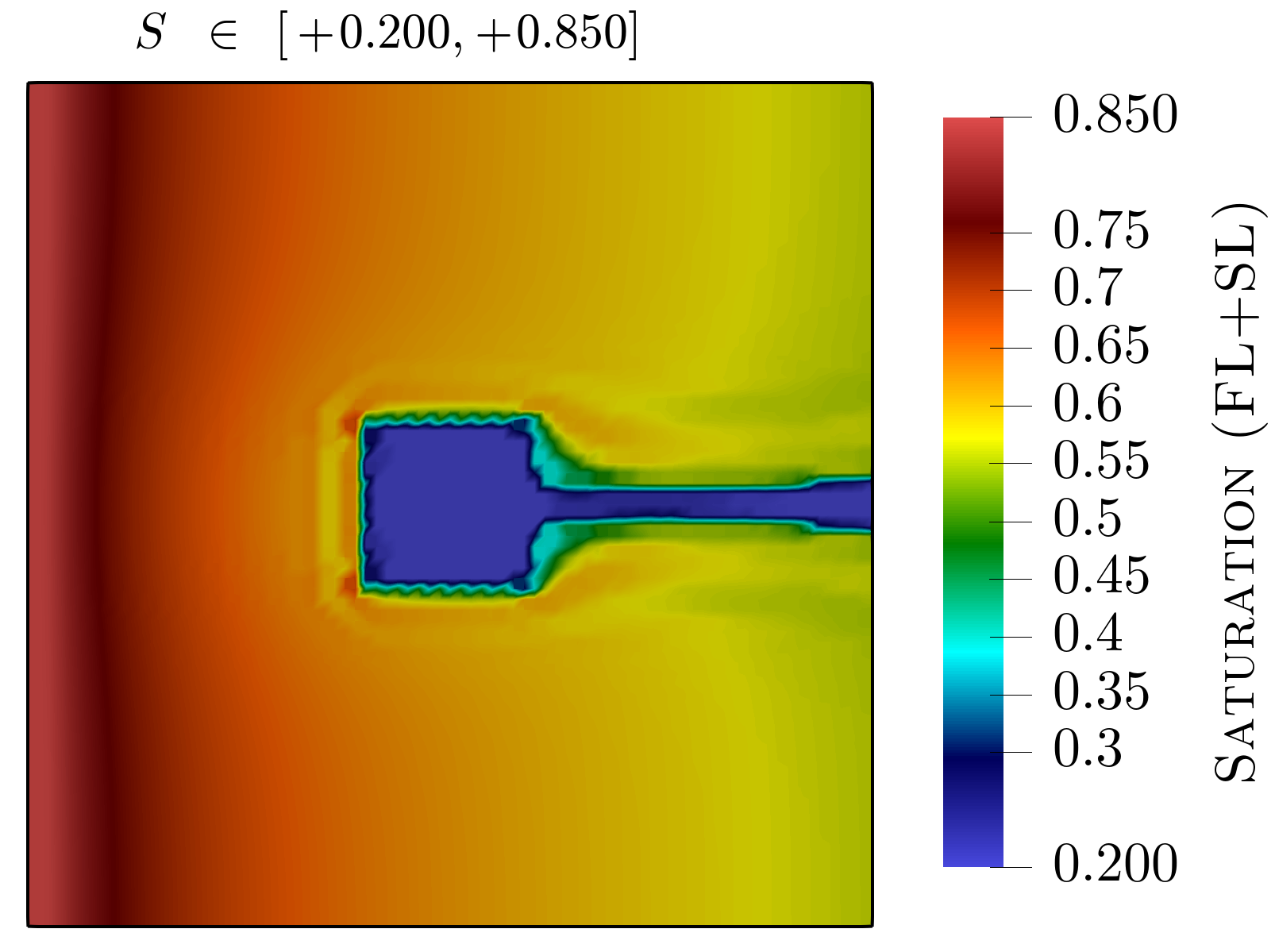}} \\
        \caption{\textsf{Non-homogeneous pressure-driven flow problem:}
            This figure shows saturation fields obtained with DG (left) and with DG+FL+SL (right). 
Values beyond the physical bounds (i.e., $S>0.85$ and  $S<0.2$) 
            are clipped away using tolerance $10^{-5}$. 
            Using either scheme, the wetting phase does not flood the inclusion and 
            saturation fronts remains sharp and propagate similarly.
            Notice that spurious oscillations and violations of the physical constraints 
            occur under the DG formulation but not under the limited DG.
        \label{Fig:2Dpatch_nonHomogen_sat}}
\end{figure}
Figures \ref{Fig:2Dpatch_nonHomogen_pres} and \ref{Fig:2Dpatch_nonHomogen_vel} depict the pressure and velocity
solutions, respectively, computed at $t=650$ \si{\second} by the DG formulation with limiter and without limiter.
Limiting scheme has minimal effect on the pressure but this is not the case for the velocity.
Velocities obtained under DG with no limiter exhibit spurious oscillations, which resemble those in saturation 
profile. On the other hand, the limiting scheme eliminates oscillations in the velocity field.
\begin{figure}
    \subfigure[DG; $t=650$ \label{Fig:}]{
        \includegraphics[clip,scale=0.135,trim=0 0cm 0cm
        0]{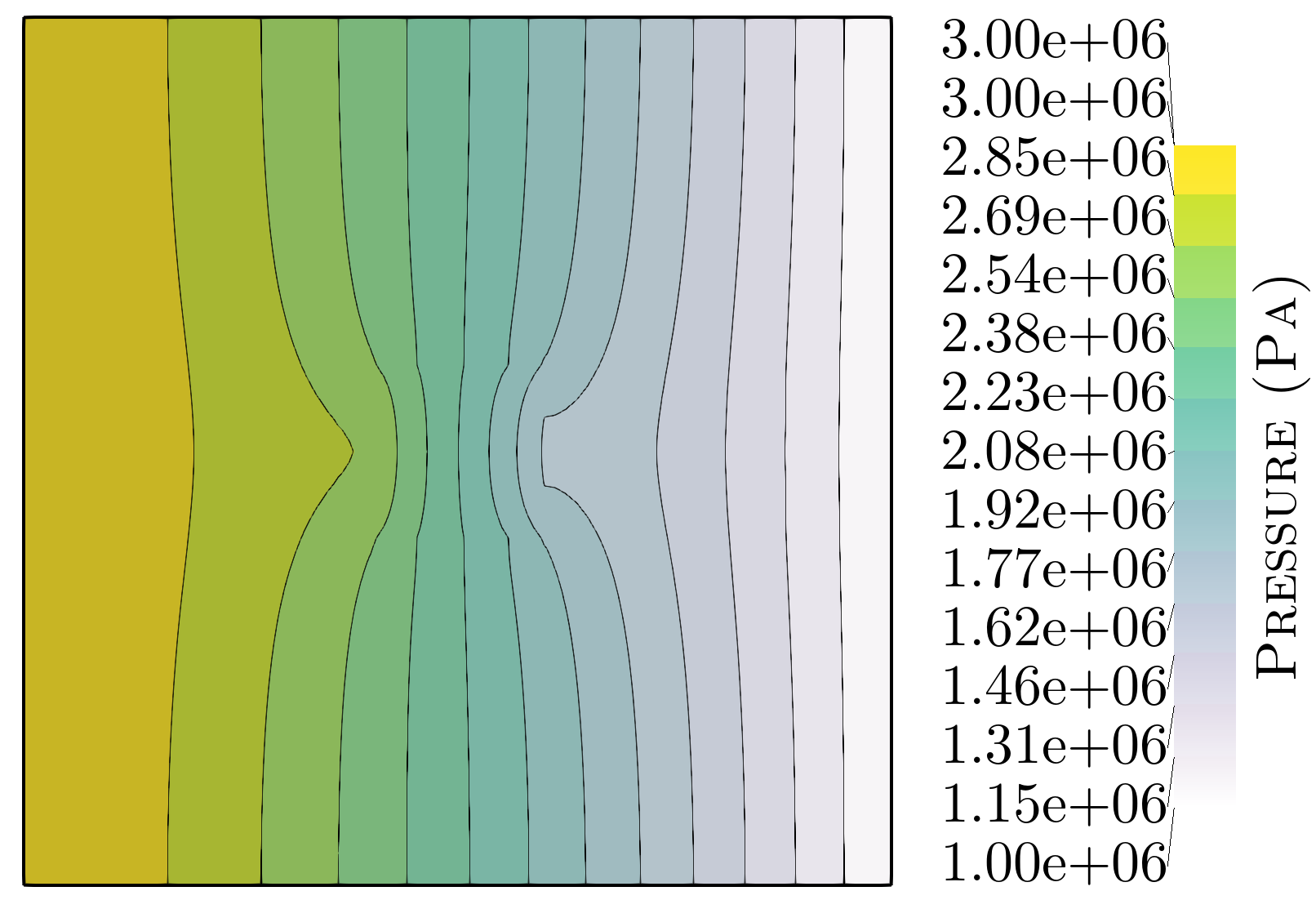}}
        \hspace{.05cm}
    \subfigure[DG+FL+SL; $t=650$ \label{fig:}]{
        \includegraphics[clip,scale=0.135,trim=0 0cm 0cm 0]{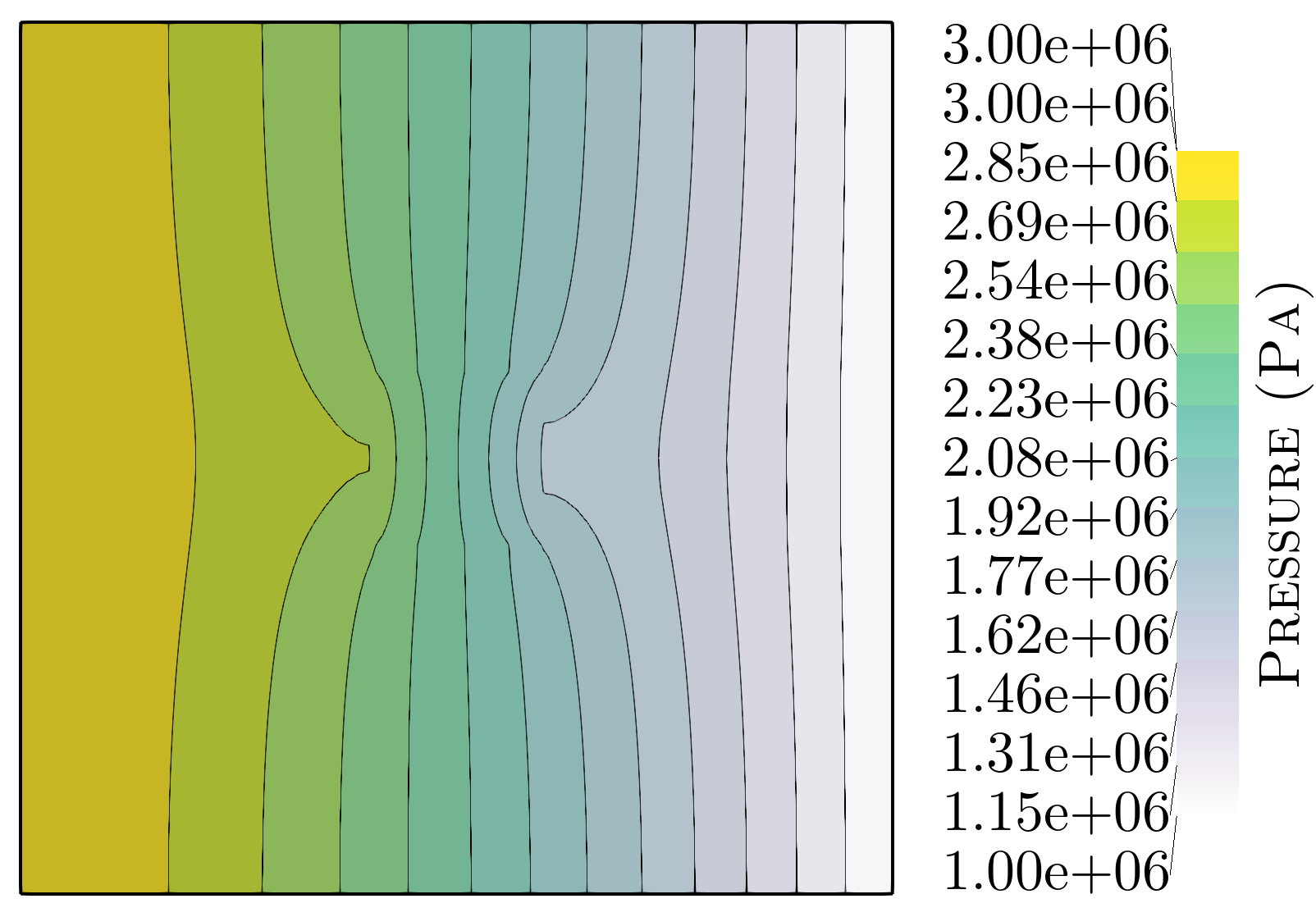}} \\
        \caption{\textsf{Non-homogeneous pressure-driven flow problem:}
            This figure shows pressure at final time $t=650$ \si{\second} using DG scheme (a) without 
            limiter and (b) with limiter. 
        \label{Fig:2Dpatch_nonHomogen_pres}}
\end{figure}
\begin{figure}
    \subfigure[DG; $t=650$ \label{Fig:}]{
        \includegraphics[clip,scale=0.13,trim=0 0cm 0cm
        0]{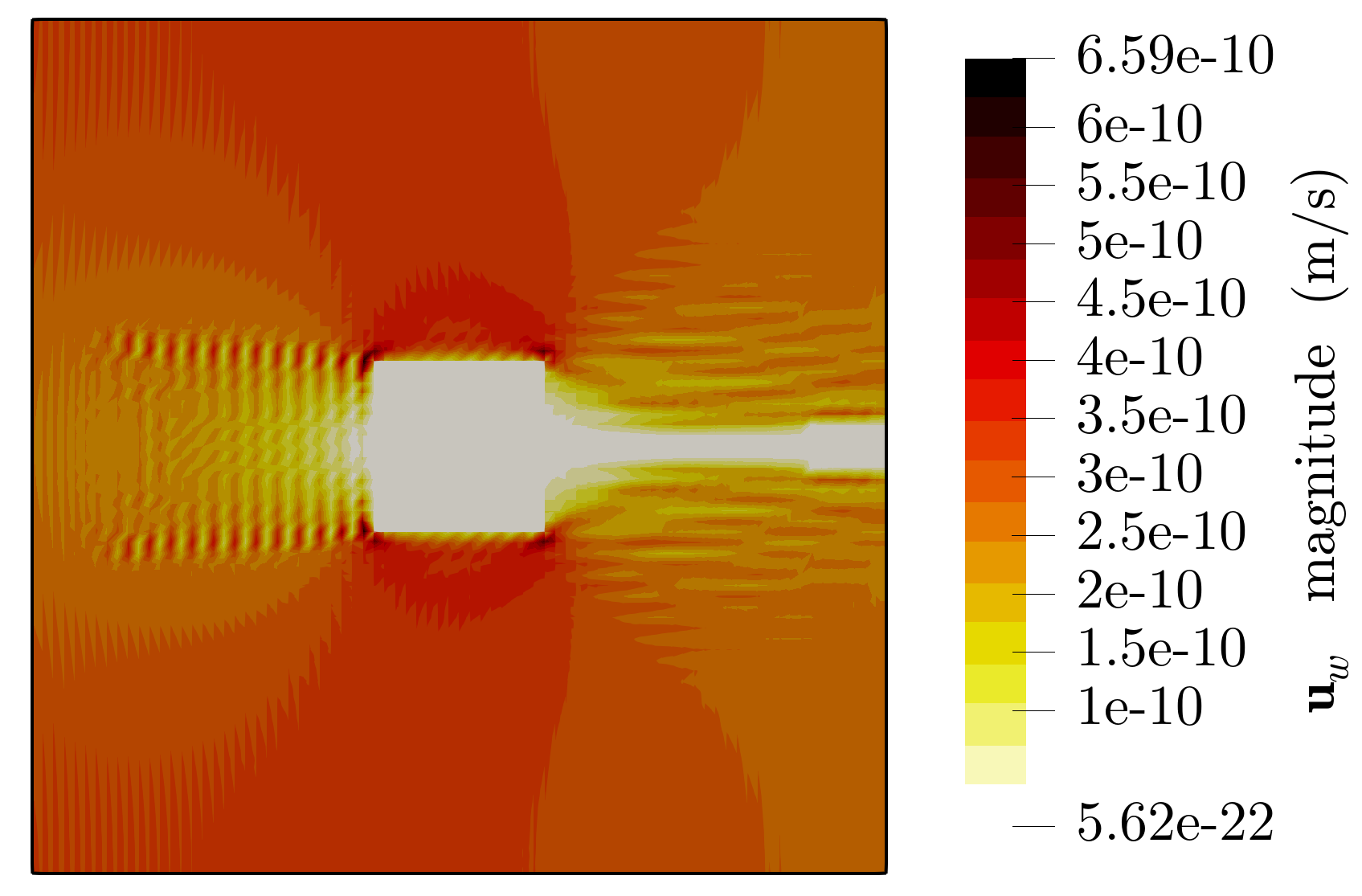}}
        \hspace{.05cm}
    \subfigure[DG+FL+SL; $t=650$ \label{fig:}]{
        \includegraphics[clip,scale=0.13,trim=0 0cm 0cm 0]{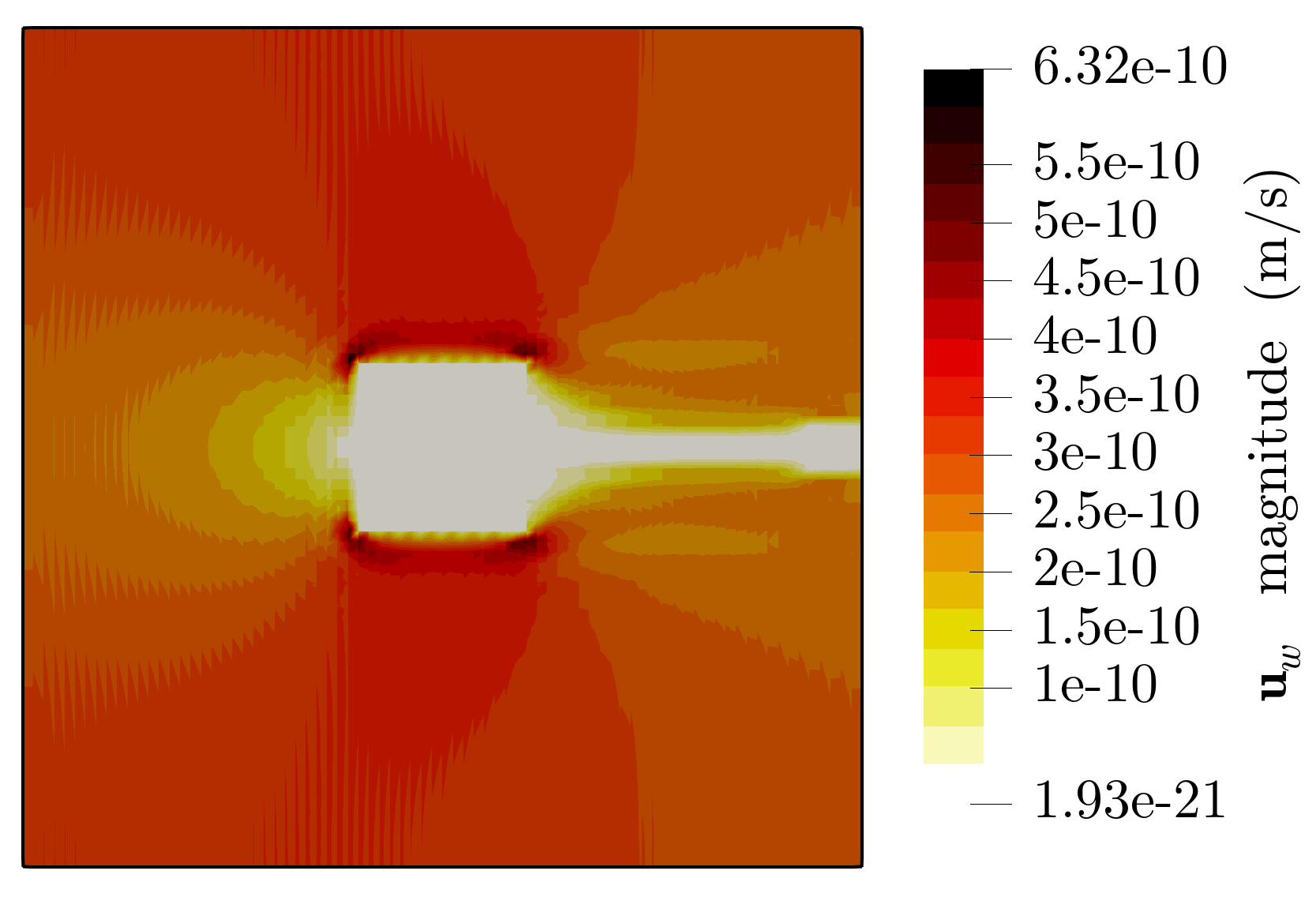}} \\
        \caption{\textsf{Two-dimensional pressure-driven flow problem:}
            This figure shows the magnitude of wetting phase velocity at final time $t=650$ \si{\second} using DG scheme 
            (a) without limiter and (b) with limiter. 
            In the  unlimited case, oscillations earlier observed in saturation are reflected in the velocity field.
            However, with limiters, oscillations are eliminated.
        \label{Fig:2Dpatch_nonHomogen_vel}}
\end{figure}

\subsection{Quarter five-spot problem}
\label{sub:Q5_homogen}
In this section, the performance and robustness of the limiters are assessed in the presence of 
wells, for both homogeneous and heterogeneous permeabilities.
We employ no flow boundary condition on the entire boundary, as shown in Figure \ref{Fig:Q5_schematic};
and assume zero capillary pressure. The flow is driven from an injection well at the bottom left corner 
to a production well at the top right corner. 
The wells are defined by source/sink terms, which are piecewise constant with compact support. 
That is, $\bar{q}$ is nonzero at injection well and $\underline{q}$ is nonzero at production well. 
The DG penalty parameters for test problems are taken as $\sigma=10$. 

\subsubsection{Homogeneous domain}
\label{sub:Q5_homogen}
The domain $\Omega=[0,100]^2$ \si{\meter\squared} is  partitioned into a crossed structured mesh 
of size $h=2.5$ \si{\meter}, as depicted in Figure \ref{Fig:Crossed_mesh}. The medium is homogeneous with 
$K=10^{-13}$ \si{\meter\squared} everywhere. We choose Brooks-Corey relative permeabilities as follows:
\begin{align}
    \label{Eqn:Q5_homogen_BrookCorey}
    k_{rw}(s_e) = s_e^2,\quad
    k_{r\ell}(s_e) = (1-s_e)^2,\quad
    s_e = \frac{S-s_{rw}}{1-s_{rw}-s_{r\ell}}.
\end{align}
The injection and production flow rates of wells are determined by the following constraint:
\begin{align}
    \int_{\Omega}\bar{q} = \int_{\Omega}\underline{q} = 7.03125\times10^{-4},
\end{align}
where $\bar{q}$ is piecewise constant on $[2.5,10]^2$ \si{\meter\squared} and $\bar{q}=0$ elsewhere and
$\underline{q}$ is piecewise constant on $[90,97.5]^2$ \si{\meter\squared} and  $\underline{q}=0$ elsewhere.
The final time is $T=21$ days and time step is $\tau=0.057$ days. 

Figure \ref{Fig:Q5_sat} shows the wetting phase saturations at two different times ($t=10$ and $t=21$ days), 
for three schemes: DG, DG+SL, DG+FL+SL. 
The figure shows that violations of the maximum principle for 
the unlimited DG solution occur in the neighborhood of the injection well and after the front; 
in addition the DG solution is not monotone before the front. Adding a slope limiter helps with the 
monotonicity of the solution and with decreasing the number of elements where the maximum principle is 
not satisfied. The proposed numerical scheme, DG+FL+SL, completely eliminates violation of maximum principle: the solution is monotone and bound-preserving. 
Figure~\ref{Fig:Q5_pres} and \ref{Fig:Q5_vel} show the wetting phase pressure contours and velocity fields 
at $t=10$ days for all three cases.  Differences are minimal for the pressure and velocity fields.
Finally we display the local mass balance error in Figure~\ref{Fig:Q5_LMB} for all three cases; the local mass
balance error is a piecewise constant field $\mathcal{M}$ defined by \eqref{Eqn:incomp_BoM}.  We observe that the local mass balance is negligible (of the order of $10^{-11}$) for the DG scheme with or without limiters.
\begin{figure}[]
    \centering
    \includegraphics[width=0.52\linewidth]{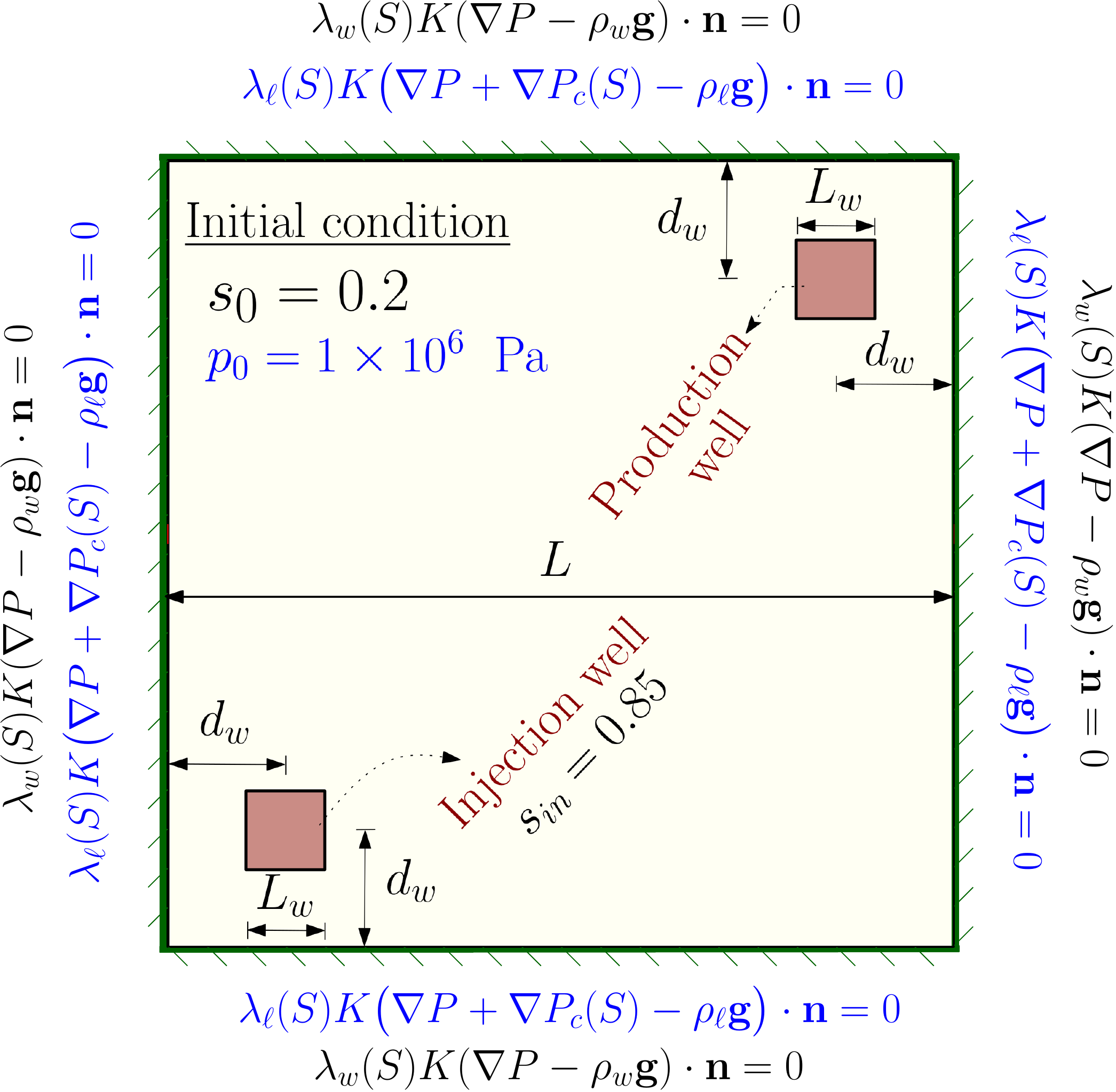}
    \caption{\textsf{Quarter five-spot problem:}
    This figure provides a pictorial description and the boundary value problem. No flow boundary conditions
    are prescribed on the entire boundary.
    }%
    \label{Fig:Q5_schematic}
\end{figure}
\begin{figure}
    \subfigure[DG; $t=10$  days\label{Fig:}]{
        \includegraphics[clip,scale=0.09,trim=0 0cm 0cm
        0]{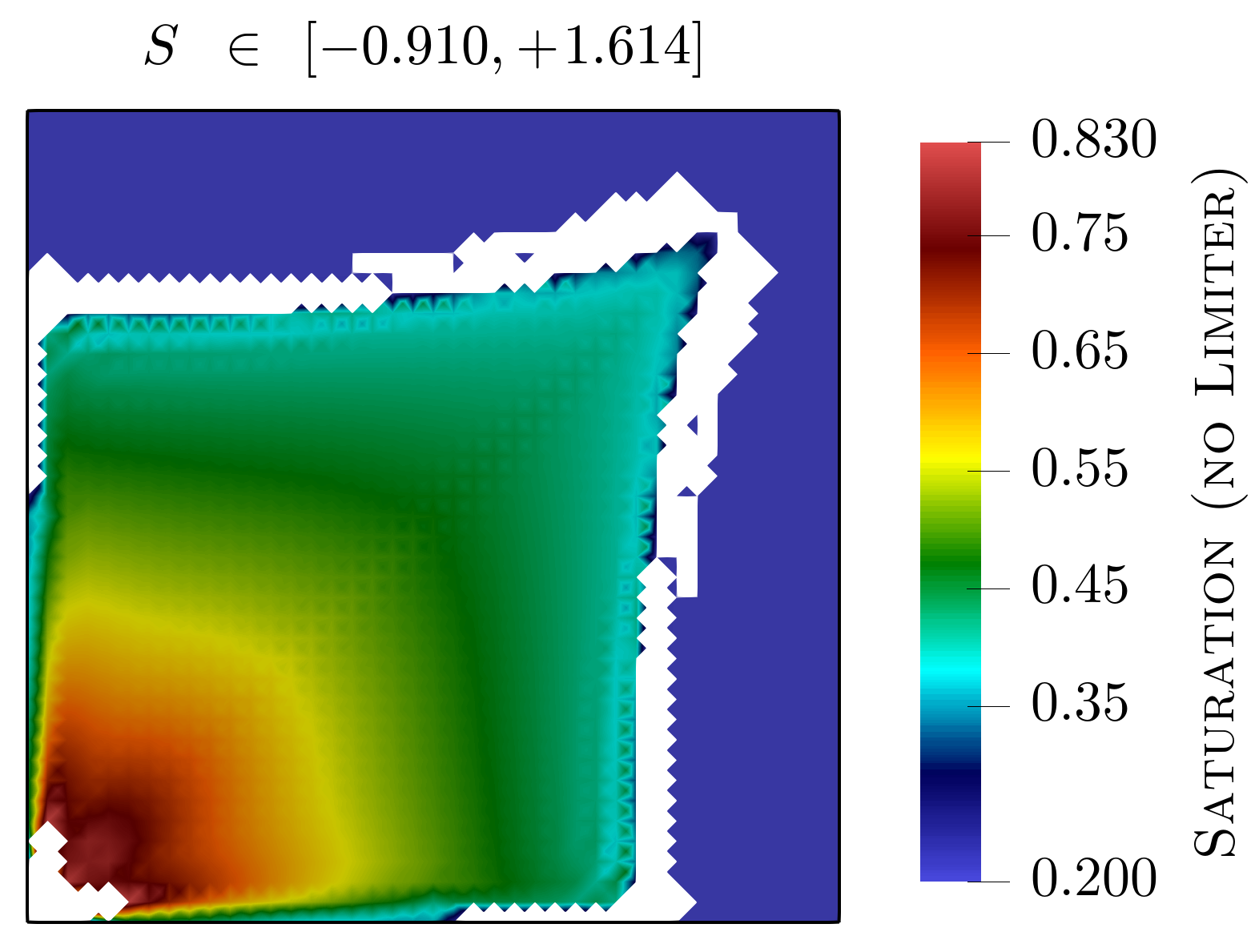}}
        \hspace{.05cm}
    \subfigure[DG+SL; $t=10$  days\label{fig:}]{
        \includegraphics[clip,scale=0.09,trim=0 0cm 0cm 0]{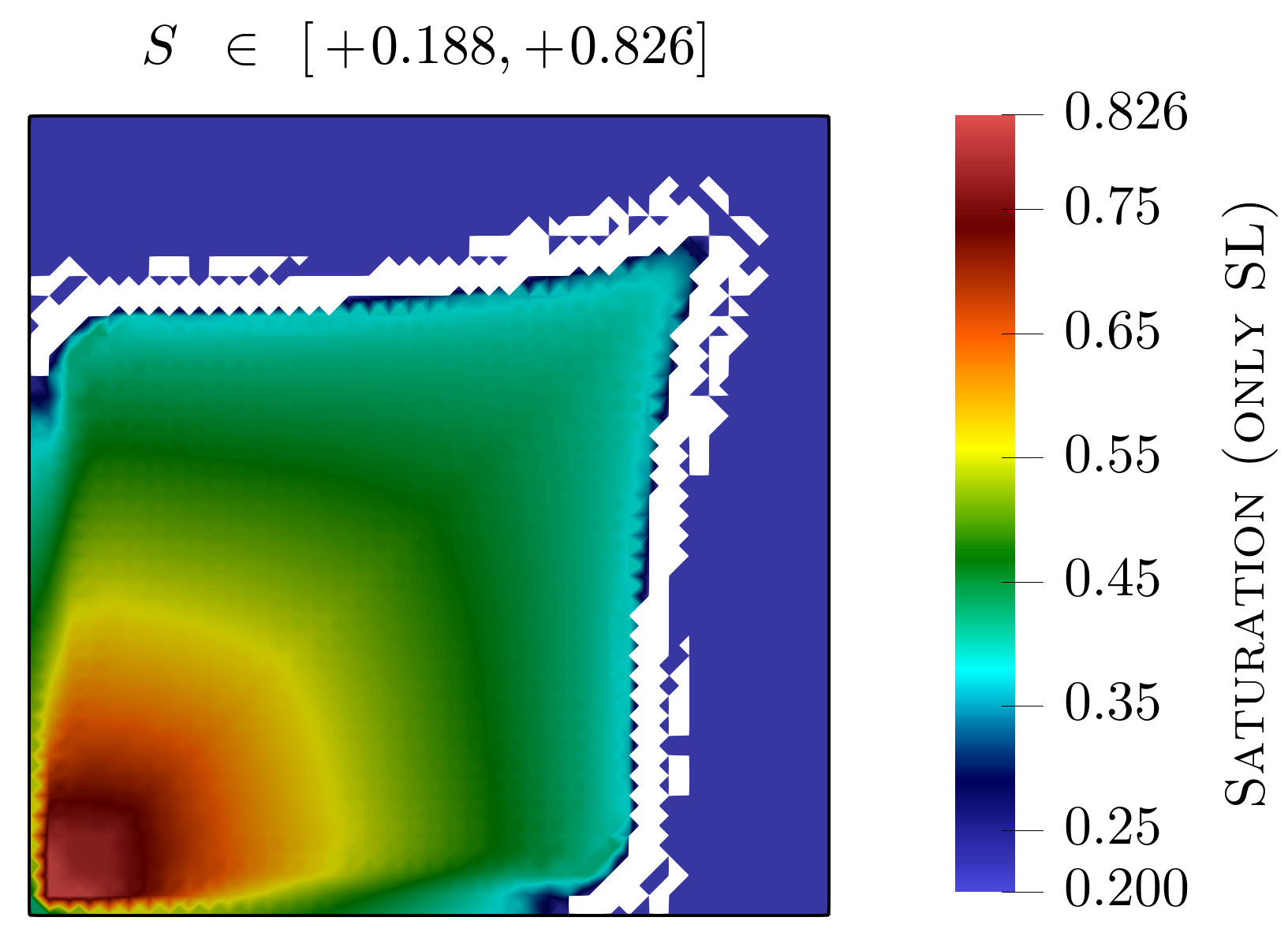}} 
        \hspace{.05cm}
    \subfigure[DG+FL+SL; $t=10$ days \label{fig:}]{
        \includegraphics[clip,scale=0.09,trim=0 0cm 0cm 0]{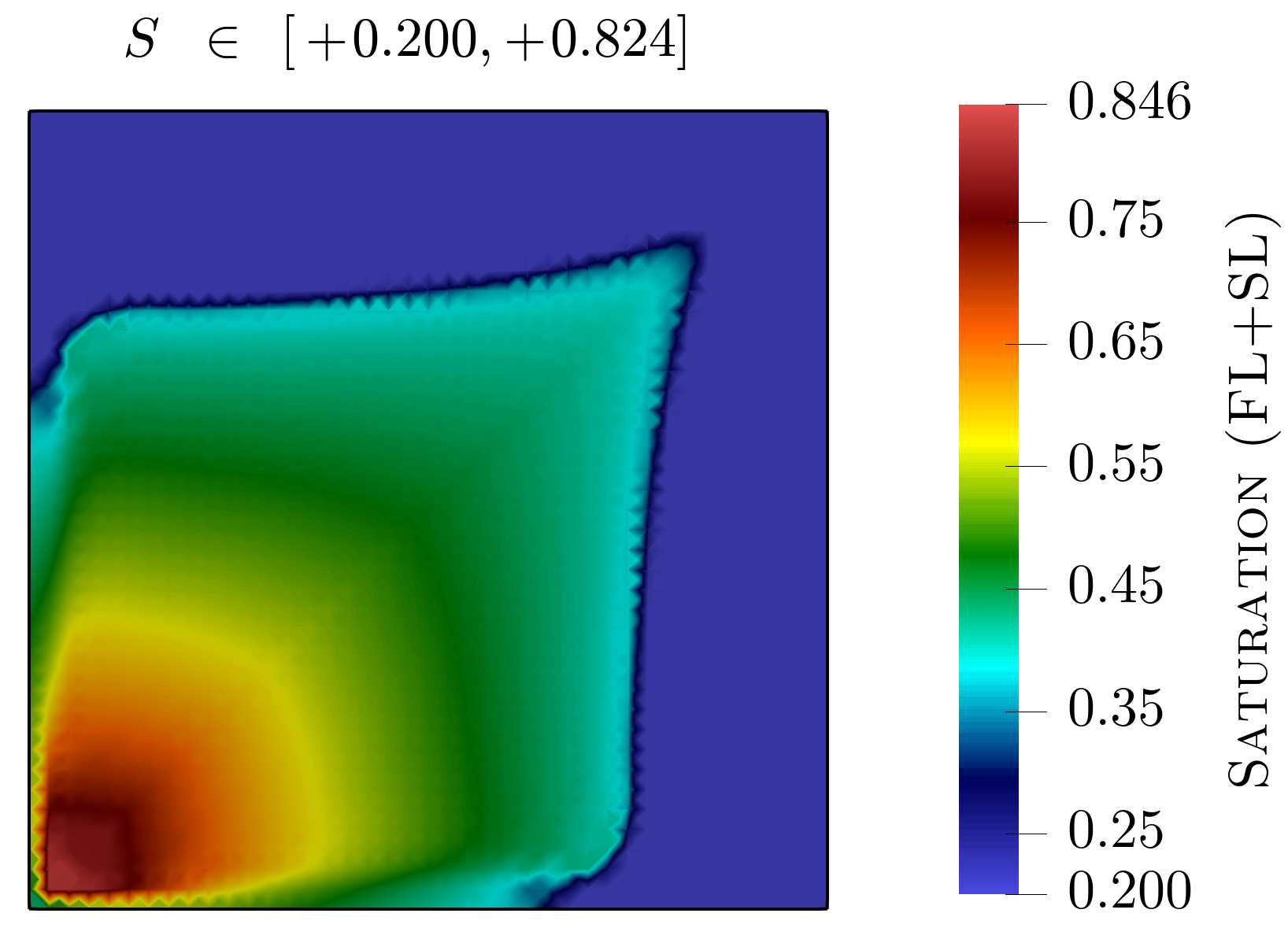}} \\
    \subfigure[DG; $t=21$ days \label{Fig:}]{
        \includegraphics[clip,scale=0.09,trim=0 0cm 0cm
        0]{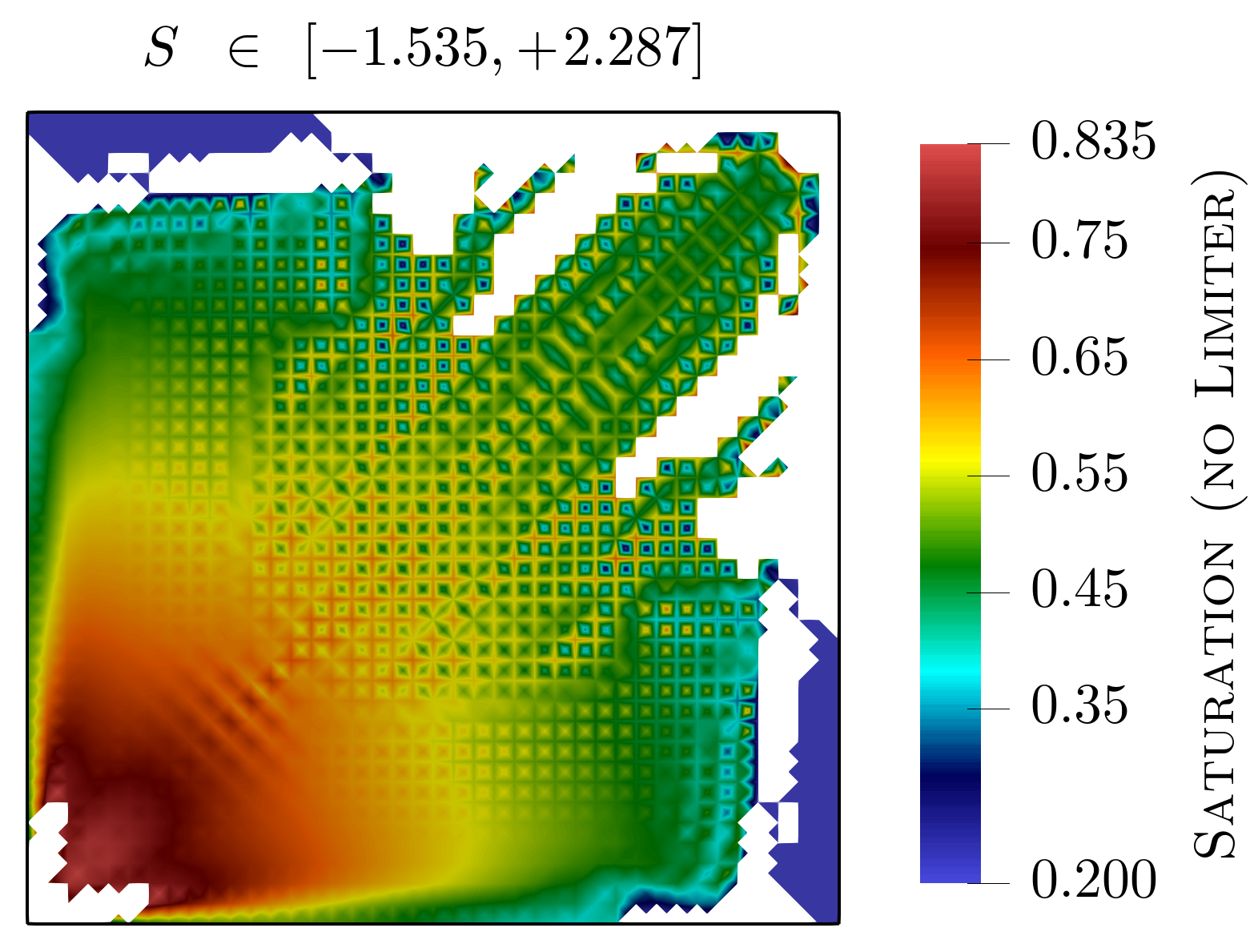}}
        \hspace{.05cm}
    \subfigure[DG+SL; $t=21$  days\label{fig:}]{
        \includegraphics[clip,scale=0.09,trim=0 0cm 0cm 0]{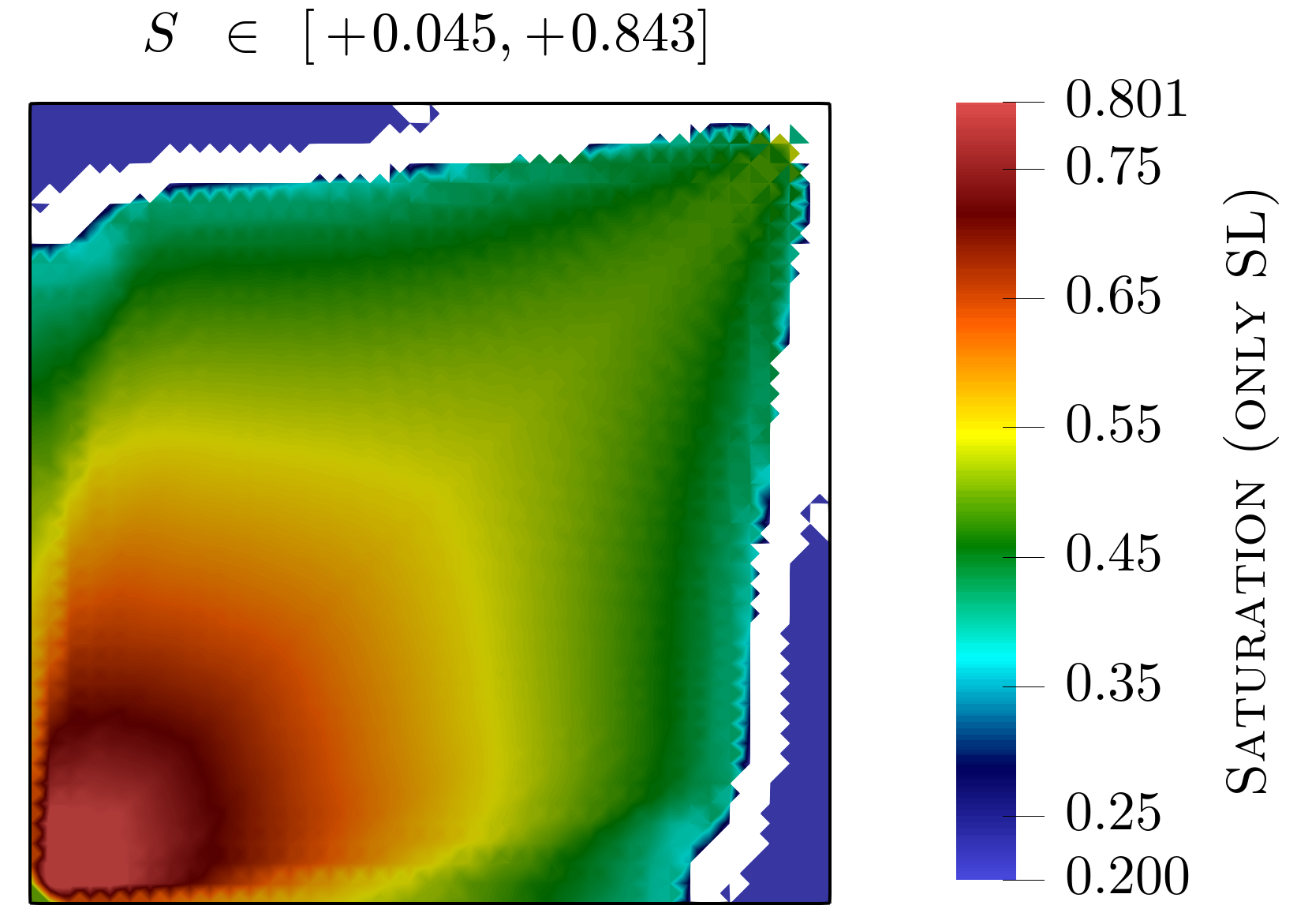}} 
        \hspace{.05cm}
    \subfigure[DG+FL+SL; $t=21$  days\label{fig:}]{
        \includegraphics[clip,scale=0.09,trim=0 0cm 0cm 0]{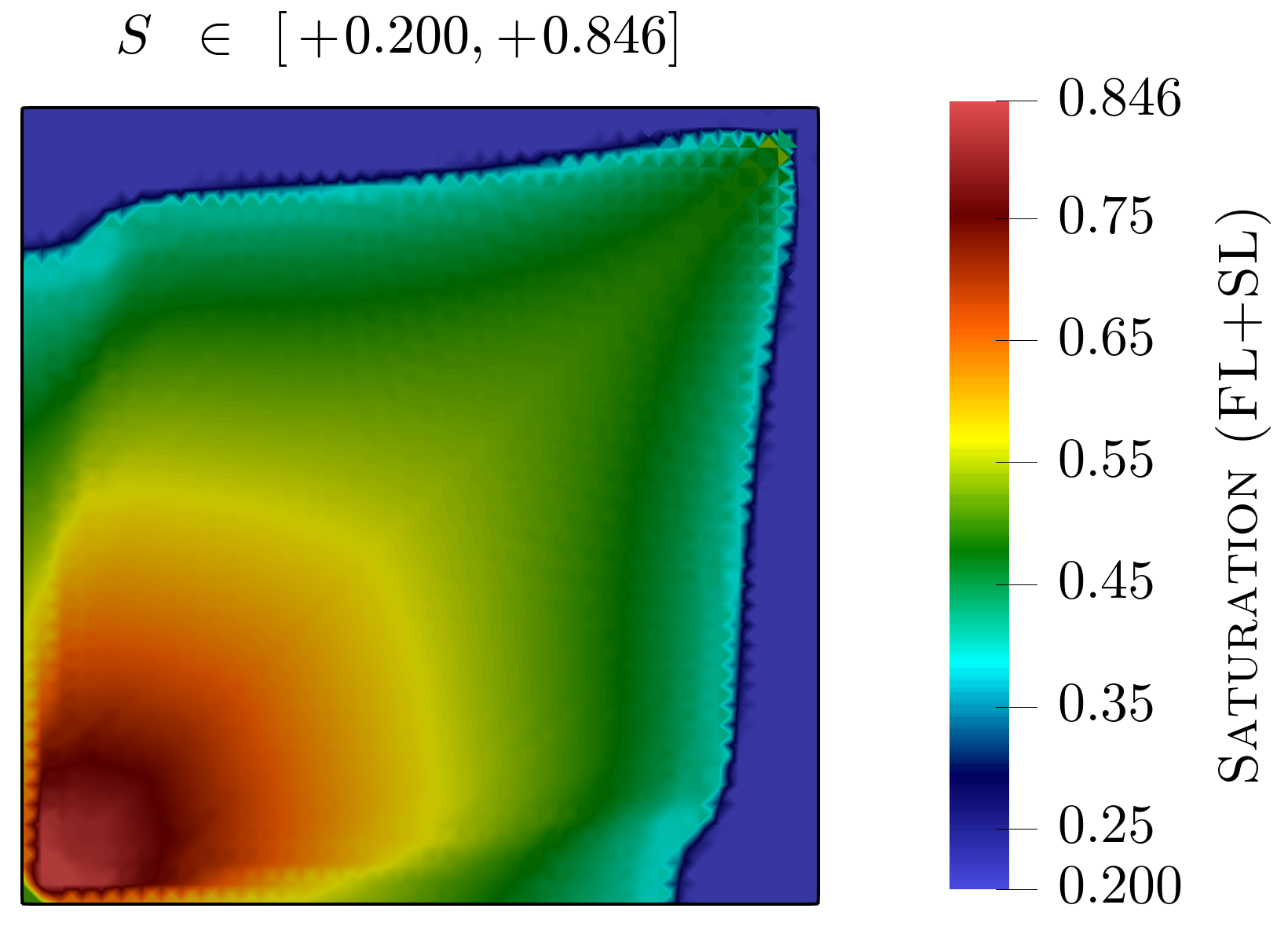}} \\
        \caption{\textsf{Quarter five-spot problem with homogeneous permeability:} 
                This figure shows the saturation solutions obtained with DG (left), DG+SL (middle), and 
                DG+FL+SL (right) at two different time steps. Values beyond the physical bounds 
                (i.e., $S>0.85$ and $S<0.2$) are clipped away using tolerance $10^{-5}$. 
                This figure suggests that DG+FL+SL, unlike the two other schemes, provides maximum-principle
                satisfying results at all time steps.
        \label{Fig:Q5_sat}}
\end{figure}
\begin{figure}
    \subfigure[DG \label{Fig:q5spres1}]{
        \includegraphics[clip,scale=0.095,trim=0 0cm 0cm
        0]{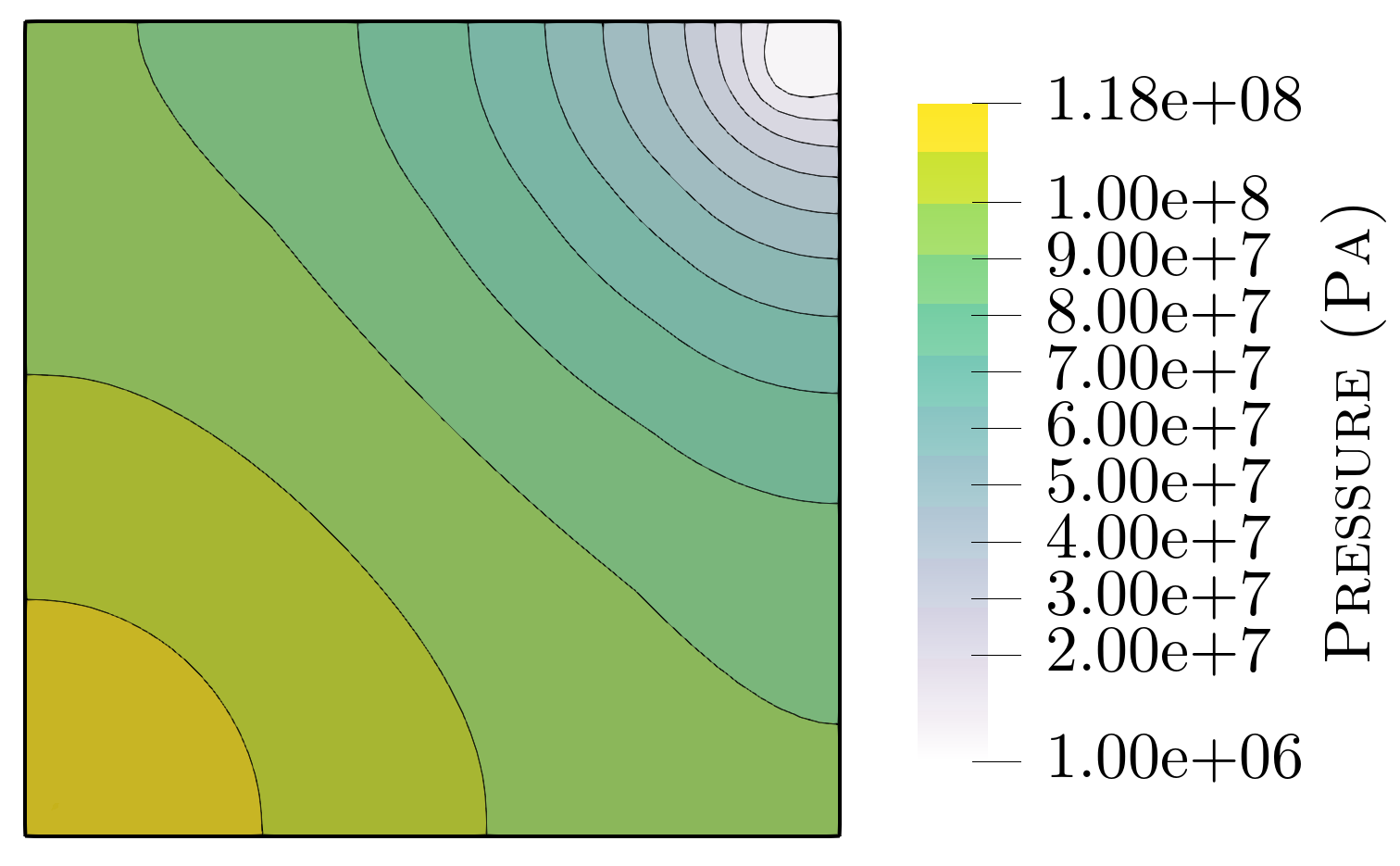}}
        \hspace{.05cm}
    \subfigure[DG+SL \label{fig:q5spres2}]{
        \includegraphics[clip,scale=0.095,trim=0 0cm 0cm 0]{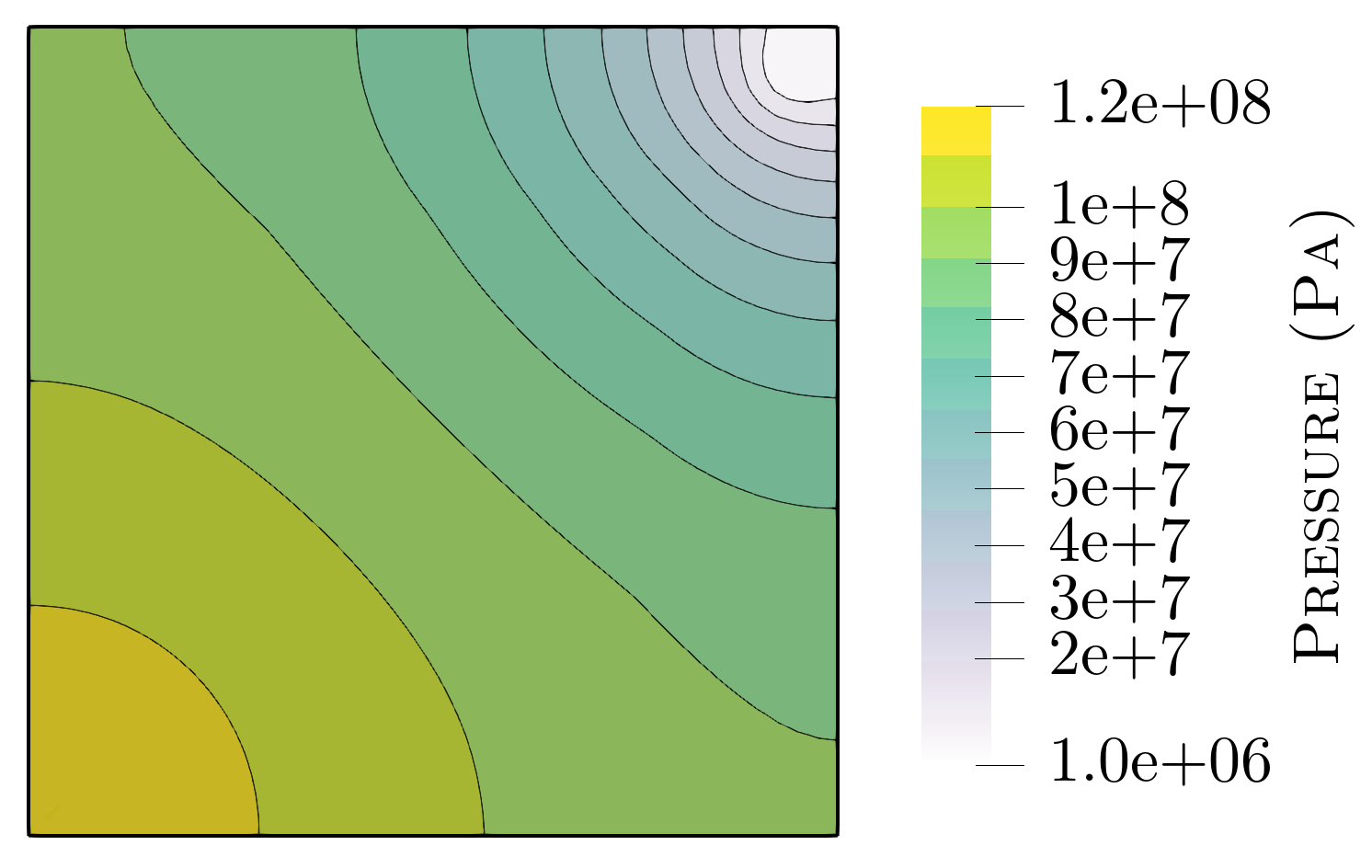}} 
        \hspace{.05cm}
    \subfigure[DG+FL+SL \label{fig:q5spres3}]{
        \includegraphics[clip,scale=0.095,trim=0 0cm 0cm 0]{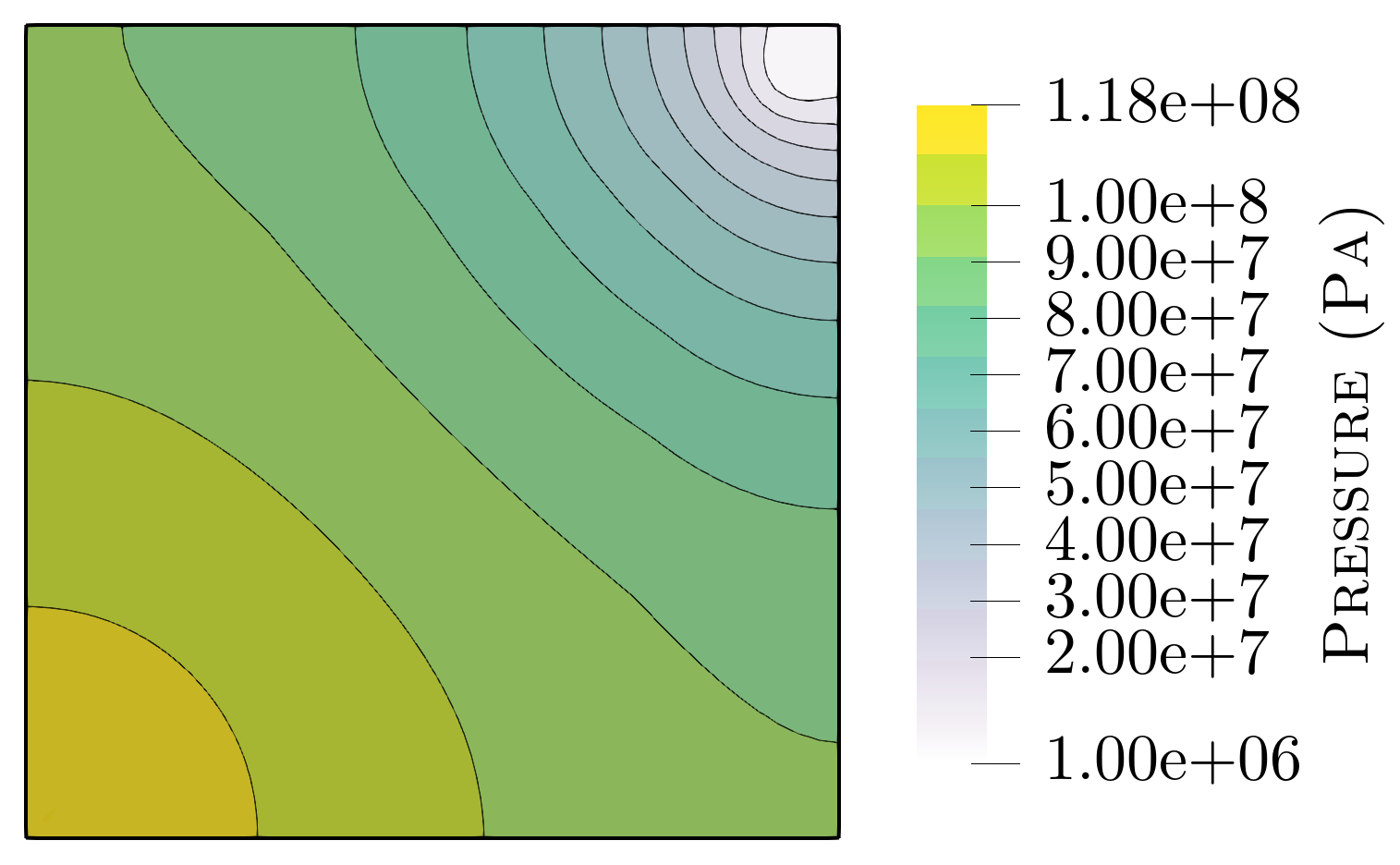}} \\
        \caption{\textsf{Quarter five-spot problem with homogeneous permeability:}
            This figure shows the wetting phase pressure at final time $t=10$ days using DG, DG+SL, and
            DG+FL+SL schemes. 
            All three cases yield similar approximations.
        \label{Fig:Q5_pres}}
\end{figure}
\begin{figure}
    \subfigure[DG \label{Fig:q5svel1}]{
        \includegraphics[clip,scale=0.095,trim=0 0cm 0cm
        0]{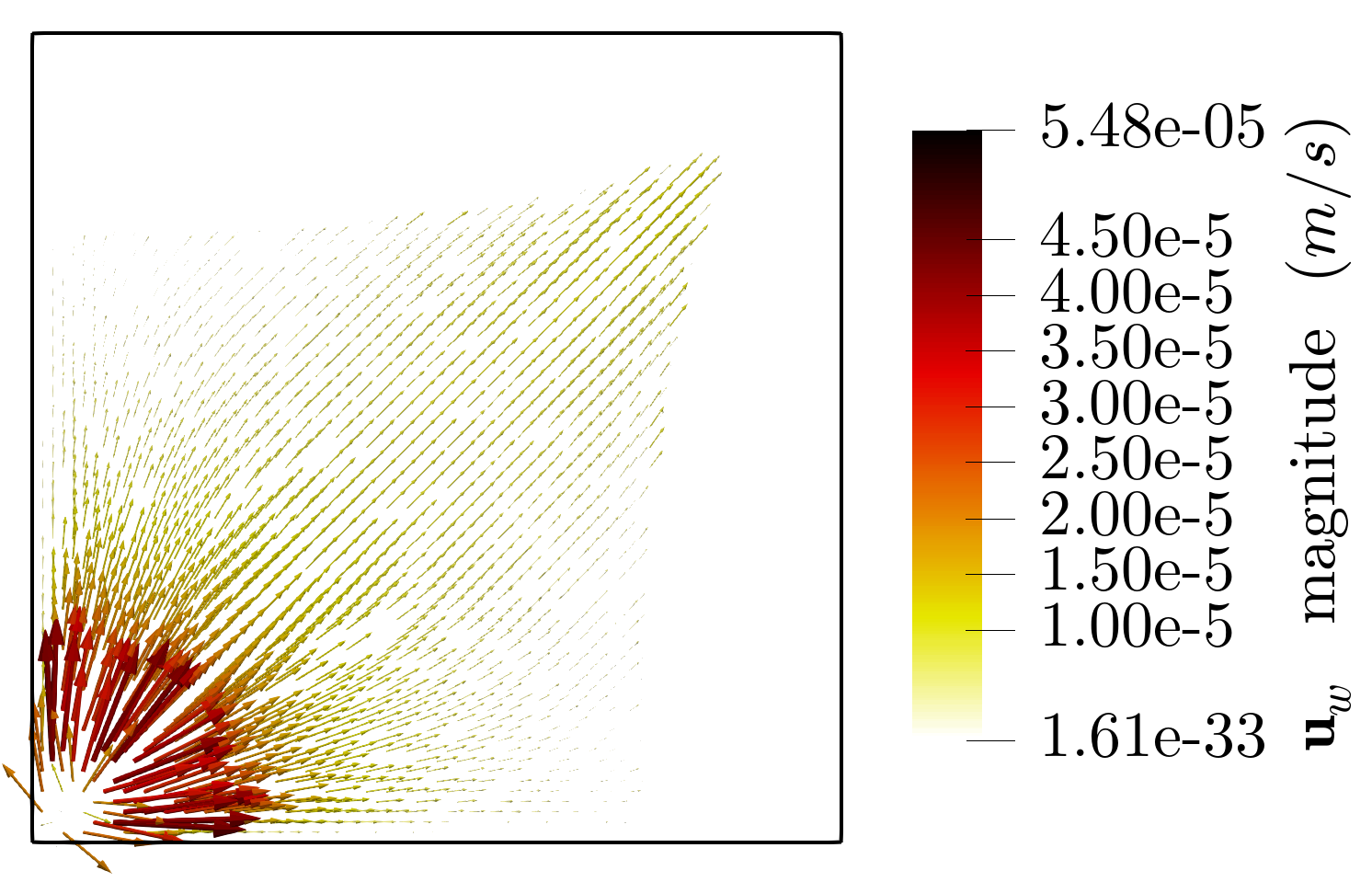}}
        \hspace{.05cm}
    \subfigure[DG+SL \label{fig:q5svel2}]{
        \includegraphics[clip,scale=0.095,trim=0 0cm 0cm 0]{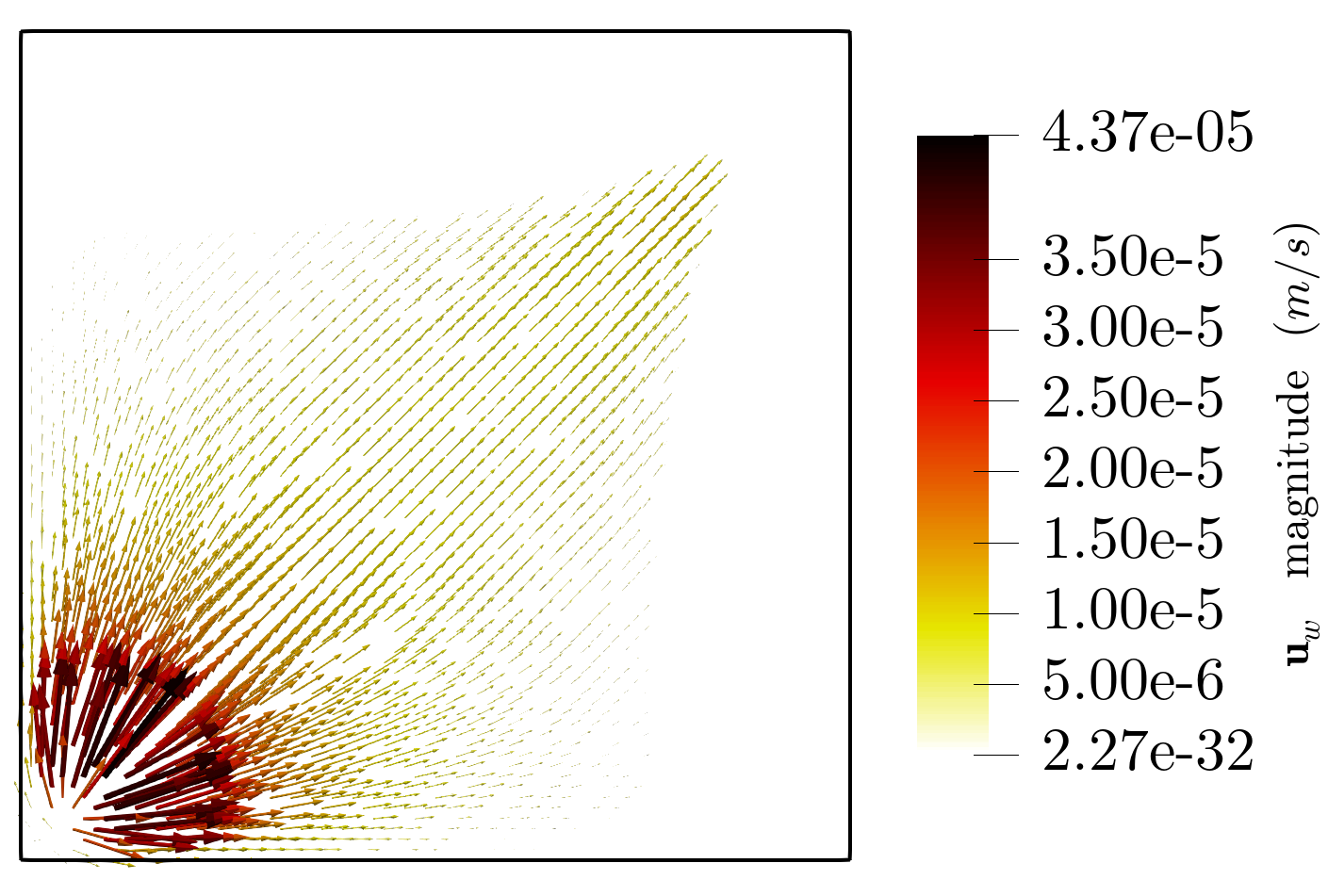}} 
        \hspace{.05cm}
    \subfigure[DG+FL+SL \label{fig:q5svel3}]{
        \includegraphics[clip,scale=0.095,trim=0 0cm 0cm 0]{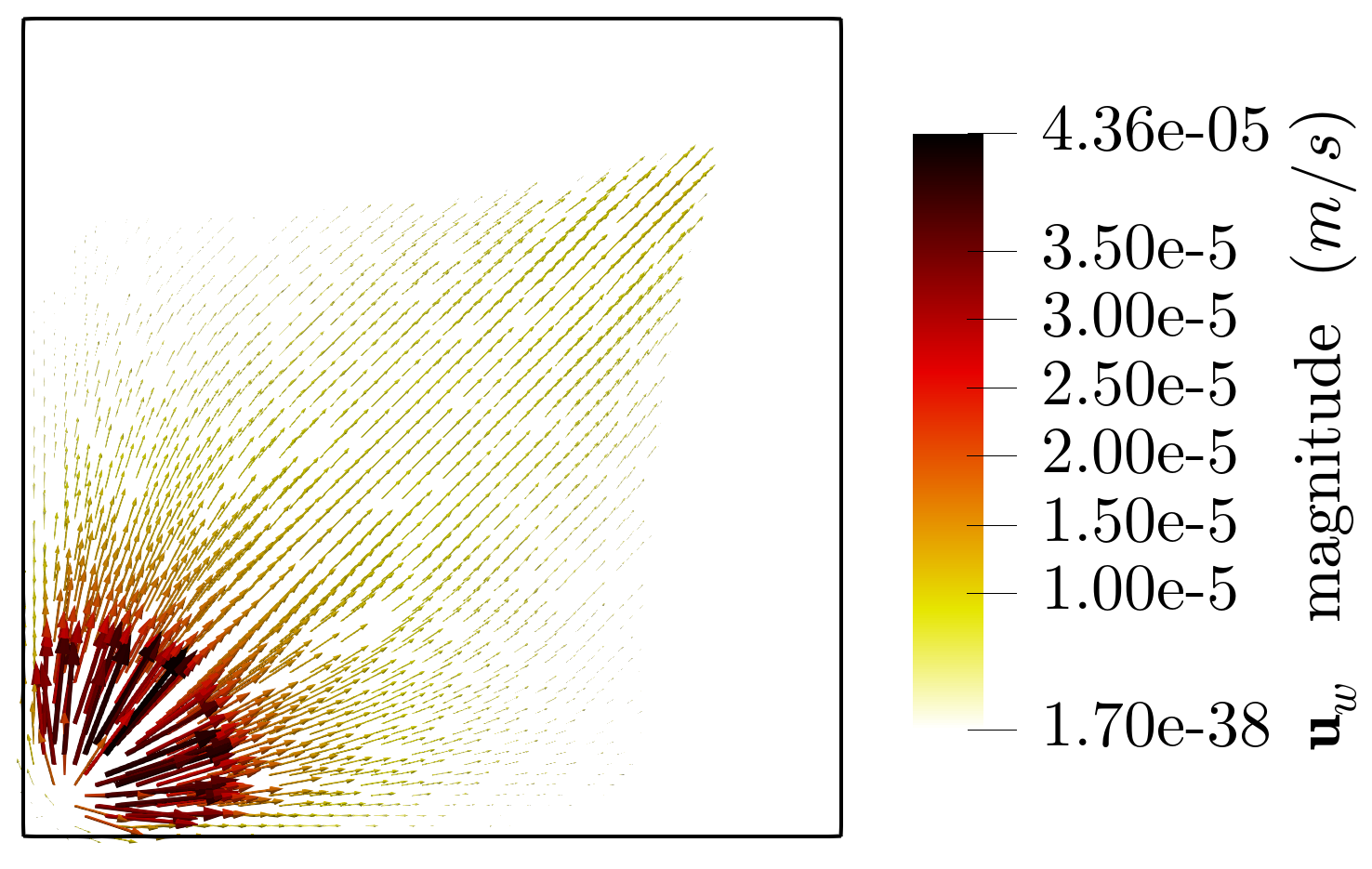}} \\
        \caption{\textsf{Quarter five-spot problem with homogeneous permeability:}
            This figure depicts the wetting phase velocity at time $t=10$ days using DG, DG+SL, and
            DG+FL+SL schemes. 
            All three cases yield similar approximations.
        \label{Fig:Q5_vel}}
\end{figure}
\begin{figure}
    \subfigure[DG \label{fig:q5sloc1}]{
        \includegraphics[clip,scale=0.105,trim=0 0.65cm 0cm 0]{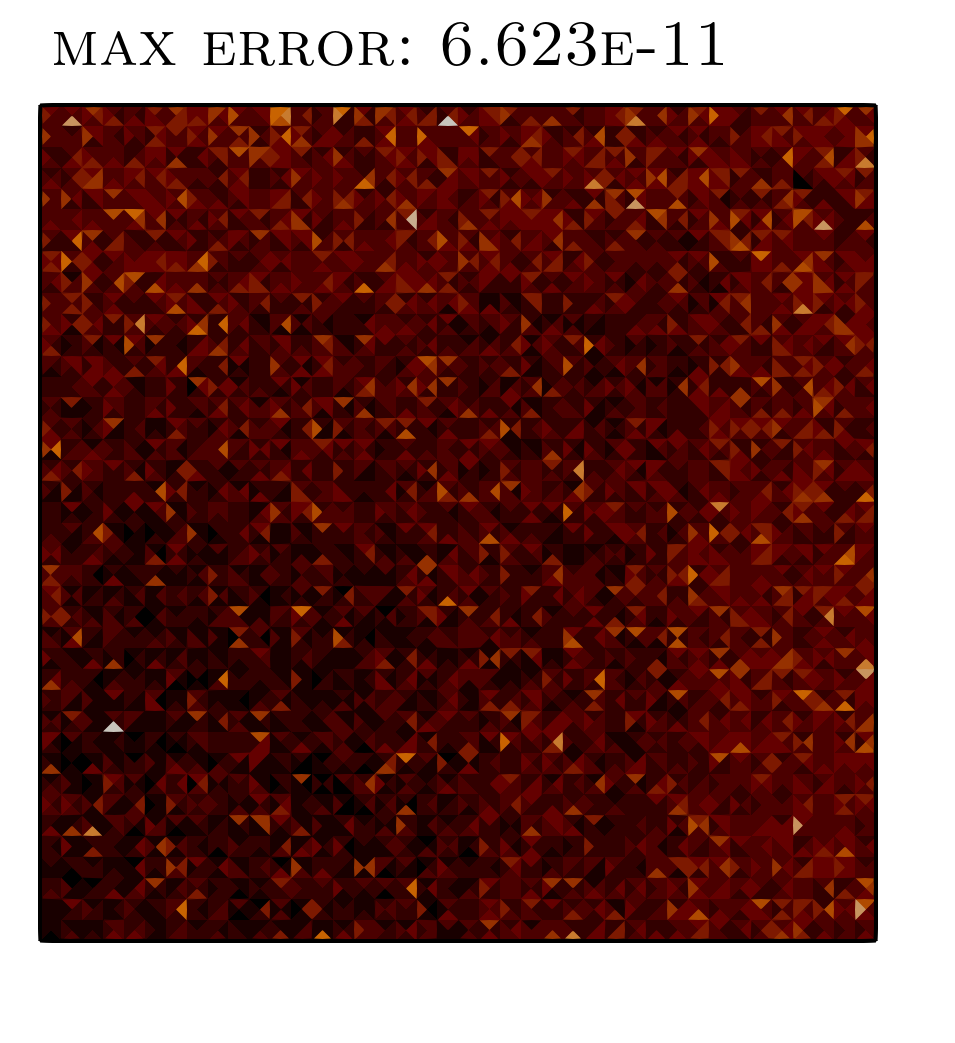}}
        \hspace{.05cm}
    \subfigure[DG+SL\label{fig:q5sloc2}]{
        \includegraphics[clip,scale=0.105,trim=0 0cm 0 0]{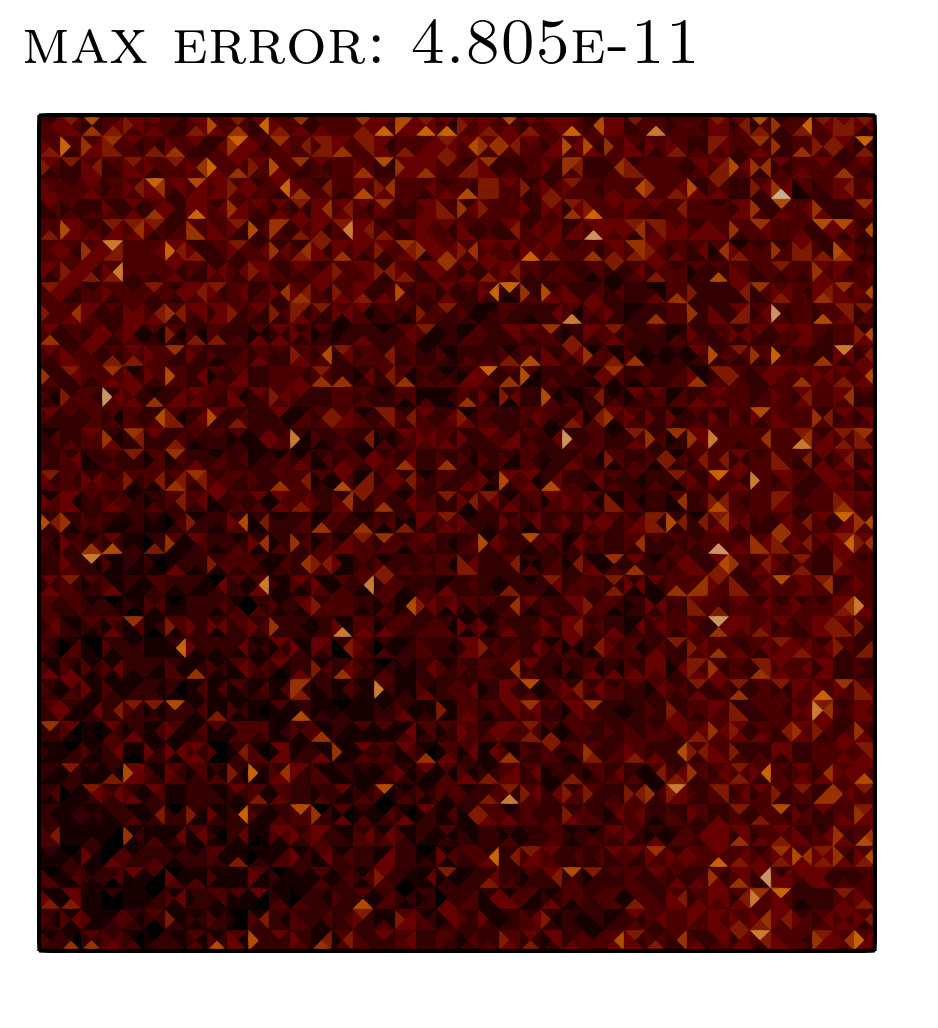}} 
        \hspace{.05cm}
    \subfigure[DG+FL+SL \label{fig:q5sloc3}]{
        \includegraphics[clip,scale=0.105,trim=0 0.2cm 0cm 0]{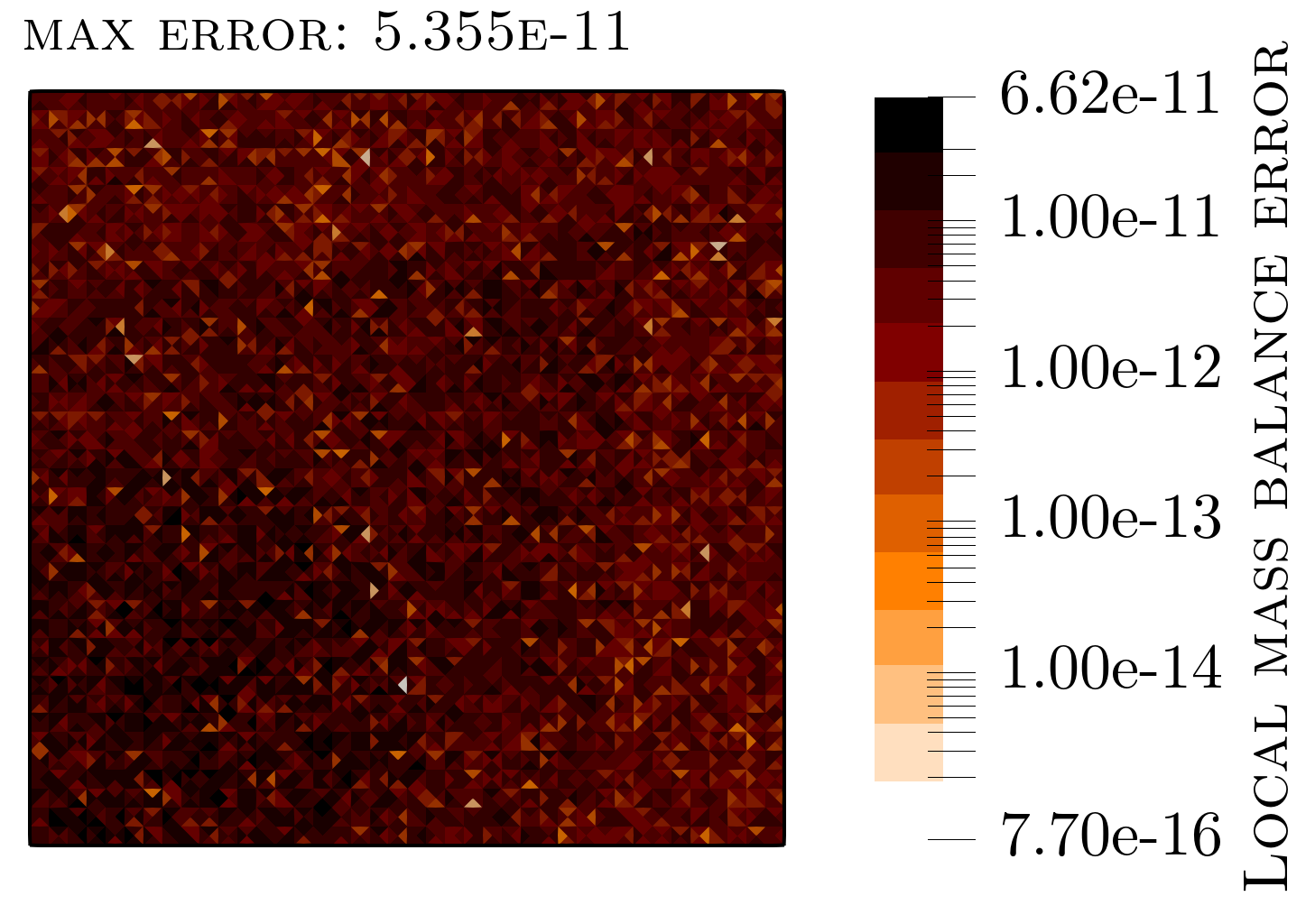}} \\
        \caption{\textsf{Local mass balance conservation for quarter five-spot problem:}                 
            This figure illustrates the local mass balance error at time $t=10$ days.
            No matter what scheme is used the errors always remain small (in the order of $10^{-11}$).  
        \label{Fig:Q5_LMB}}
\end{figure}

\subsubsection{Quarter five-spot problem with heterogeneous domain}
\label{sub:Q5_hetero}
We repeat the experiments in Section~\ref{sub:Q5_homogen} with heterogeneous medium of 
$\Omega=[0,1000]^2$ \si{\meter\squared}. 
The permeability fields are discontinuous and values vary over seven orders of magnitude. 
The permeability data is taken from two layers of the SPE 10 data-set 
\citep{christie2001tenth}; and are scaled to a crossed structured mesh of size $h=20$ \si{\meter} 
(see permeability fields in log-scale in Figure~\ref{Fig:SPE10_K}).  
We note that layer $13$ varies relatively smoothly, whereas layer $73$ contains well-defined 
channels, which form an additional challenge for any numerical method. 
We set viscosities to $\mu_w=5\times10^{-4}$ \si{\pascal\cdot\second} and 
$\mu_{\ell}=2\times10^{-3}$ \si{\pascal\cdot\second} and
invoke Brooks-Corey relative permeabilities as follows:
\begin{align}
    k_{rw}(s_e) = s_e^5,\quad
    k_{r\ell}(s_e) = (1-s_e)^2(1-s_e^5),\quad
    s_e = \frac{S-s_{rw}}{1-s_{rw}-s_{r\ell}}.
\end{align}
The production and injection wells of size $L_w=100$ \si{\meter} with $\bar{q}=\underline{q}=2.8\times10^{-5}$
are positioned at opposite corners such that $d_w=70$ \si{\meter} (see Figure \ref{Fig:Q5_schematic}). 
The time step is $\tau = 4.17\times10^{-3}$ days and the final time is $T=1.375$ days.

We apply our proposed DG scheme with both flux and slope limiters to these porous media.
Figure~\ref{Fig:SPE10_sat} displays the wetting phase saturation contours at different times 
($t=0.417$, $0.83$, $1.375$ days) for both layers. As expected the wetting phase floods the domain from the injection 
well to the production well while avoiding low permeable regions. Because of the location of channels in 
layer $73$, the wetting phase has reached the production well at time $t=1.375$ days whereas this is not the 
case for layer $13$. We also observe that the saturation satisfies the maximum principle. 
Figure~\ref{Fig:SPE10_vel} shows the magnitude of the wetting phase velocity at the same times.  
The effect of the heterogeneities can be seen in the velocity fields. 
\begin{figure}
    \subfigure[layer 13 \label{fig:q5shetperm13}]{
        \includegraphics[clip,scale=0.12,trim=0 0cm 0cm
        0]{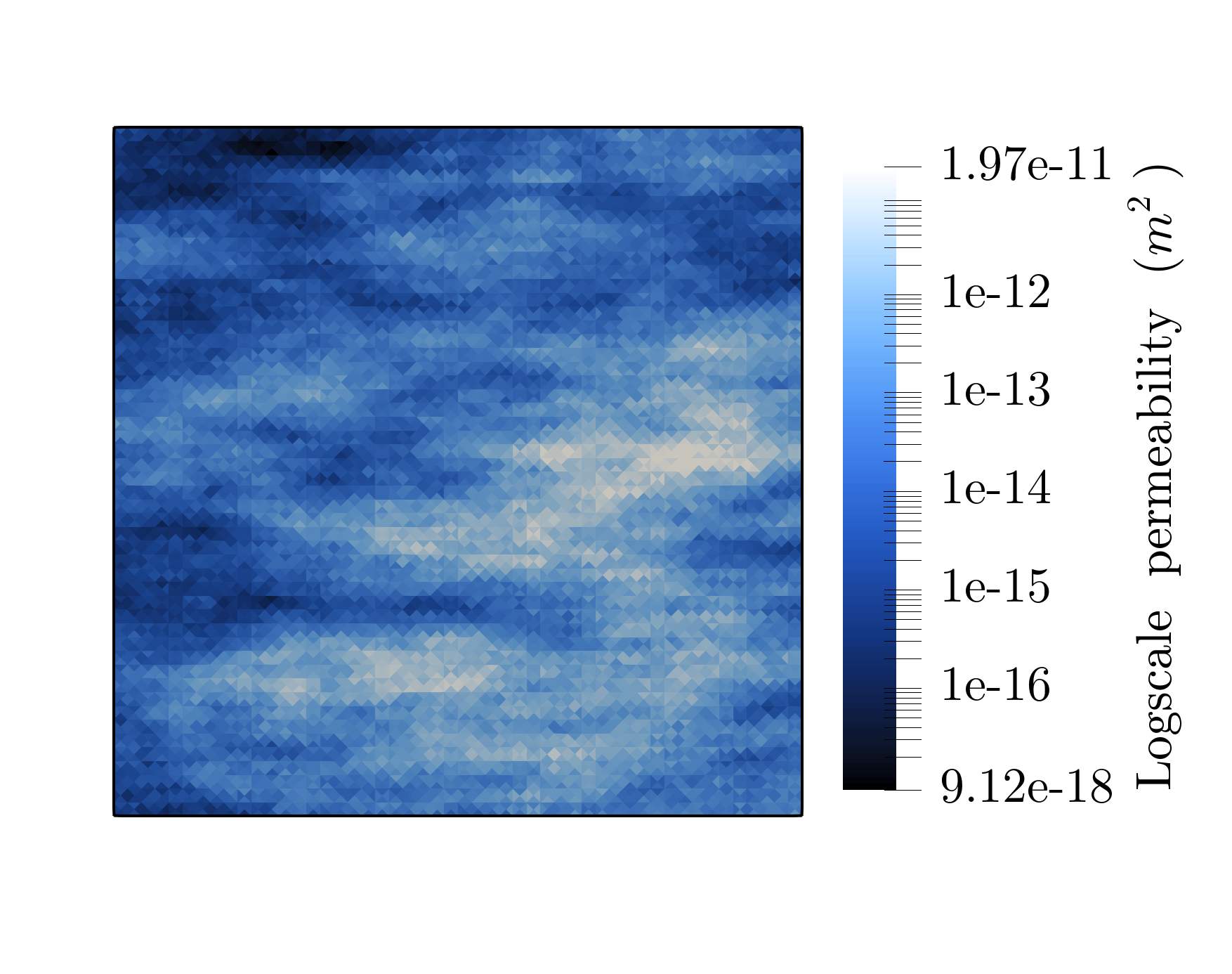}}
        \hspace{.25cm}
    \subfigure[layer 73 \label{fig:q5shetperm73}]{
        \includegraphics[clip,scale=0.12,trim=0 0cm 0cm 0]{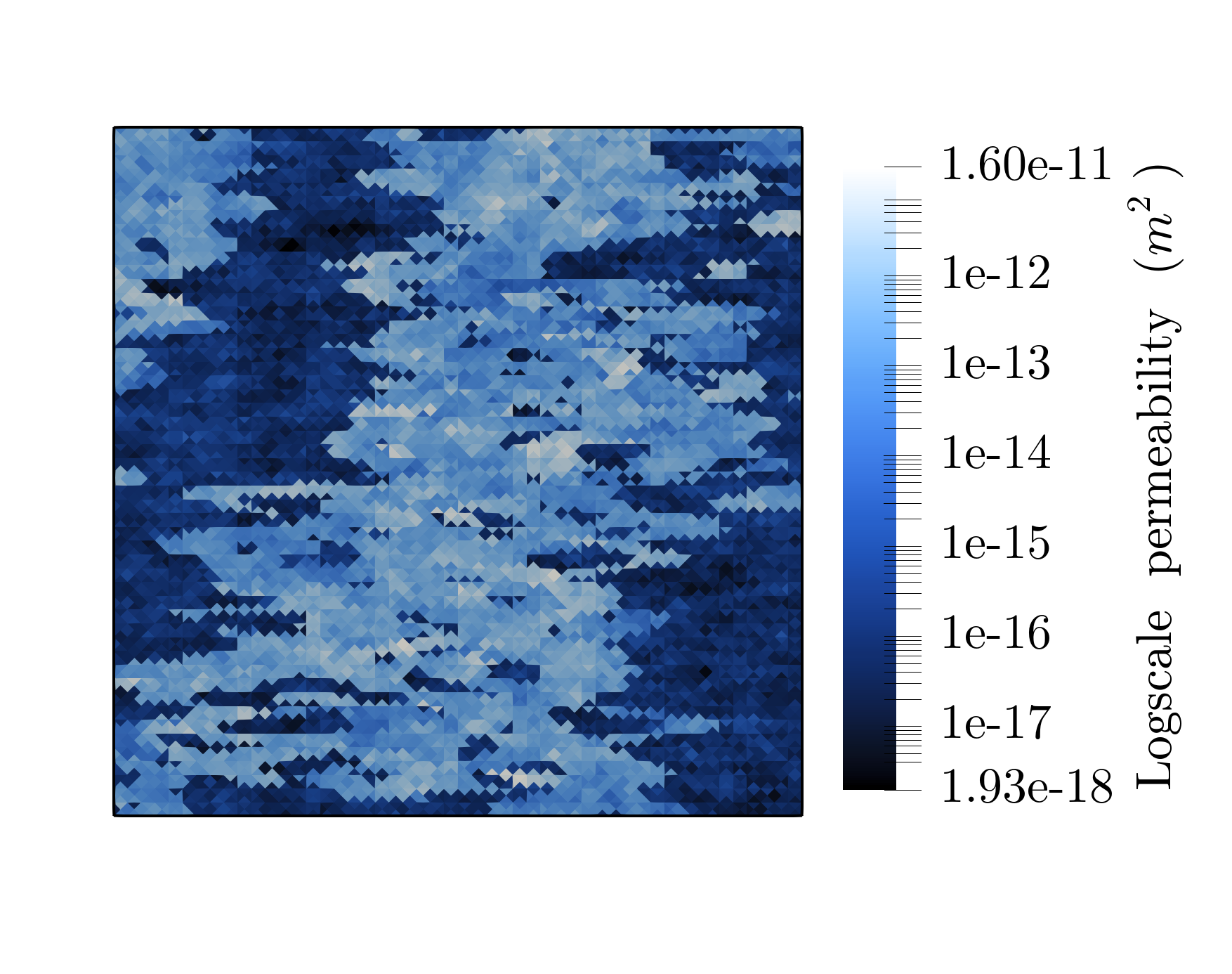}} \\
        \caption{\textsf{Quarter five-spot problem with heterogeneous permeability:}
            This figure illustrates the permeability fields adopted from two horizontal layers of SPE10 
            model 2 data-set. Layer $13$ is taken from relatively smooth Tarbert formation, whereas layer $73$ 
            is taken from a highly varying Upper-Ness formation. Values are presented in logarithmic scale. 
        \label{Fig:SPE10_K}}
\end{figure}
\begin{figure}
    \subfigure[Layer 13; $t=0.417$ days \label{fig:q5shet1}]{
        \includegraphics[clip,scale=0.15,trim=0 0cm 0cm
        0]{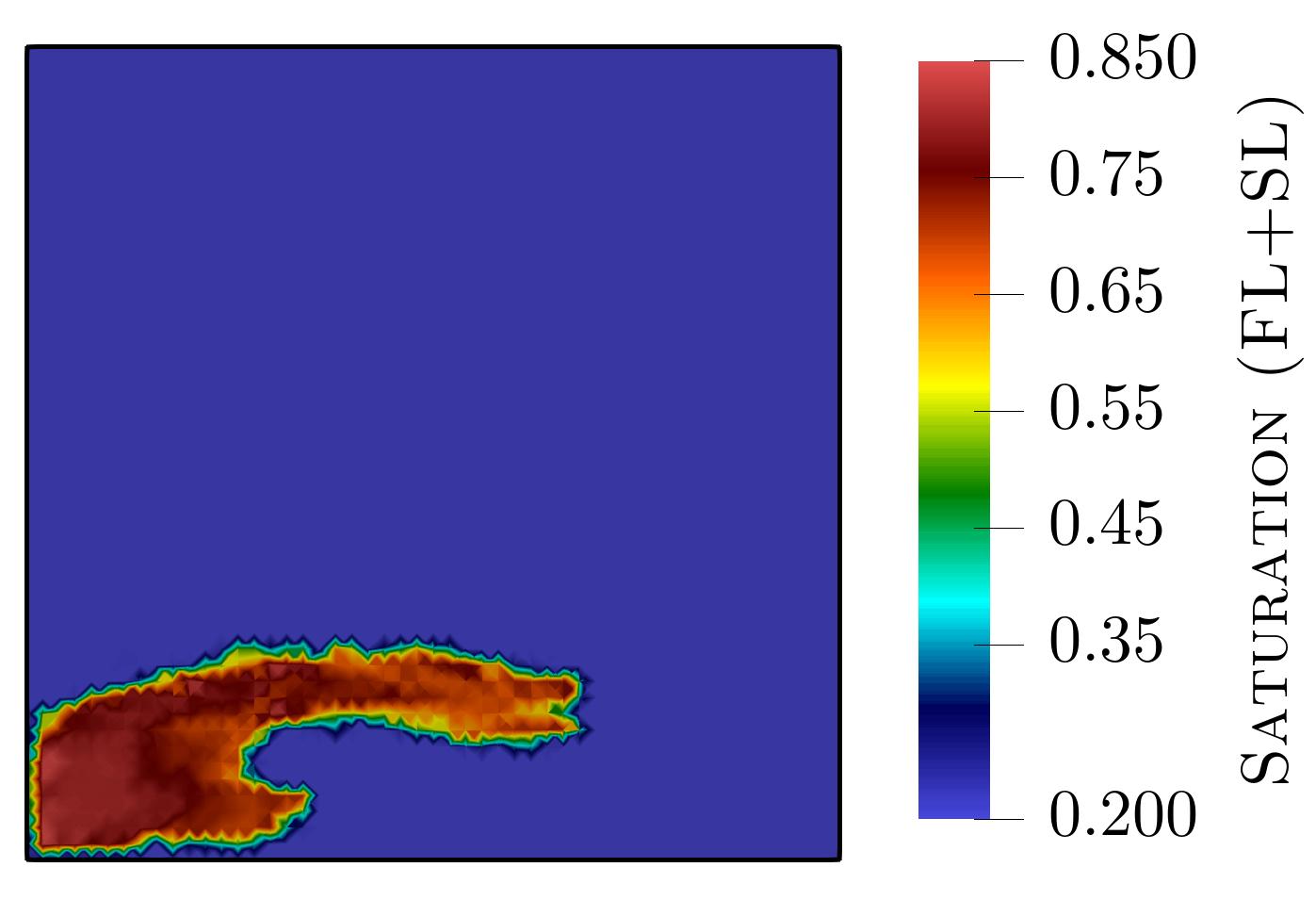}}
        \hspace{.1cm}
    \subfigure[Layer 73; $t=0.417$ days \label{fig:q5shet2}]{
        \includegraphics[clip,scale=0.15,trim=0 0cm 0cm 0]{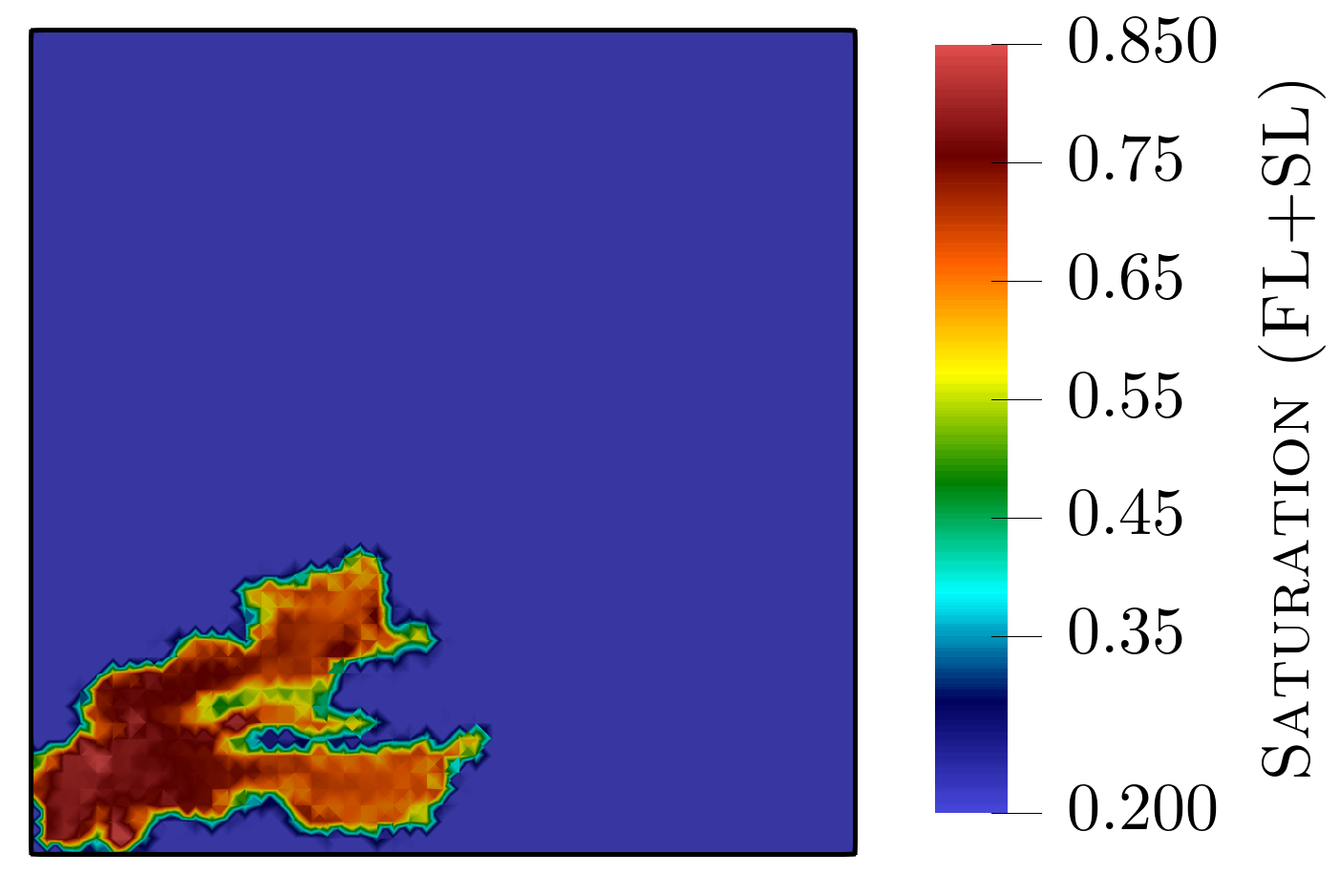}} \\
    \subfigure[Layer 13; $t=0.83$ days \label{fig:q5shet3}]{
        \includegraphics[clip,scale=0.15,trim=0 0cm 0cm
        0]{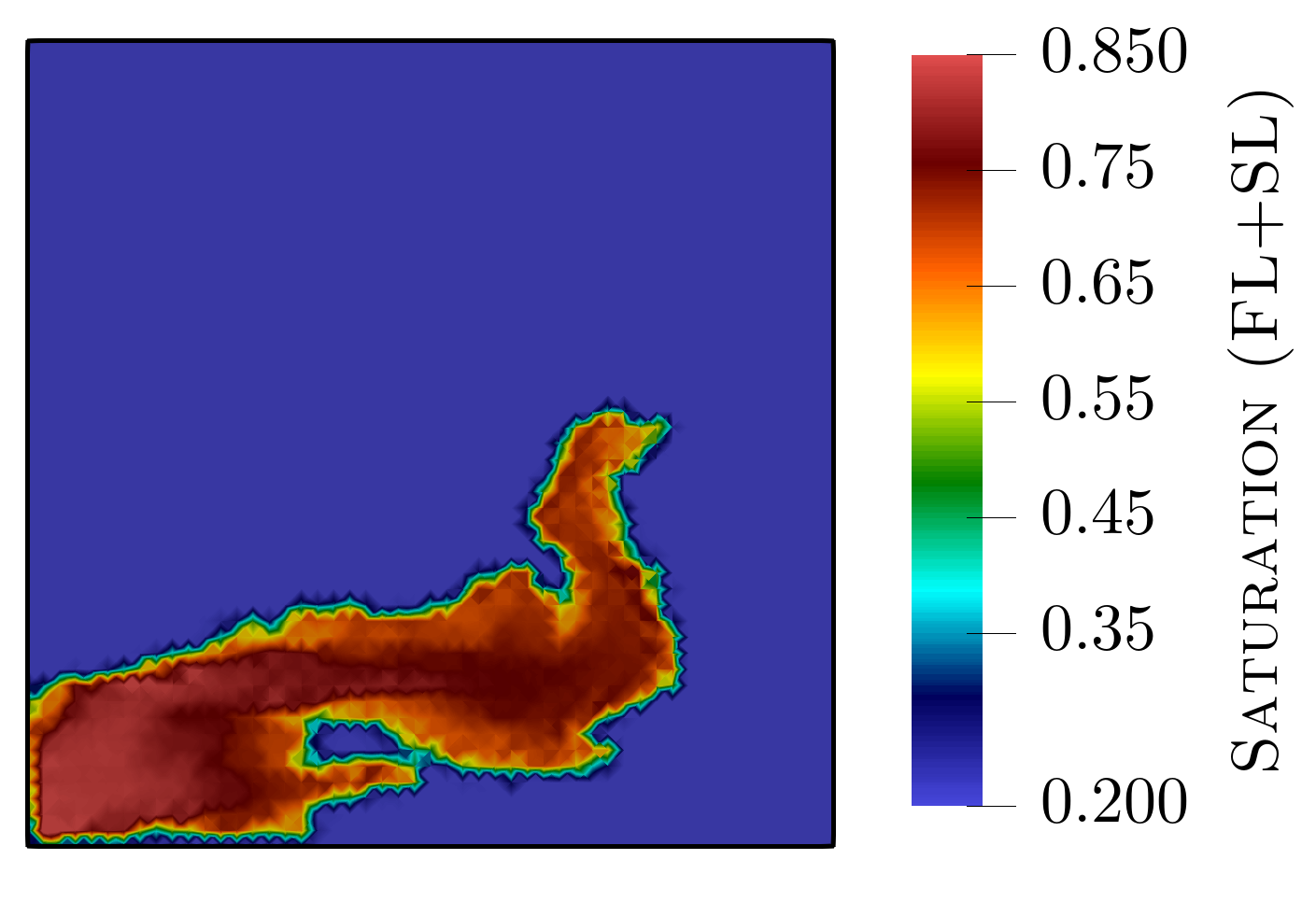}}
        \hspace{.1cm}
    \subfigure[Layer 73; $t=0.83$ days \label{fig:q5shet4}]{
        \includegraphics[clip,scale=0.15,trim=0 0cm 0cm 0]{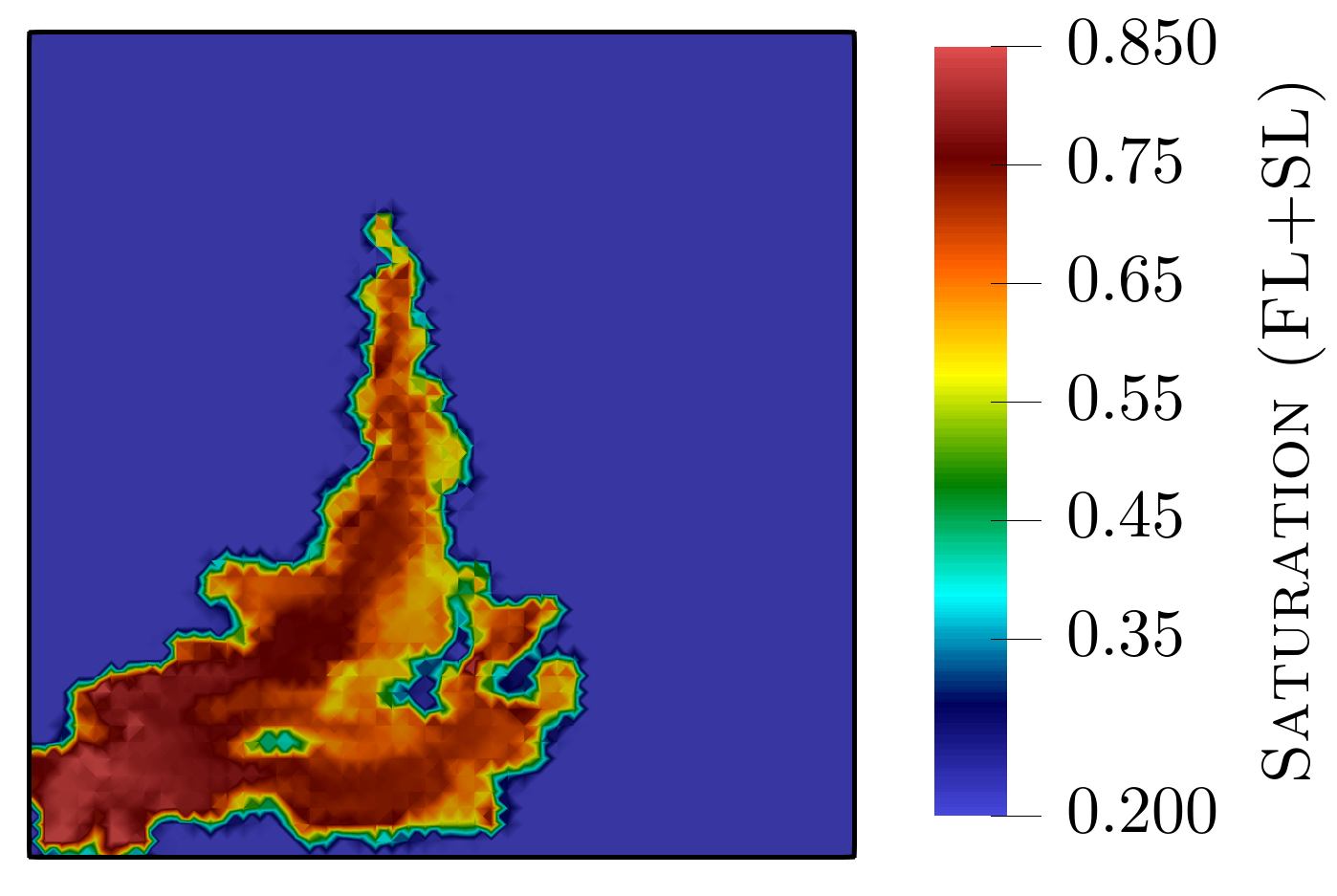}} \\
    \subfigure[Layer 13; $t=1.375$ days \label{Fig:q5shet5}]{
        \includegraphics[clip,scale=0.15,trim=0 0cm 0cm
        0]{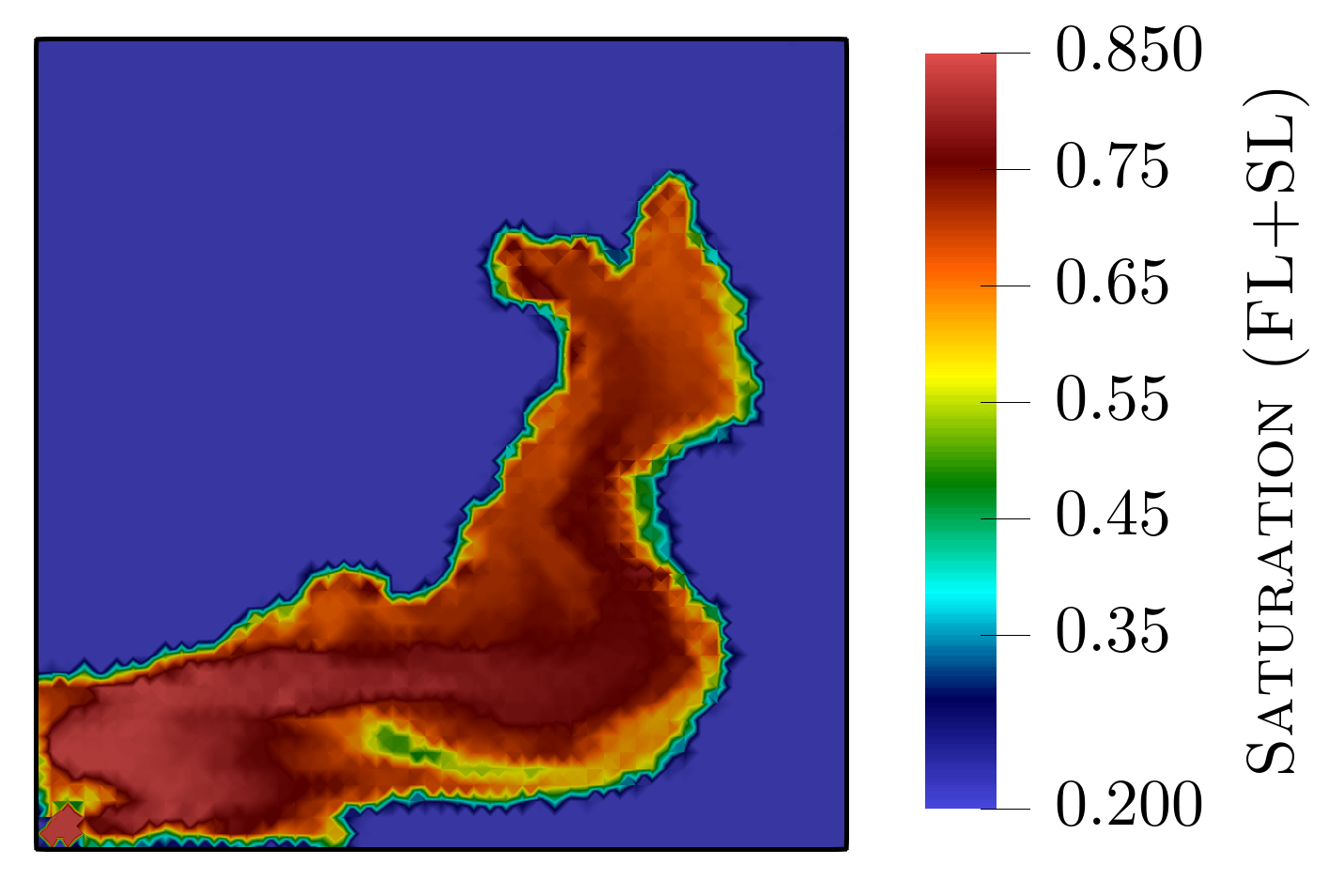}}
        \hspace{.1cm}
    \subfigure[Layer 73; $t=1.375$ days \label{fig:q5shet6}]{
        \includegraphics[clip,scale=0.15,trim=0 0cm 0cm 0]{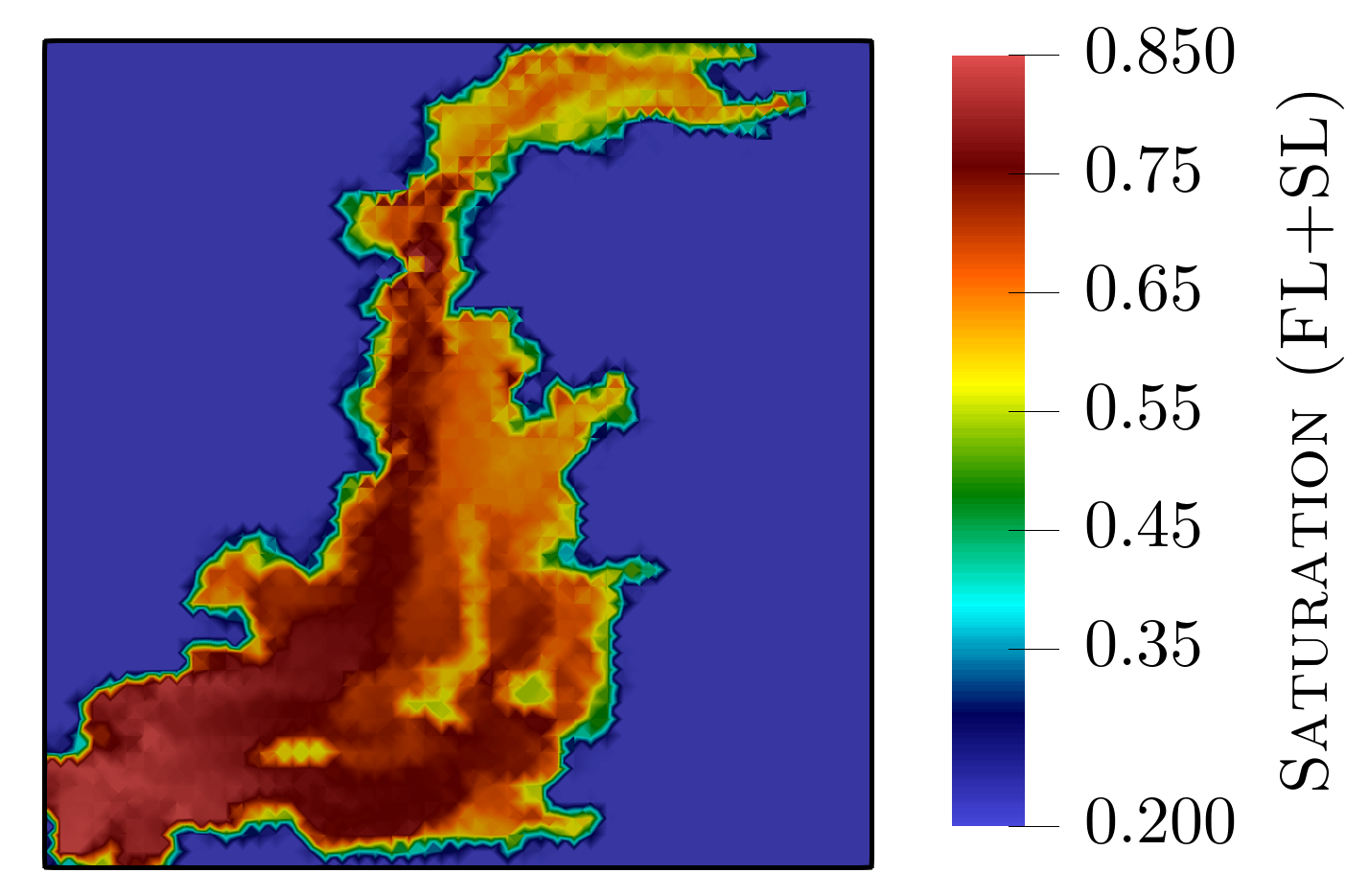}} \\
        \caption{\textsf{Quarter five-spot problem with heterogeneous permeability:} 
            This figure shows the evolution of the saturation obtained using DG+FL+SL scheme for 
            layer $13$ (left) and layer $73$ (right). 
            For both cases, the wetting phase moves toward the production well by sweeping the regions with 
            highest permeability values. Another inference is that proposed limiters yield physical
            values of saturation, without any overshoots and undershoots, even for domains with
            permeabilities that vary over several orders of magnitudes.
        \label{Fig:SPE10_sat}}
\end{figure}
\begin{figure}
    \subfigure[Layer 13; $t=0.417$ days \label{fig:q5shetvel1}]{
        \includegraphics[clip,scale=0.13,trim=0 0cm 0cm
        0]{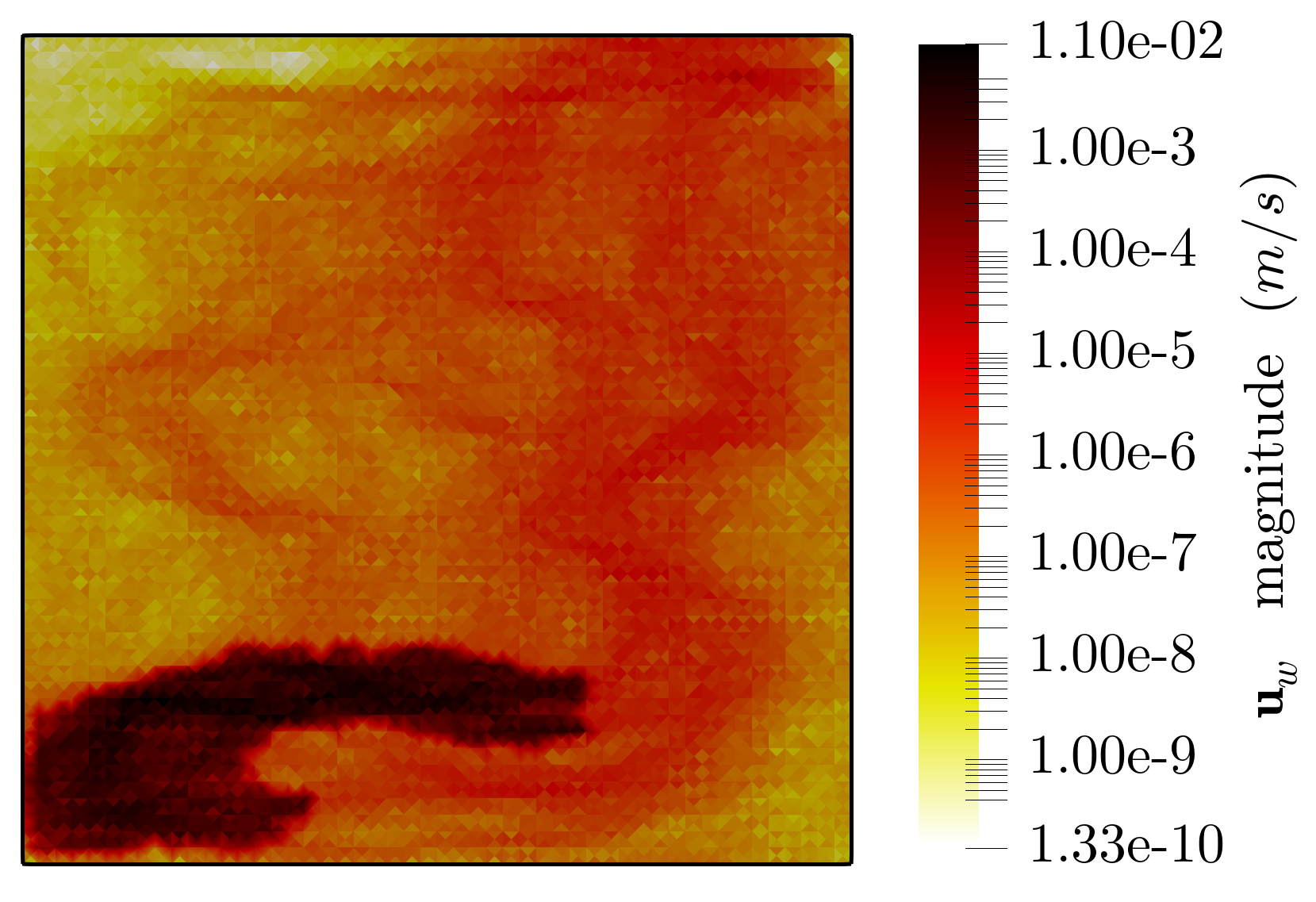}}
        \hspace{.1cm}
    \subfigure[Layer 73; $t=0.417$ days \label{fig:q5shetvel2}]{
        \includegraphics[clip,scale=0.135,trim=0 0cm 0cm 0]{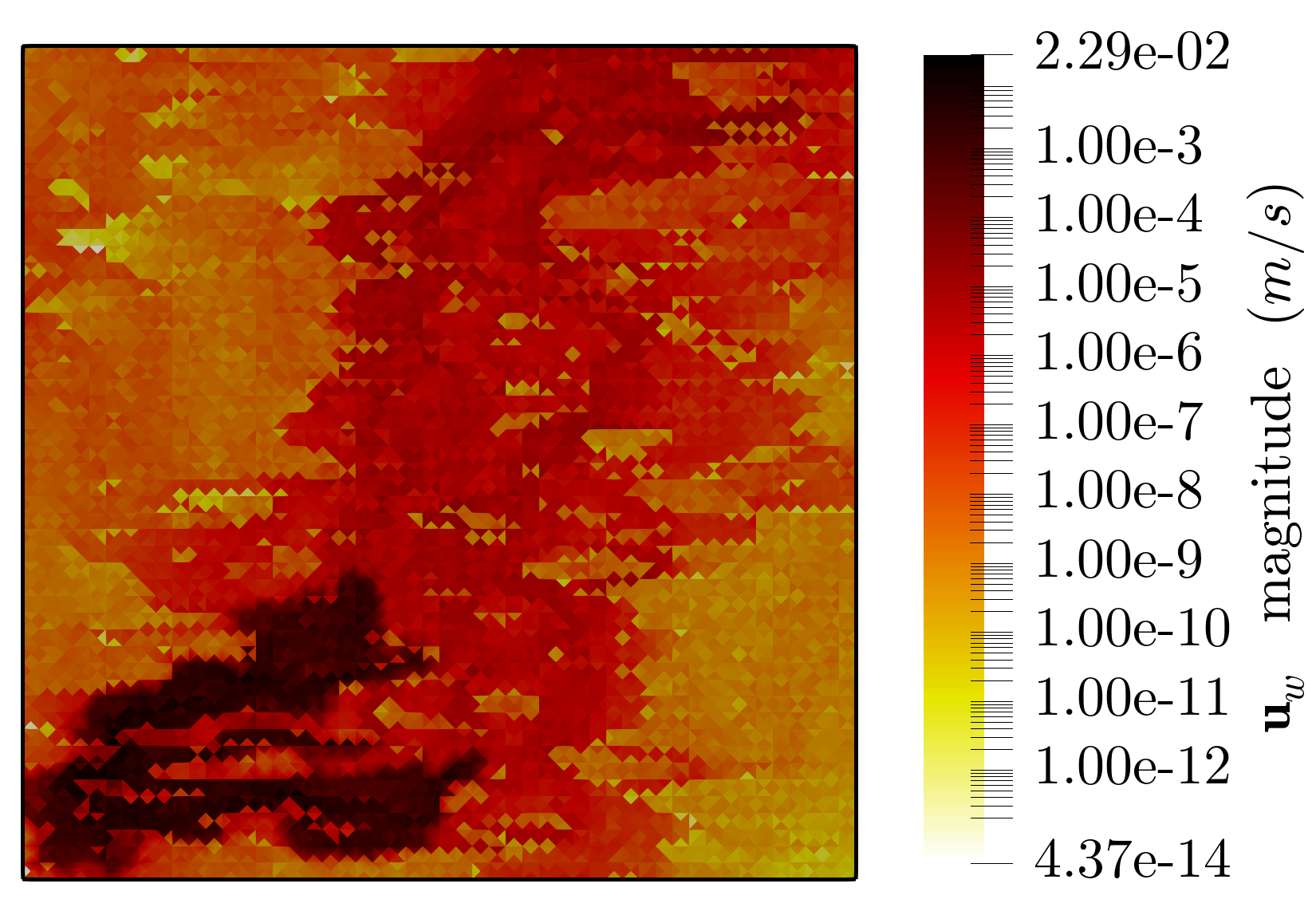}} \\
    \subfigure[Layer 13; $t=0.83$ days \label{fig:q5shetvel3}]{
        \includegraphics[clip,scale=0.135,trim=0 0cm 0cm
        0]{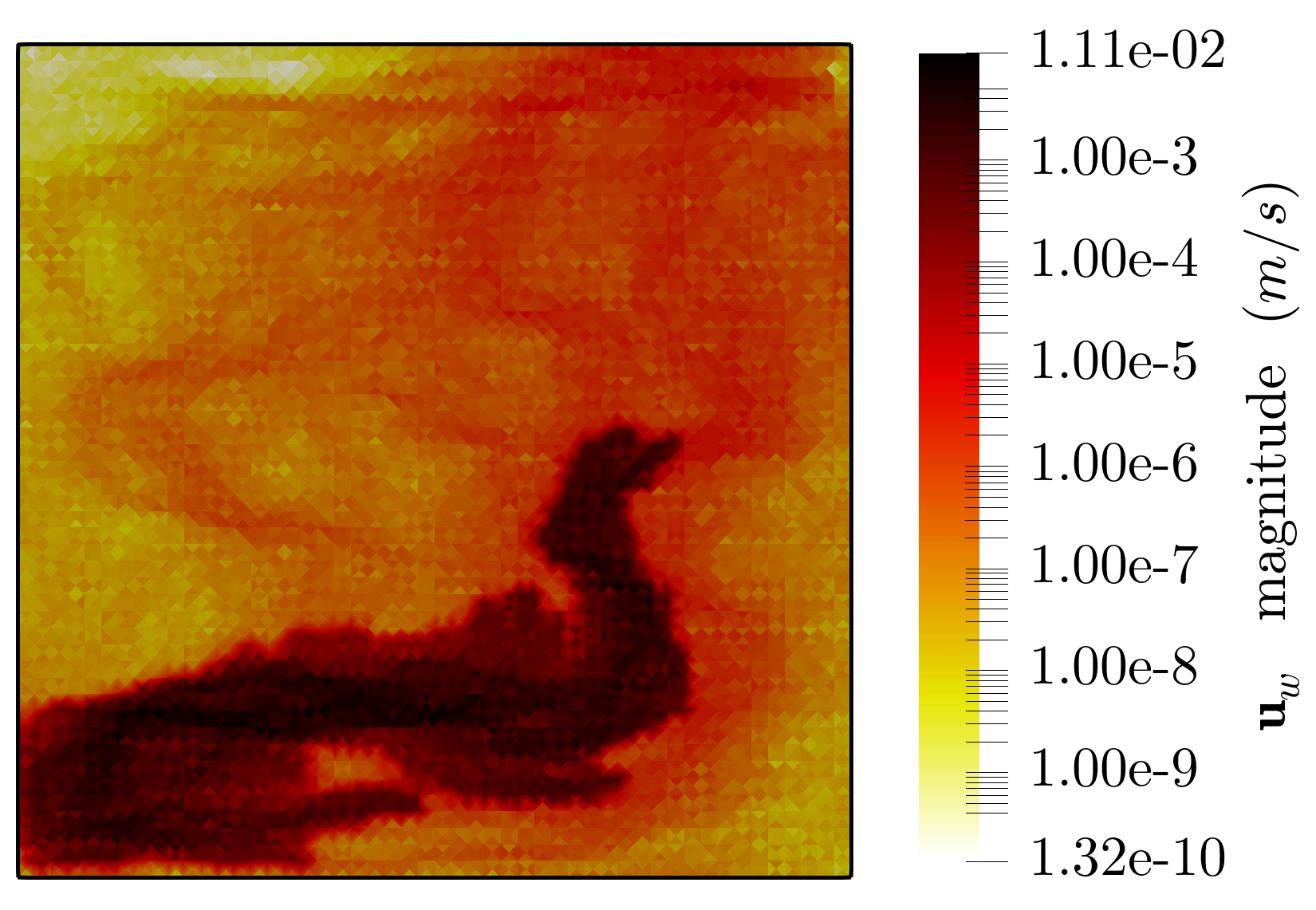}}
        \hspace{.1cm}
    \subfigure[Layer 73; $t=0.83$ days \label{fig:q5shetvel4}]{
        \includegraphics[clip,scale=0.135,trim=0 0cm 0cm 0]{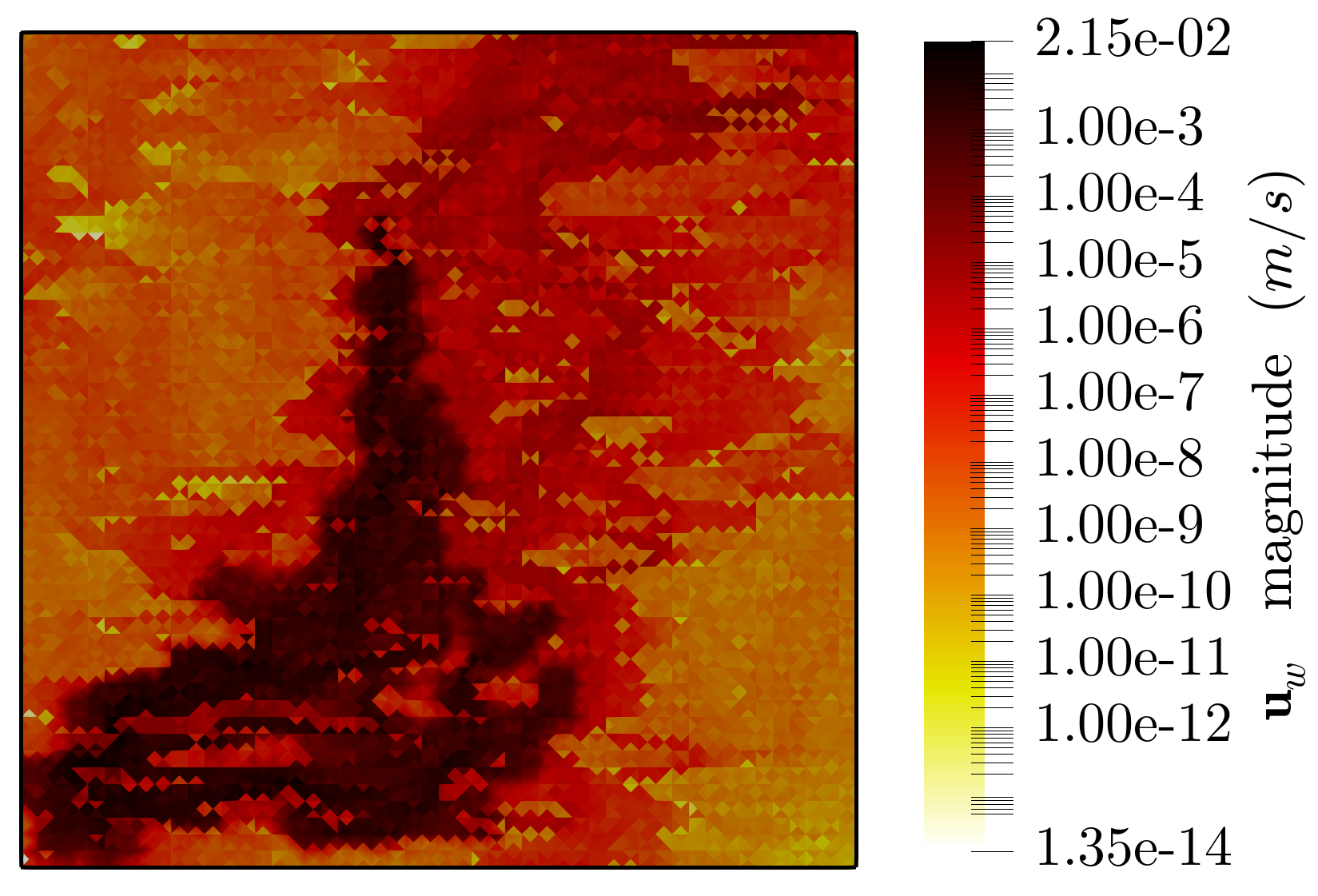}} \\
    \subfigure[Layer 13; $t=1.375$ days \label{fig:q5shetvel5}]{
        \includegraphics[clip,scale=0.135,trim=0 0cm 0cm
        0]{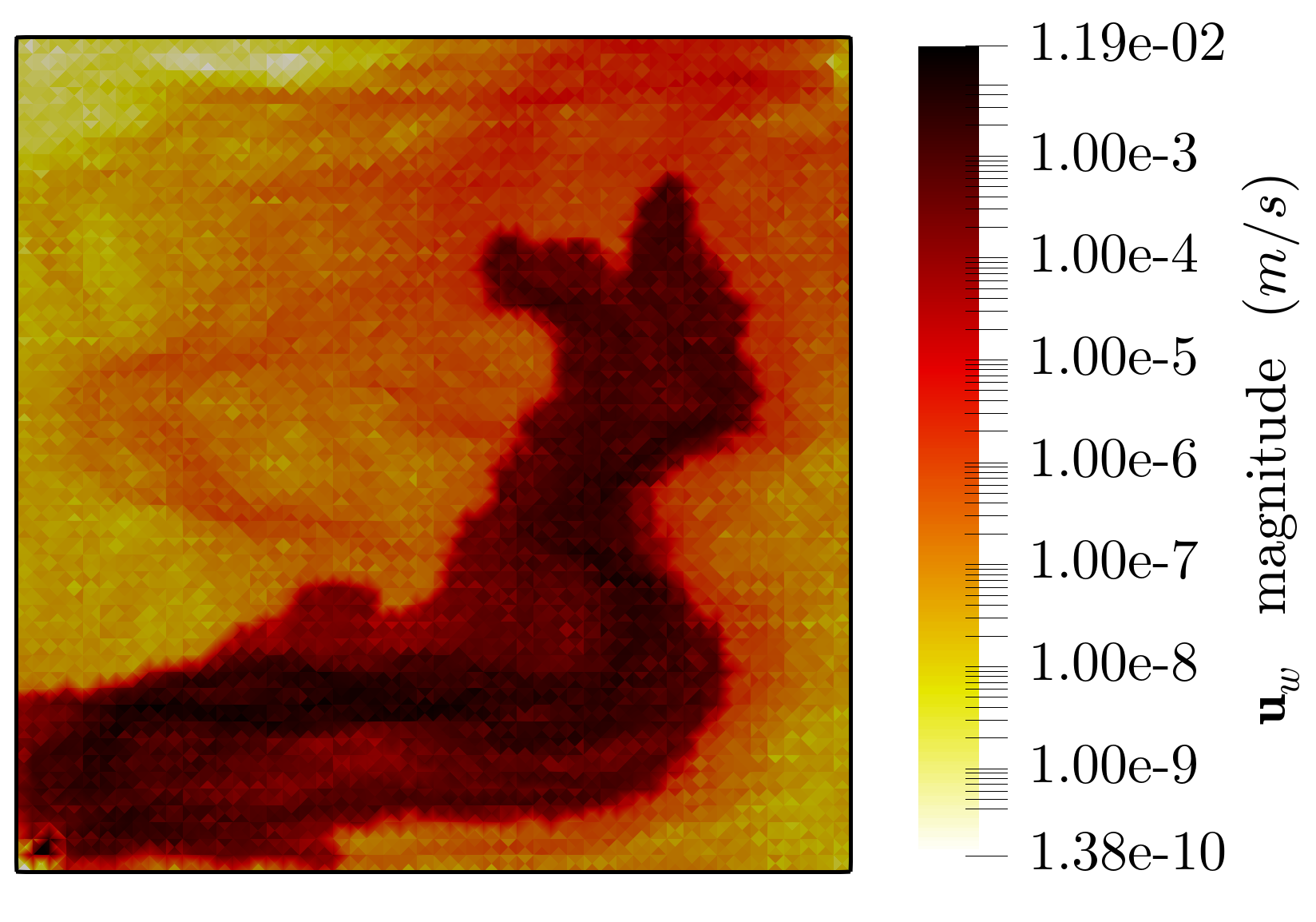}}
        \hspace{.1cm}
    \subfigure[Layer 73; $t=1.375$ days \label{fig:q5shetvel6}]{
        \includegraphics[clip,scale=0.135,trim=0 0cm 0cm 0]{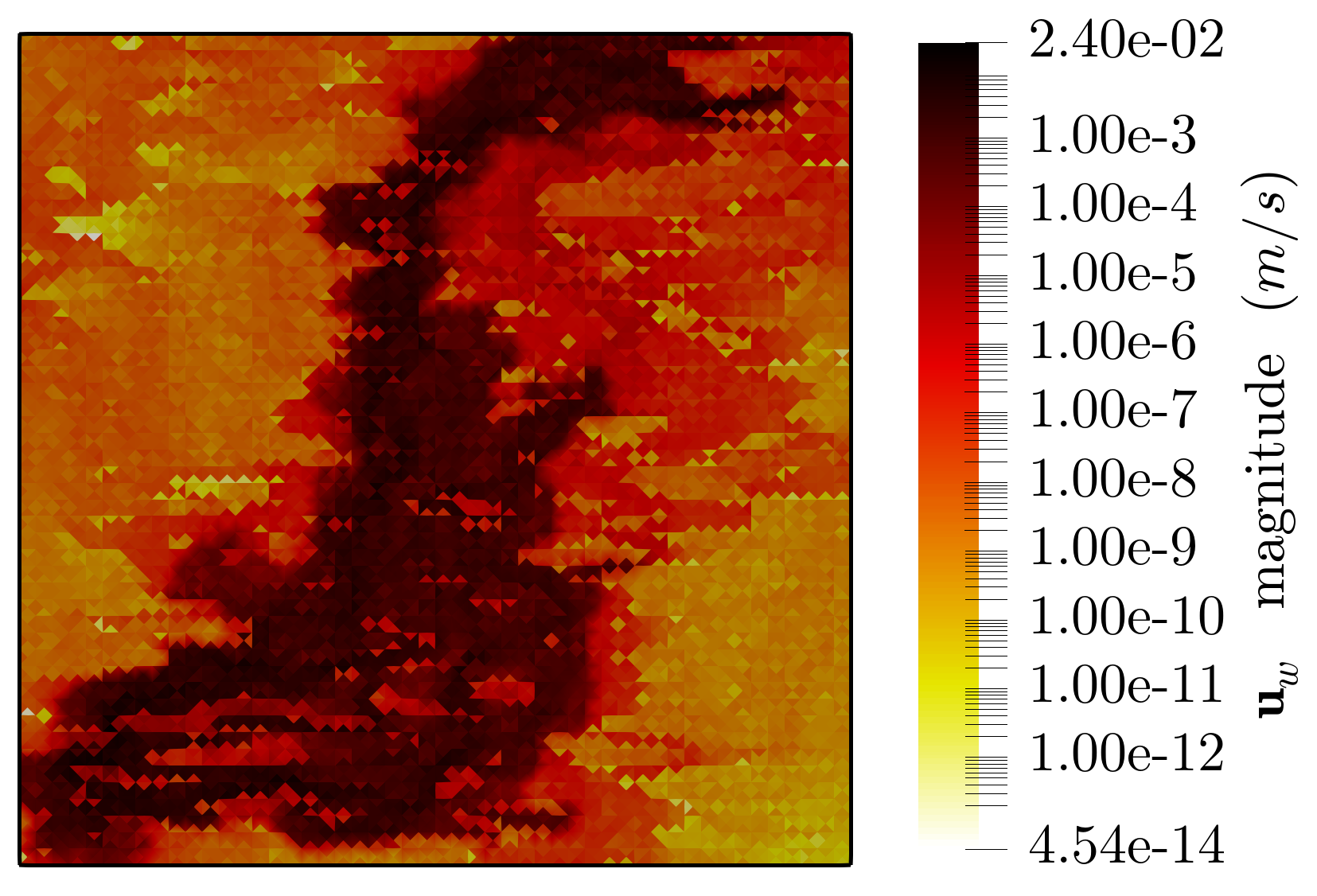}} \\
        \caption{\textsf{Quarter five-spot problem with heterogeneous permeability:}
        This figure shows the magnitude of wetting phase velocities obtained using DG+FL+SL for layer $13$
        (left) and layer $73$ (right).
        The effect of heterogeneities is reflected in the velocity fields.
        \label{Fig:SPE10_vel}}
\end{figure}

\subsection{Effect of gravity}
\label{Sub:gravity}
In this section, we examine the success of our limiting scheme in the presence of gravity field
and then study the impact of gravity on the pressure-driven flows and quarter five-spot problems.
The ratio of gravitational to viscous forces can be represented as a gravity number, Gr.
This dimensionless parameter depends on the difference between phase densities; and
following the work of \cite{riaz2006numerical,hassanizadeh2005upscaling,tchelepi2006numerical},
can be defined as follows:
\begin{align}
    \mbox{Gr} = \frac{K(\rho_w-\rho_{\ell})g}{\mu_w U},
\end{align}
where $U$ is the characteristic magnitude of velocity. 

\subsubsection{Pressure-driven flows}
The domain $\Omega = [0,200]\times[0,100]$ \si{\meter\squared} is partitioned into a crossed mesh with 
$7200$ triangular elements. The viscosities are $\mu_w=2.5\times10^{-4}$ \si{\pascal\cdot\second} and 
$\mu_{\ell}=5\times10^{-3}$ \si{\pascal\cdot\second}. Here, the characteristic velocity is estimated 
to be $U \approx 0.1$ \si{\meter\per\second} (using $U\approx -K(P\vert_{x=200}-P\vert_{x=0})/
\mu_w L$).
The wetting phase density is $\rho_w = 1000$~kg/m\textsuperscript{3} and the non-wetting phase density 
takes three different values $\rho_{\ell} = 925, 850, 600$~kg/m\textsuperscript{3}, which yields three values 
for the gravity number Gr$= 0.3, 0.6$ and $1.6$ respectively. Other parameters and Dirichlet boundary 
conditions are the same as in Section~\ref{sub:2Dpatch_homogen}. 
The time step is $\tau = 0.6$ \si{\second} and the final time is $T=600$ \si{\second}.
The proposed DG scheme with flux and slope limiters is applied and the penalty parameter is set to 
$\sigma=1000$.
Figure~\ref{Fig:2Dpatch_gravity_sat} shows the saturation contours at the time $t=600$ \si{\second}. 
As the gravity number increases, the wetting phase saturation, which is the heaviest, deposits more 
and more at the bottom of the domain; and the narrow gravity tongue along the bottom edge becomes more 
pronounced.
It should be also noted that similar to earlier problems, the limiting scheme exhibits satisfactory results 
with respect the maximum principle. This means that for all three cases, solutions always remain between $0.2$ 
and $0.85$. 
Pressure contours and velocity fields are displayed in Figure~\ref{Fig:2Dpatch_gravity_pres} and \ref{Fig:2Dpatch_gravity_vel}. Both show the impact of gravity on the solutions.
\begin{figure}
    \subfigure[Gr $= 0.3$\label{fig:pdfg1}]{
        \includegraphics[clip,scale=0.095,trim=0 0cm 0cm
        0]{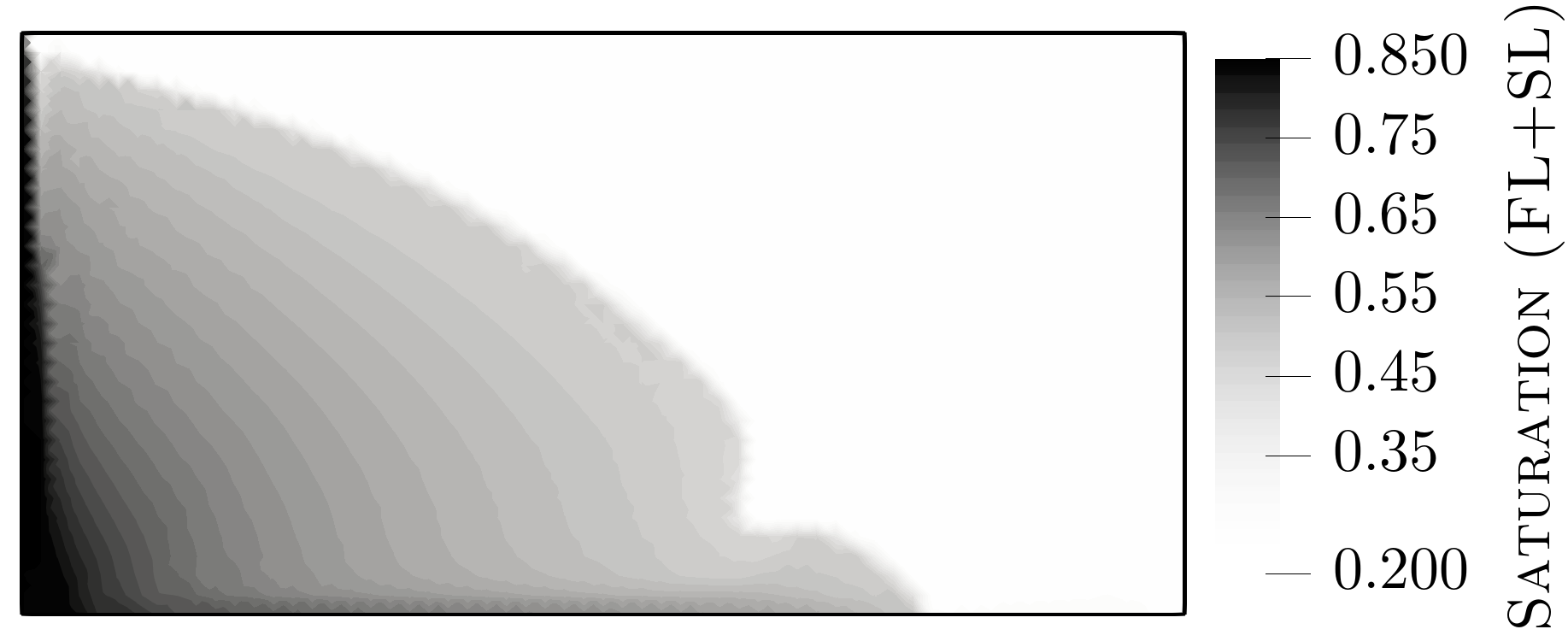}}
    \subfigure[Gr $= 0.6$ \label{fig:pdfg2}]{
        \includegraphics[clip,scale=0.095,trim=0 0cm 0cm 0]{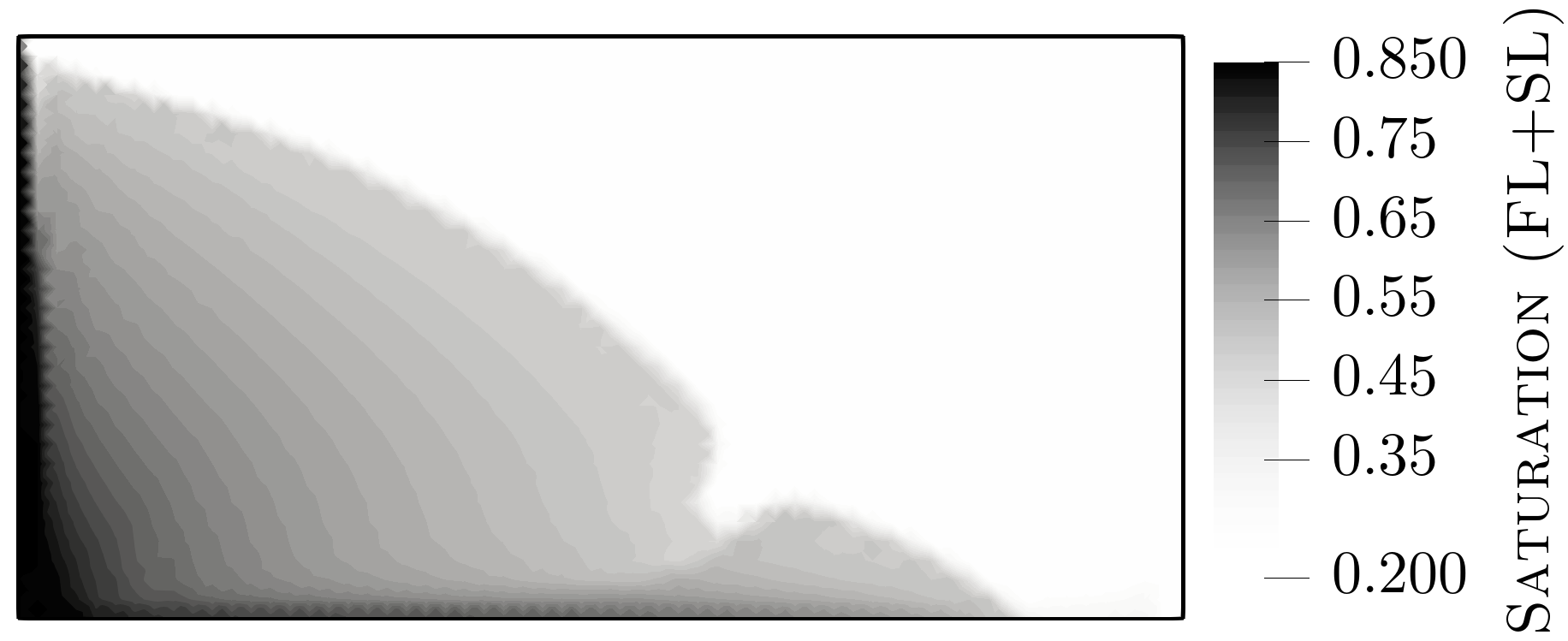}} 
    \subfigure[Gr $= 1.6$ \label{fig:pdfg3}]{
        \includegraphics[clip,scale=0.095,trim=0 0cm 0cm 0]{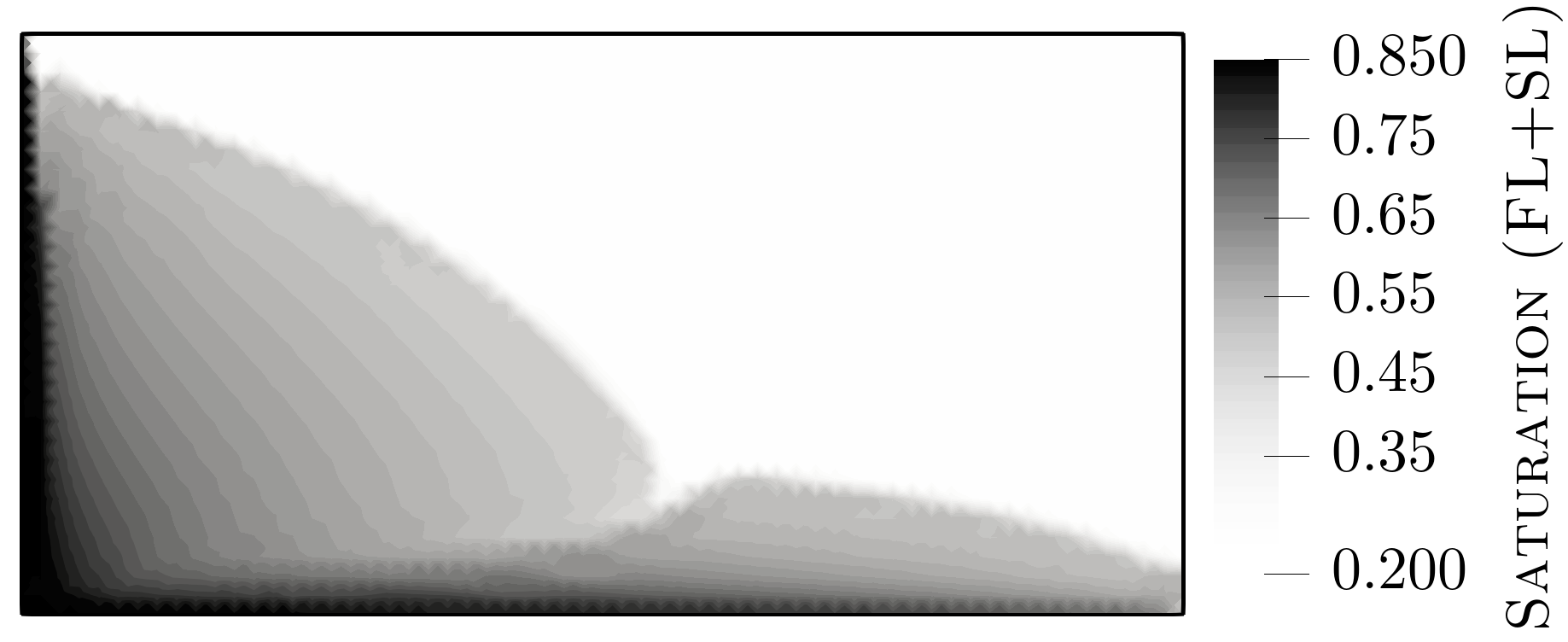}} 
        \caption{\textsf{Two-dimensional pressure-driven flow problem with gravity field:}
            This figure shows the wetting phase saturation solutions at $t=600 \si{\second}$ for different 
            gravity numbers. Both flux and slope limiters are used. 
            As the gravity number increases, more wetting phase accumulates at the bottom of the domain.
            For all three cases, no violation of maximum principle is observed.
        \label{Fig:2Dpatch_gravity_sat}}
\end{figure}

\begin{figure}
    \subfigure[Gr $= 0.3 $\ \label{fig:pdfg1p}]{
        \includegraphics[clip,scale=0.085,trim=0 0cm 0cm
        0]{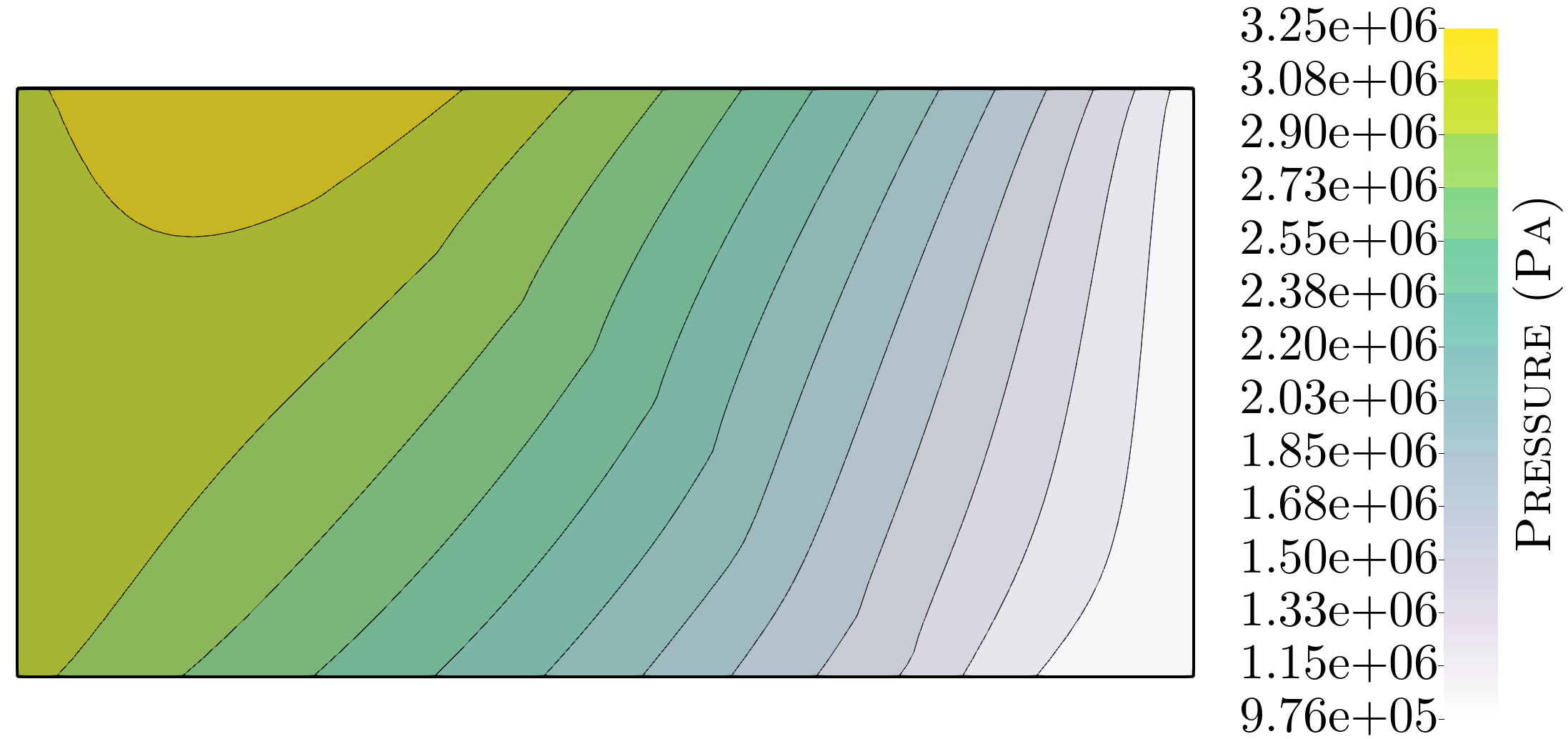}}
    \subfigure[Gr $= 0.6 $ \label{fig:pdfg2p}]{
        \includegraphics[clip,scale=0.085,trim=0 0cm 0cm 0]{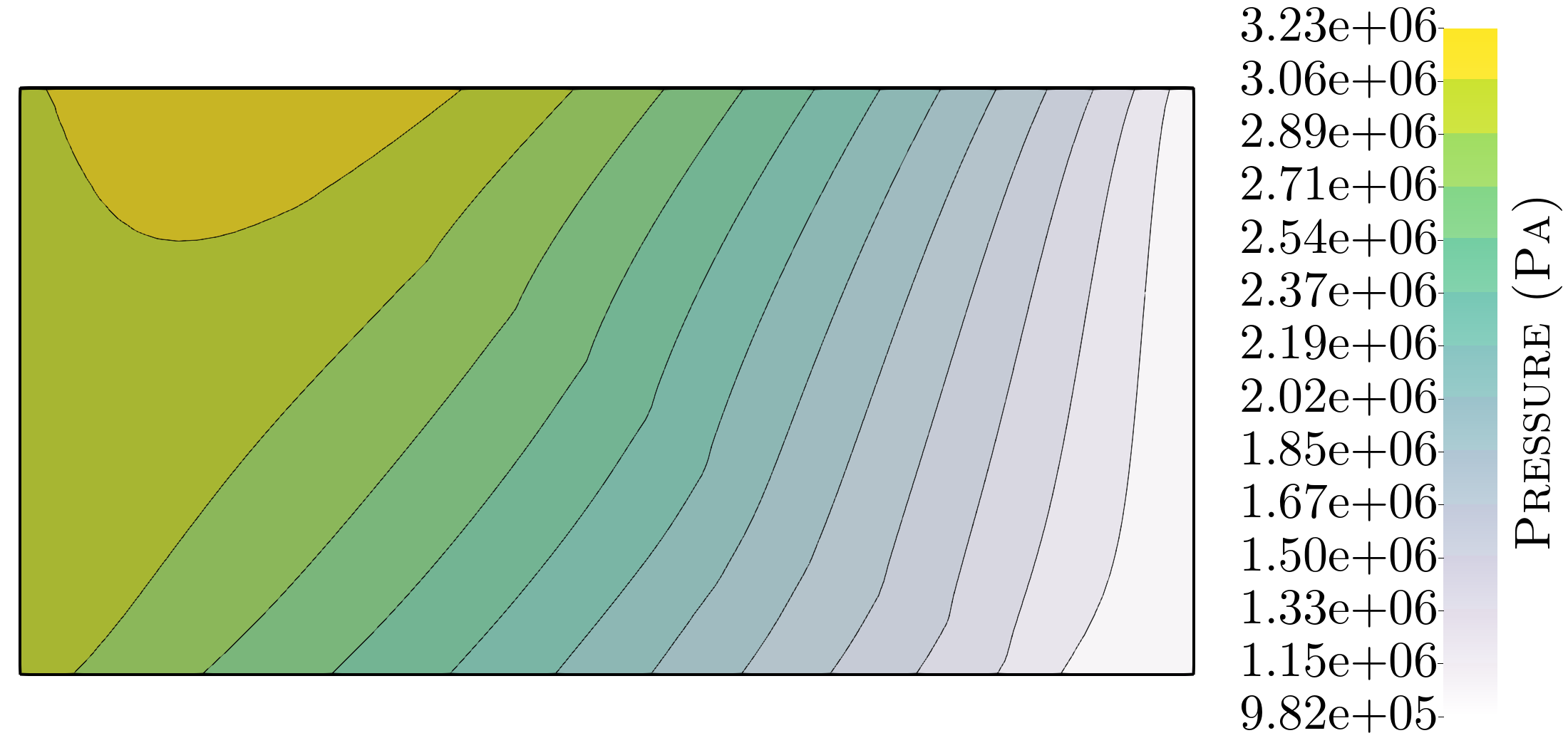}} 
    \subfigure[Gr $= 1.6 $ \label{fig:pdfg3p}]{
        \includegraphics[clip,scale=0.085,trim=0 0cm 0cm 0]{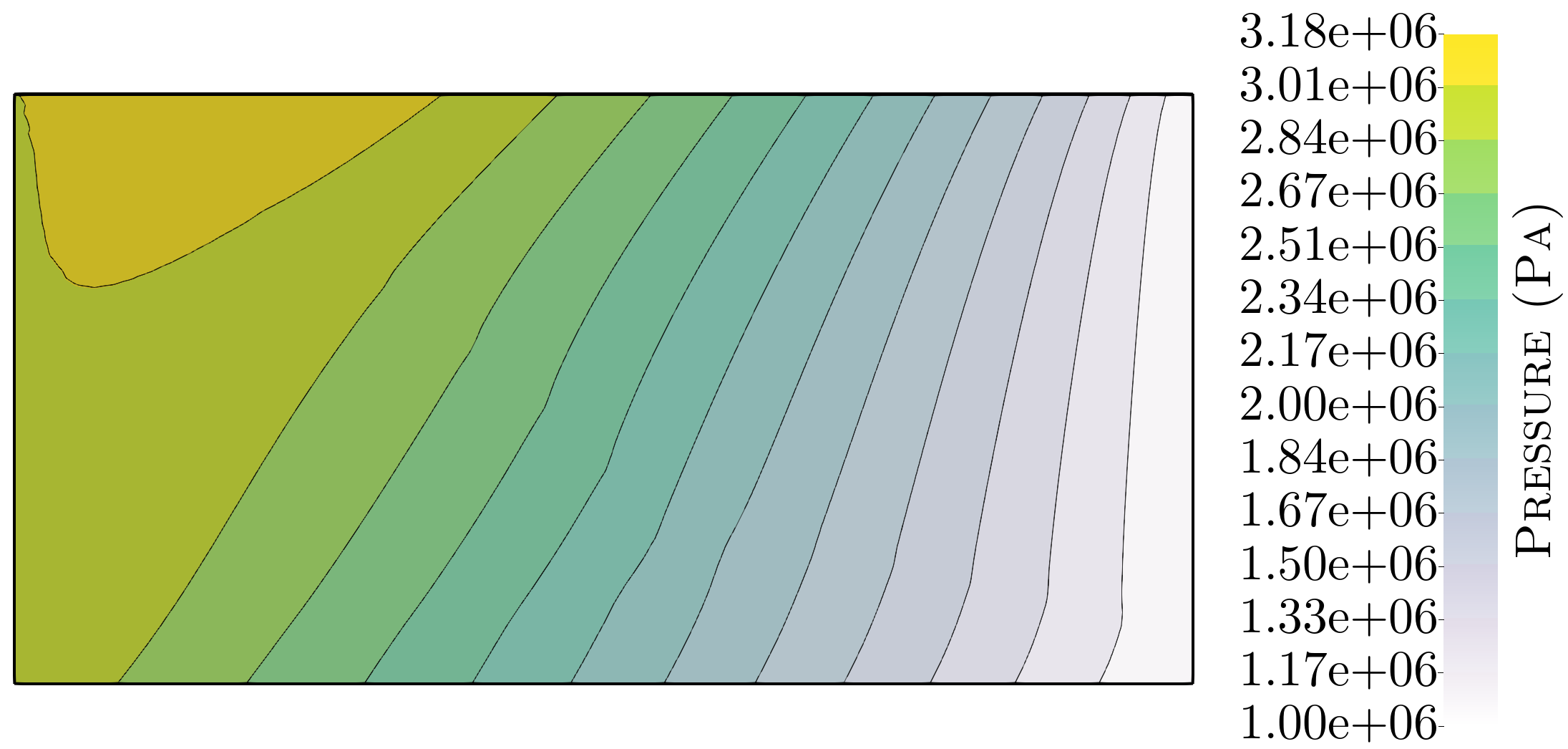}} 
        \caption{\textsf{Two-dimensional pressure-driven flow problem with gravity field:}
            This figure shows the wetting phase pressure solutions at $t=600 \si{\second}$ for different 
            gravity numbers.  
        \label{Fig:2Dpatch_gravity_pres}}
\end{figure}

\begin{figure}
    \subfigure[Gr$=0.3 $ \label{fig:pdfg1v}]{
        \includegraphics[clip,scale=0.095,trim=0 0cm 0cm
        0]{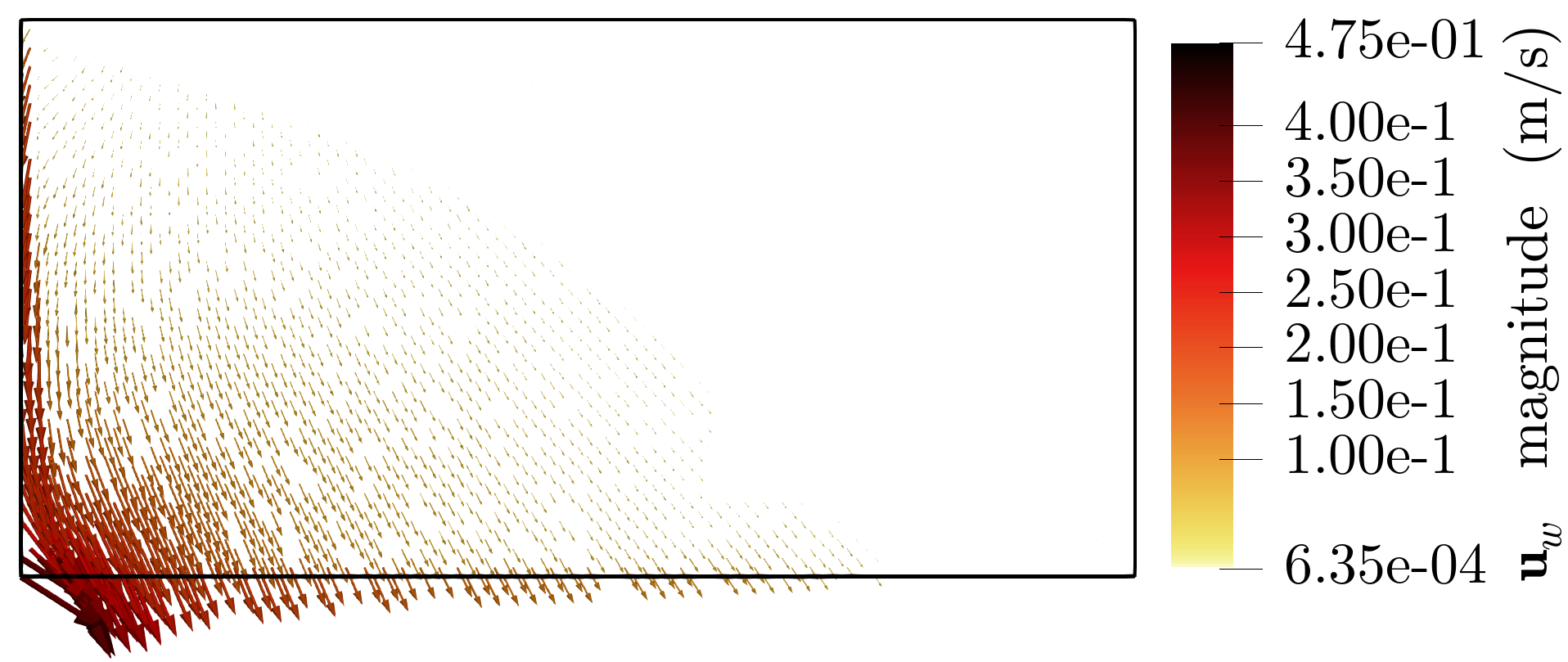}}
    \subfigure[Gr$=0.6 $ \label{fig:pdfg2v}]{
        \includegraphics[clip,scale=0.095,trim=0 0cm 0cm 0]{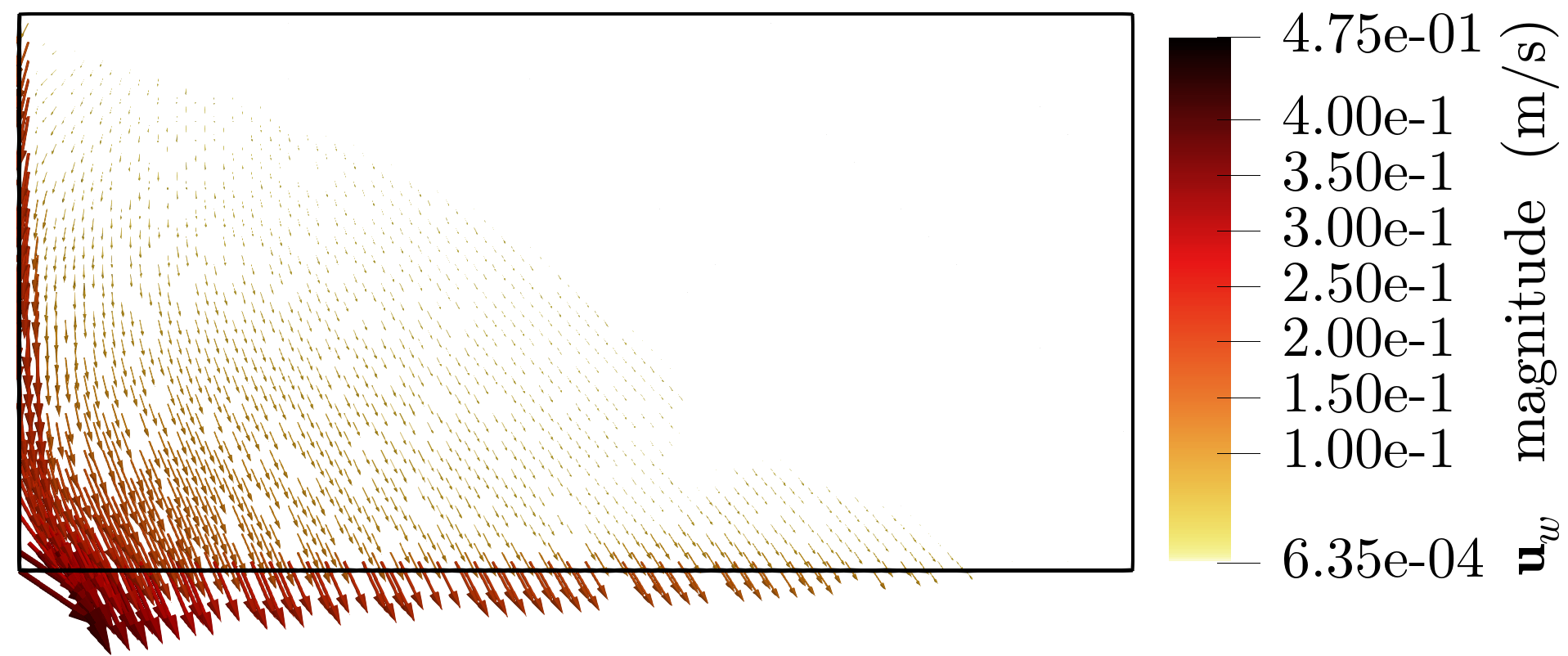}} 
    \subfigure[Gr$=1.6 $ \label{fig:pdfg3v}]{
        \includegraphics[clip,scale=0.095,trim=0 0cm 0cm 0]{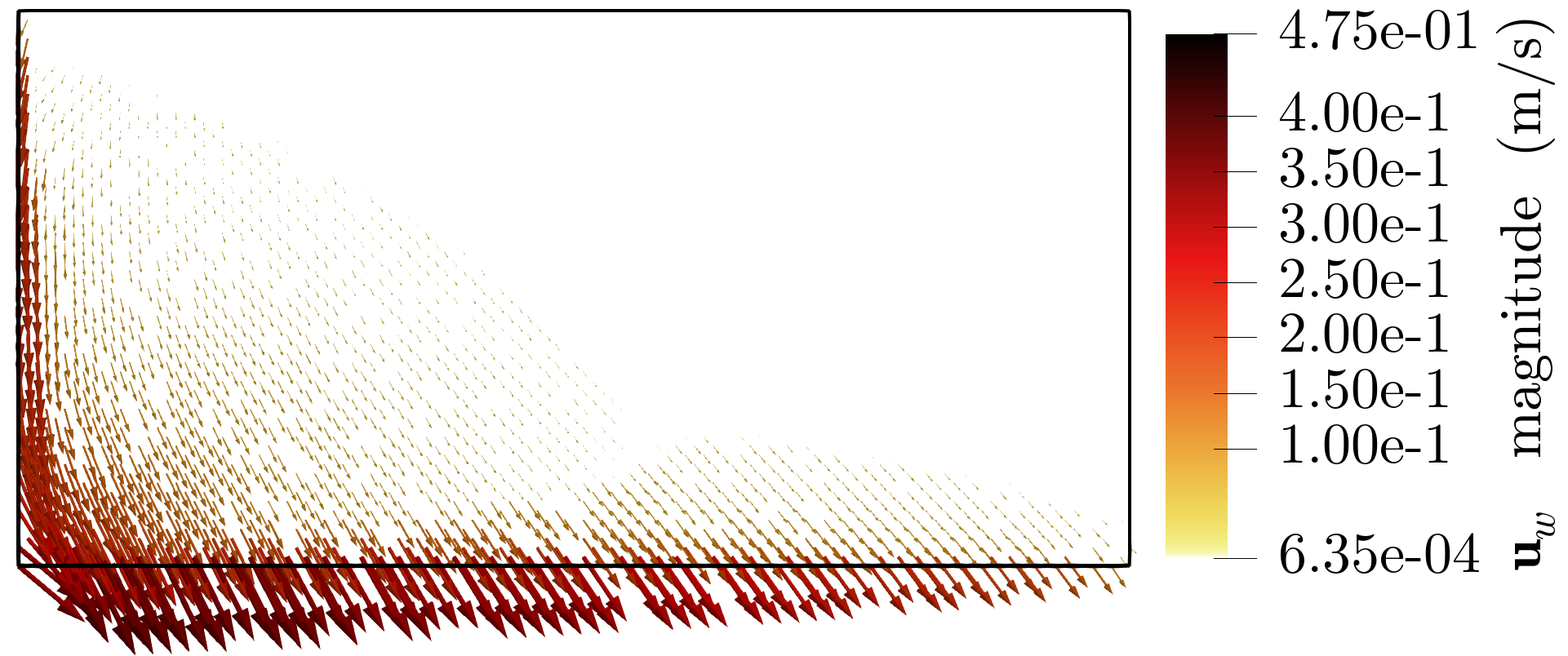}} 
        \caption{\textsf{Two-dimensional pressure-driven flow problem with gravity field:} 
            This figure depicts the magnitude and direction of the wetting phase velocity at $t=600$ 
            \si{\second} for different gravity numbers.
        \label{Fig:2Dpatch_gravity_vel}}
\end{figure}


\subsubsection{Quarter five-spot problem}
The domain is $\Omega=[0,1000]^2$ \si{\meter\squared} with permeability of $K=3\times10^{-11}$ everywhere.  Capillary pressure and relative permeabilities
are defined in equation \eqref{Eqn:CapillaryPres} and \eqref{Eqn:Q5_homogen_BrookCorey}, respectively, with 
entry pressure $P_d=1000$ \si{\pascal}, $\theta=2$ and $R=0.05$. 
To address wells, we fix the following parameters:
$L_w=80$ \si{\meter}, $d_w=80$ \si{\meter}, $\bar{q}=\underline{q}=9.33\times10^{-6}$. 
The wetting phase density is set to $\rho_w = 1000$~kg/m\textsuperscript{3} and the non-wetting phase density 
takes three different values $\rho_{\ell} = 925, 850, 600$~kg/m\textsuperscript{3}, which yields three values 
for the gravity number Gr$= 0.8, 1.6$ and $4.3$ respectively. The characteristic velocity in Gr
estimation is taken as $U\approx5.5\times10^{-5}$ \si{\meter\per\second} (or $4.8$ \si{\meter\per{day}}).
Other parameters are the same as in Section~\ref{sub:Q5_hetero}.
The simulation runs to $T=11$ days with $750$ time steps.
Wetting phase saturation contours, wetting phase pressure contours and wetting phase velocity fields are 
shown in Figure \ref{Fig:Q5_gravity_sat}, \ref{Fig:Q5_gravity_pres}, and \ref{Fig:Q5_gravity_vel} respectively.
We observe that as the gravity number increases, the inertial forces prevent the saturation to reach the 
production well.  As in the previous section, the discrete solution satisfies the maximum principle. 
The numerical examples in this section confirm that our proposed numerical method is accurate and robust when 
gravity dominates.
\begin{figure}
    \subfigure[Gr$=0.8 $ \label{fig:q5sgs1}]{
        \includegraphics[clip,scale=0.095,trim=0 0cm 0cm
        0]{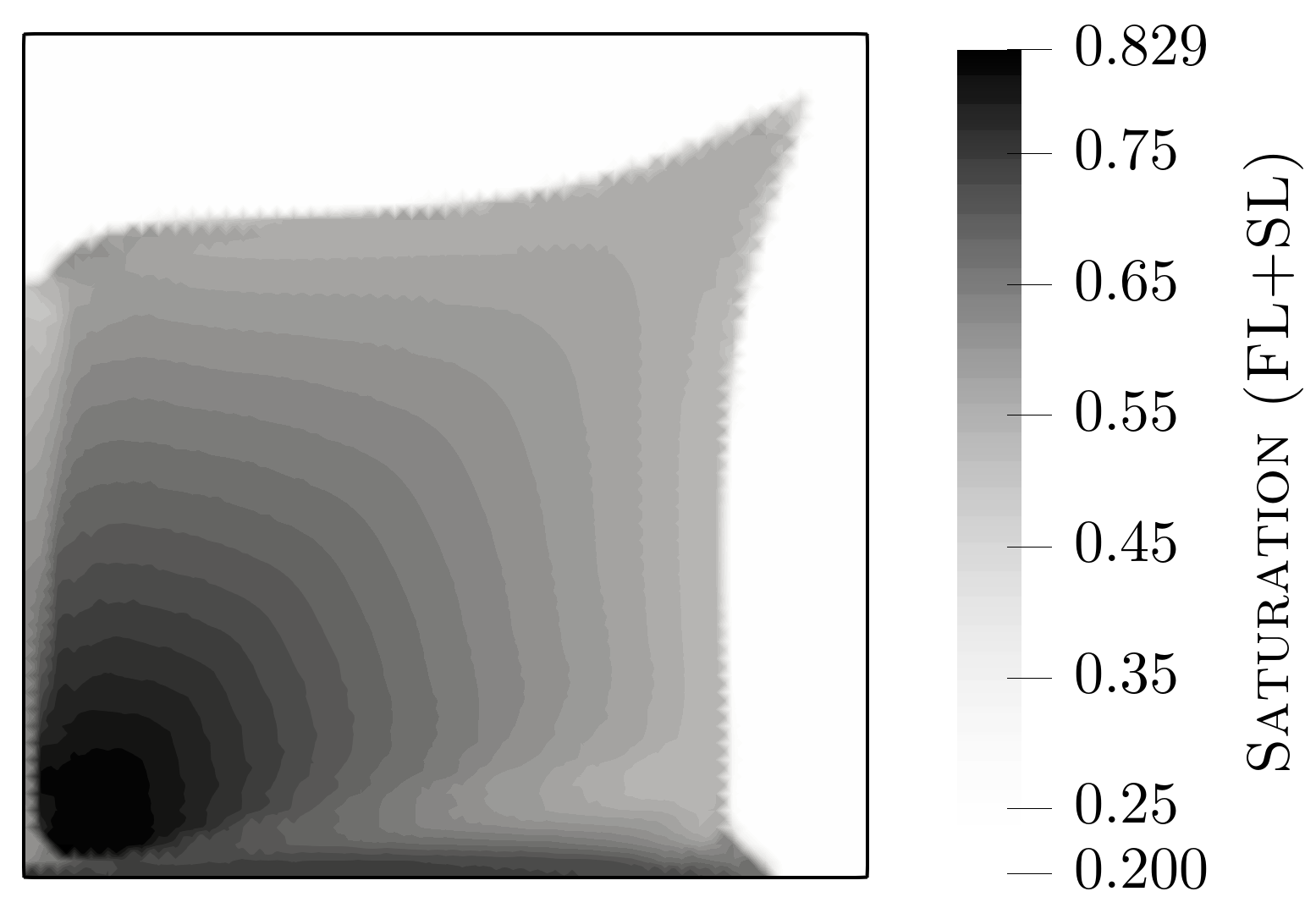}}
    \subfigure[Gr$ = 1.6$ \label{fig:q5sgs2}]{
        \includegraphics[clip,scale=0.095,trim=0 0cm 0cm 0]{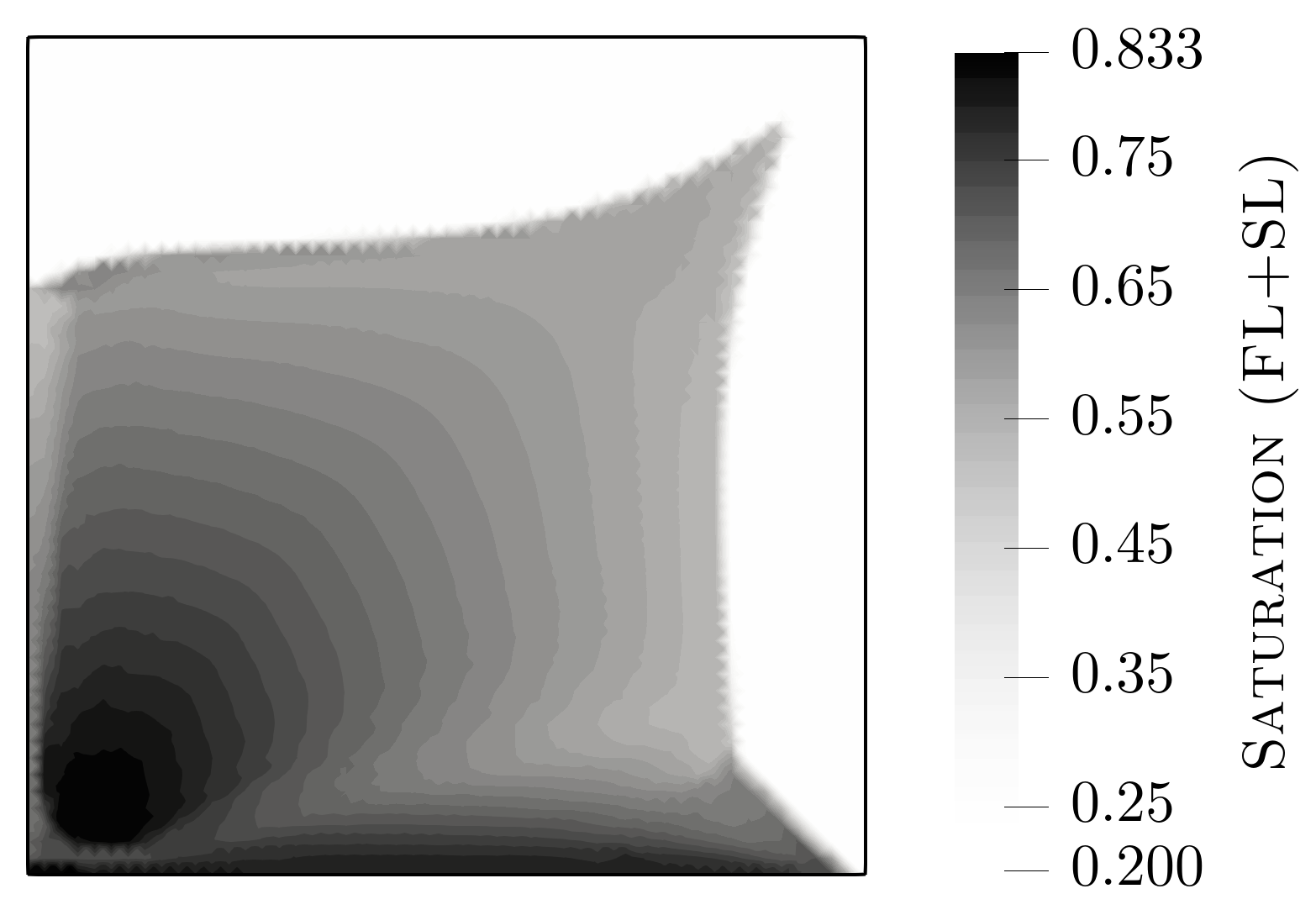}} 
    \subfigure[Gr$=4.3$ \label{fig:q5sgs3}]{
        \includegraphics[clip,scale=0.095,trim=0 0cm 0cm 0]{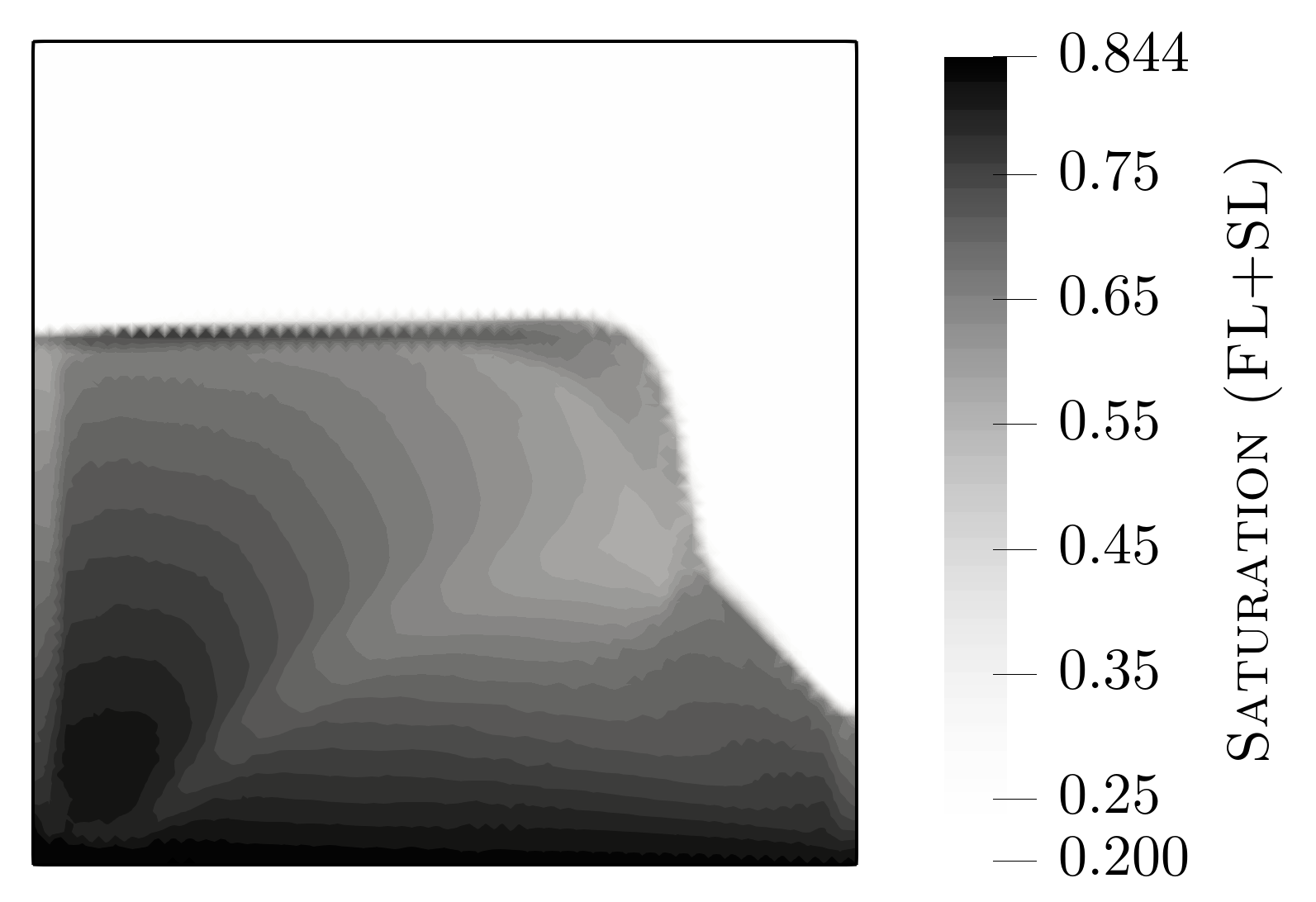}} 
        \caption{\textsf{Quarter five-spot problem with gravity field:} 
            This figure depicts saturation contours at $t=11$ days for different gravity numbers.
            DG+FL+SL scheme is applied that leads to satisfactory results with respect to maximum principle.
            By increasing the difference in phases density, gravitational force dominates the viscous force 
            (from left to right). This results in more wetting phase saturation to be deposited at the bottom 
            of domain and hence less non-wetting phase is recovered at the production well. 
        \label{Fig:Q5_gravity_sat}
}
\end{figure}

\begin{figure}
    \subfigure[G$r=0.8 $ \label{fig:q5sgp1}]{
        \includegraphics[clip,scale=0.095,trim=0 0cm 0cm
        0]{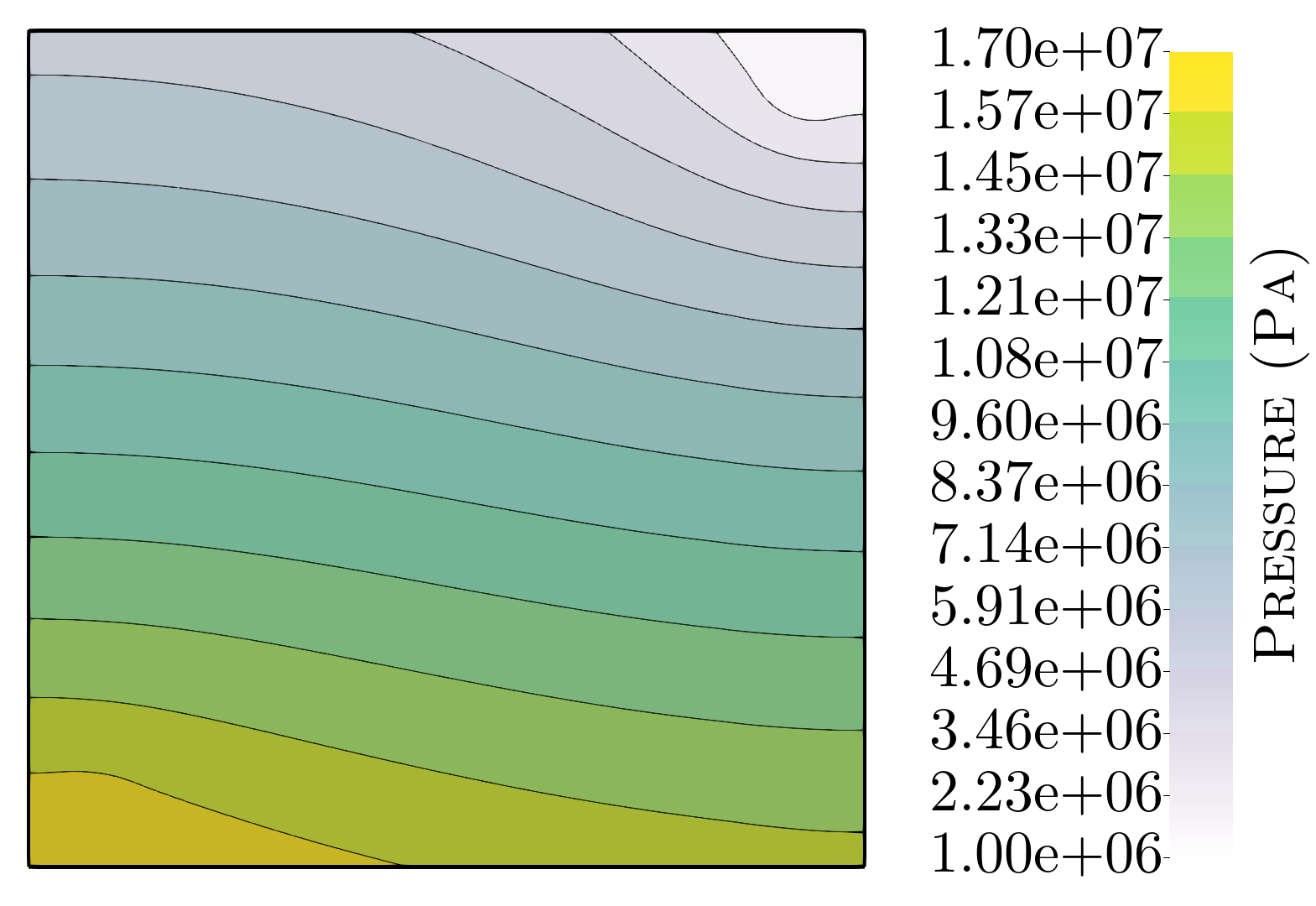}}
    \subfigure[Gr$=1.6 $ \label{fig:q5sgp2}]{
        \includegraphics[clip,scale=0.095,trim=0 0cm 0cm 0]{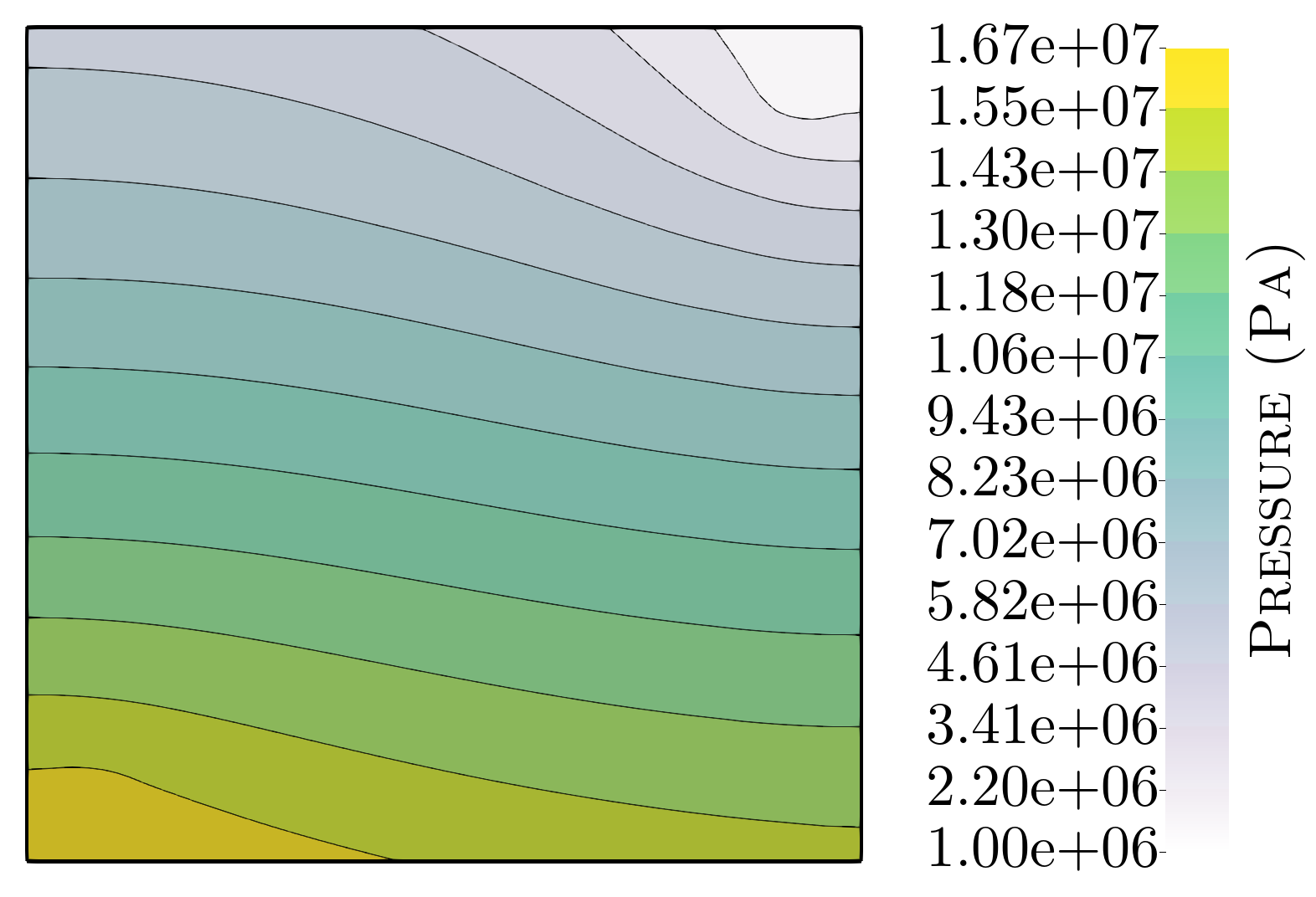}} 
    \subfigure[Gr $= 4.3 $ \label{fig:q5sgp3}]{
        \includegraphics[clip,scale=0.095,trim=0 0cm 0cm 0]{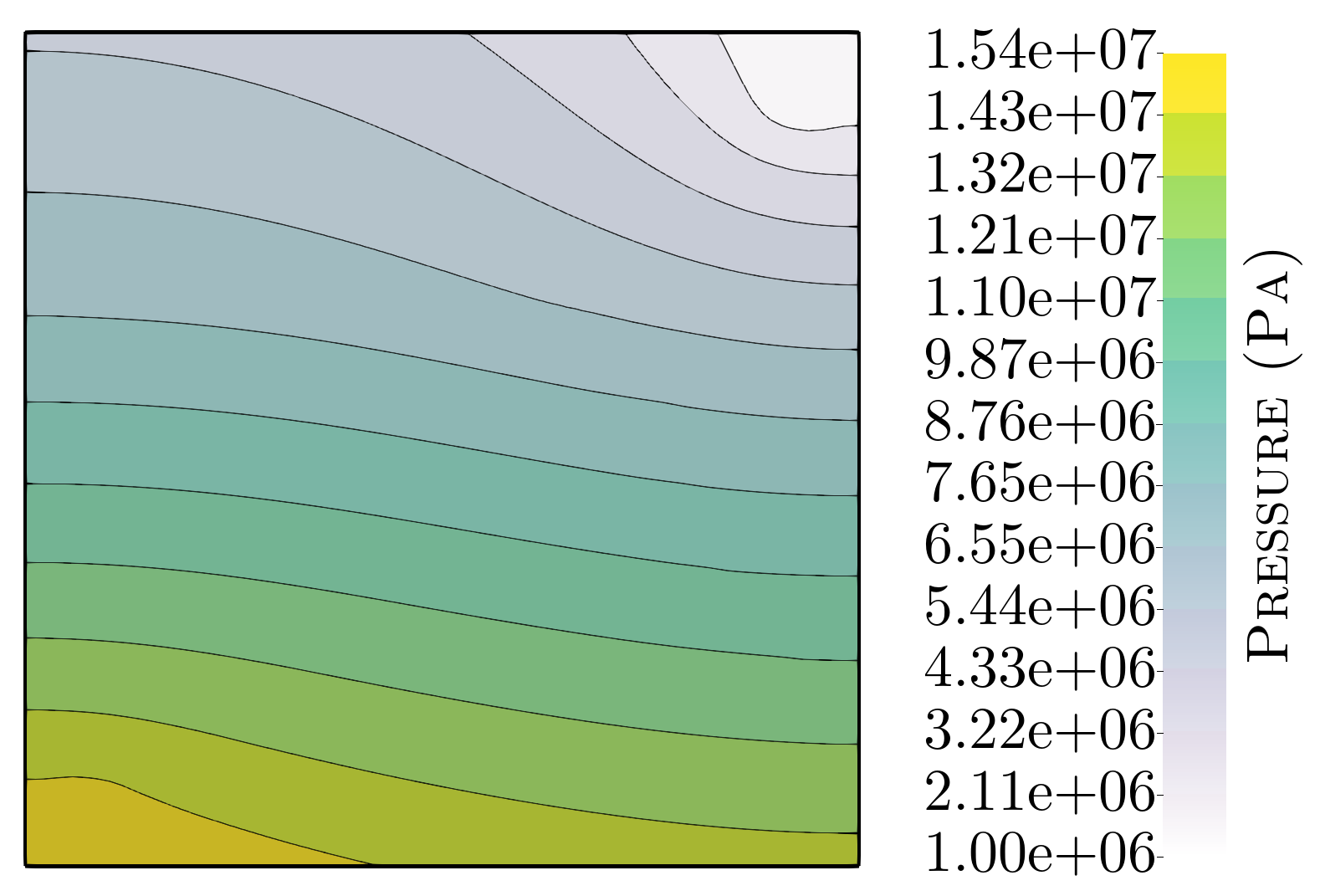}} 
        \caption{\textsf{Quarter five-spot problem with gravity field:} 
            This figure shows pressure contours at $t=11$ days for different gravity numbers.
            As gravity number increases (from left to right), pressure difference between the injection and 
            the production wells reduces. This is because as the gravitational force dominates, it hinders 
            the wetting phase flow from reaching the production well.
        \label{Fig:Q5_gravity_pres}}
\end{figure}

\begin{figure}
    \subfigure[Gr $= 0.8$ \label{Fig:q5sgv1}]{
        \includegraphics[clip,scale=0.09,trim=0 0cm 0cm
        0]{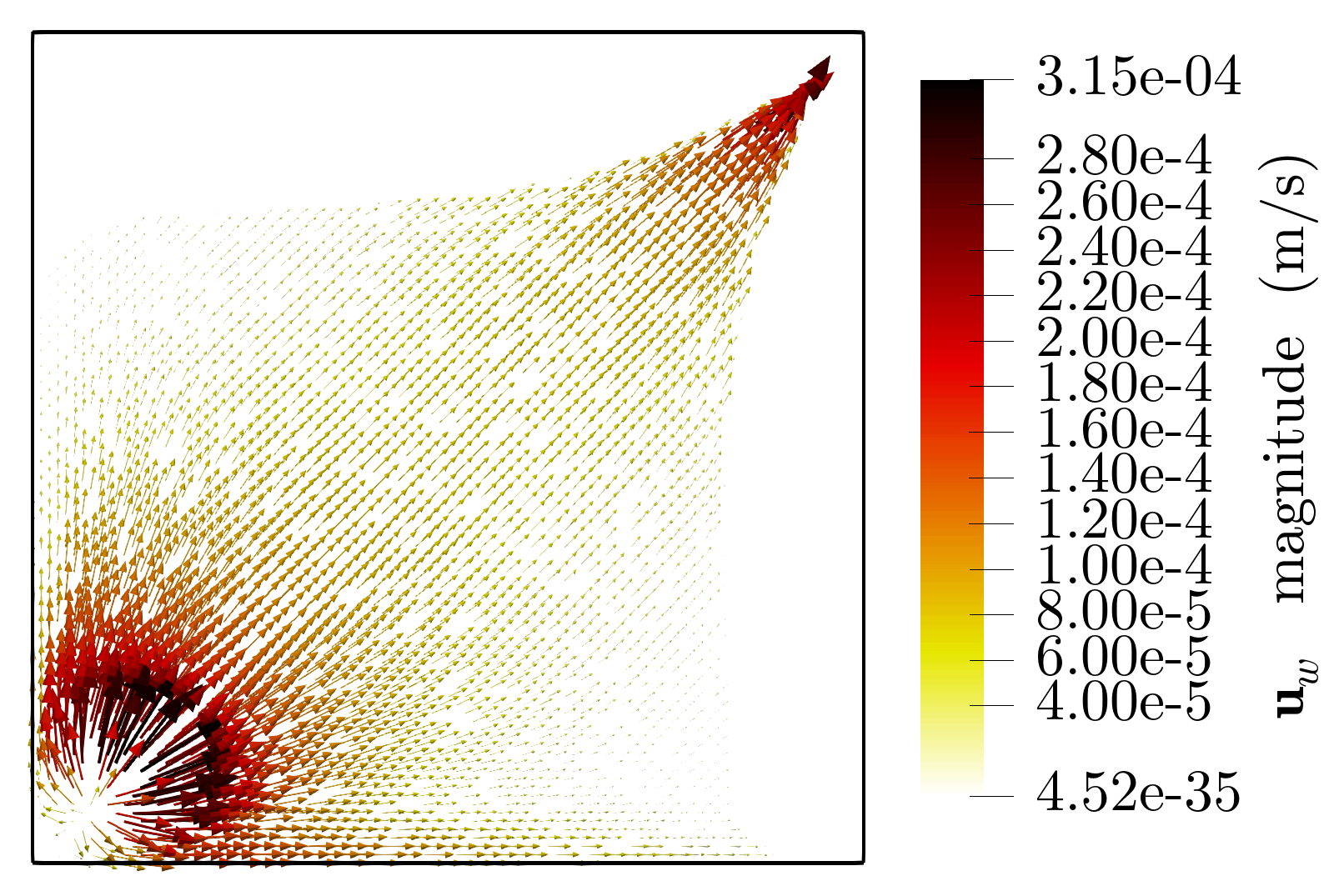}}
    \subfigure[Gr $= 1.6$ \label{fig:q5sgv2}]{
        \includegraphics[clip,scale=0.09,trim=0 0cm 0cm 0]{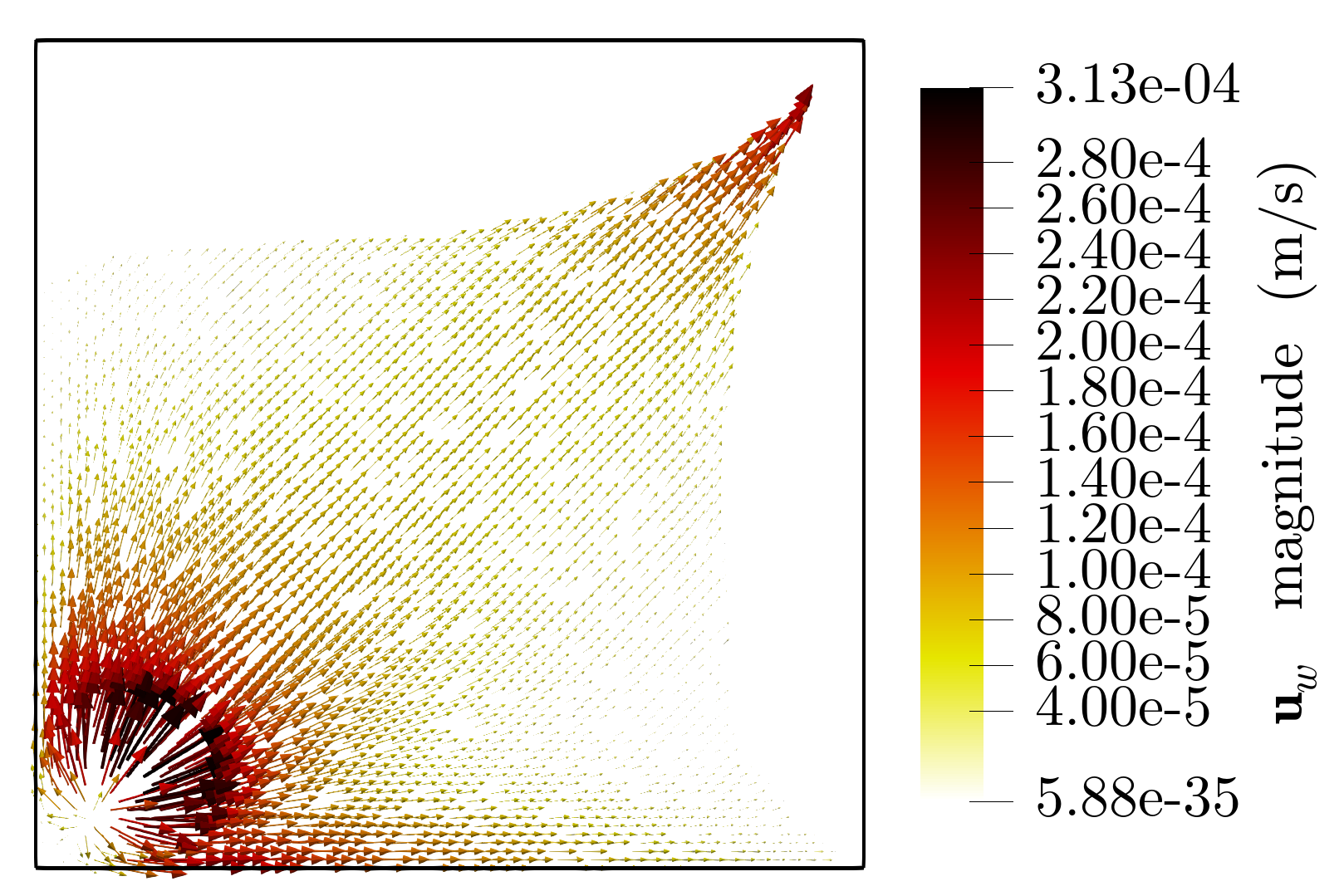}} 
    \subfigure[Gr=$4.3 $ \label{fig:q5sgv3}]{
        \includegraphics[clip,scale=0.09,trim=0 0cm 0cm 0]{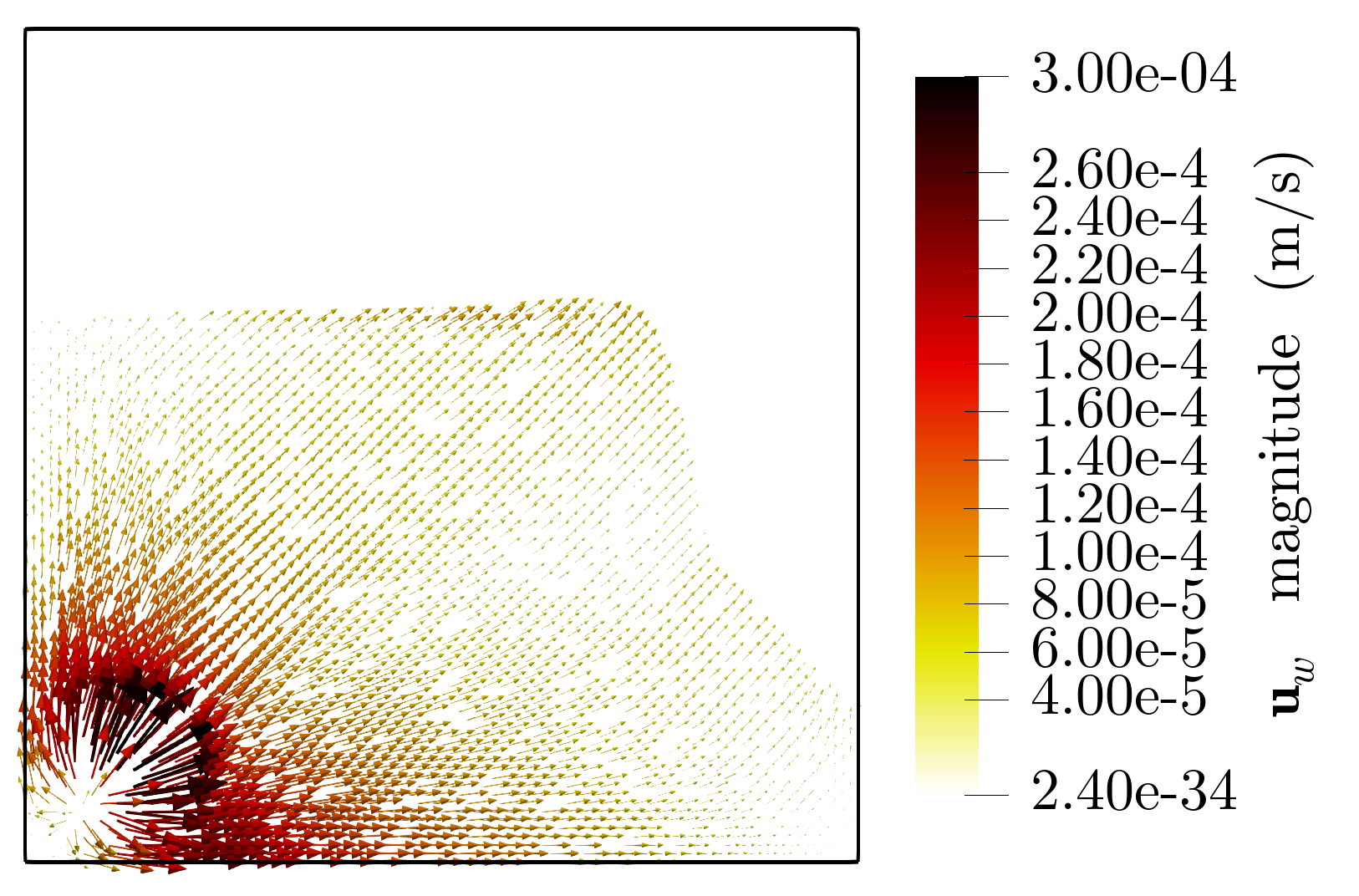}} 
        \caption{\textsf{Quarter five-spot problem with gravity field:} 
            This figure shows velocity field at $t=11$ days for different gravity numbers.
            Increase in gravity number pushes the wetting phase toward the bottom edge and 
            less recovery at the production well, which is also reflected in the decrease 
            in the magnitude of velocities.
        \label{Fig:Q5_gravity_vel}}
\end{figure}

\section{Conclusions}

A fully implicit discontinuous Galerkin method is formulated for solving the incompressible two-phase flow equations in porous media. Primary unknowns are the wetting phase pressure and saturation.  Nonlinear sytems are solved by Newton's method.
Post-processing flux are developed and combined with slope limiters to ensure a bound-preserving saturation at each time step. 
The numerical method is validated on several benchmark problems and it is applied to problems where permeability fields are highly varying or where gravitational forces are significant.  
Flooding of the medium is driven by either pressure boundary conditions or by injection and production wells.  The various numerical examples show that the scheme is robust and locally mass conservative. 
The approximation of the saturation is shown to satisfy the maximum principle both theoretically and computationally.

\section*{Acknowledgments}
The authors gratefully acknowledge Rustem Zaydullin and Romain De-Loubens for their valuable suggestions and discussions.        
This work is partially supported by the National Science Foundation (NSF-DMS 1913291).

%

\bibliographystyle{plainnat}
\bibliography{./Master_References/Master_References,./Master_References/Books}


\end{document}